\documentclass[11pt,reqno]{amsart}

\usepackage{amsmath,amssymb,amsthm}
\usepackage{bm}
\usepackage{natbib}
\usepackage{graphicx}
\usepackage{multirow}
\usepackage{here}
\usepackage{fullpage}
\usepackage{bm}
\usepackage{url}

\numberwithin{equation}{section}
\theoremstyle{plain}
\newtheorem{theorem}{Theorem}[section]
\newtheorem{corollary}{Corollary}[section]
\newtheorem{lemma}{Lemma}[section]
\newtheorem{proposition}{Proposition}[section]
\newtheorem{assumption}{Assumption}[section]
\theoremstyle{definition}
\newtheorem{definition}{Definition}[section]
\newtheorem{remark}{Remark}[section]

\makeatletter

\@addtoreset{equation}{section}
\makeatletter

\newcommand{\R}{\mathbb{R}}

\renewcommand{\tilde}{\widetilde}
\renewcommand{\hat}{\widehat}

\DeclareMathOperator{\Var}{Var}

\DeclareMathOperator{\supp}{supp}

\newcommand{\bS}{\bm{S}}

\begin{document}

\title[]{Nonparametric regression for locally stationary random fields under stochastic sampling design}
\thanks{D. Kurisu is partially supported by JSPS KAKENHI Grant Numbers 17H02513 and 20K13468. The author would like to thank Paul Doukhan, Kengo Kato, Yasumasa Matsuda, Taisuke Otsu and Robert Stelzer for their helpful comments and suggestions.} 

\author[D. Kurisu]{Daisuke Kurisu}

\date{First version: April 16, 2020. This version: \today}

\address[D. Kurisu]{Graduate School of International Social Sciences, Yokohama National University\\
79-4 Tokiwadai, Hodogaya-ku, Yokohama 240-8501, Japan.
}
\email{kurisu-daisuke-jr@ynu.ac.jp}

\begin{abstract}
In this study, we develop an asymptotic theory of nonparametric regression for locally stationary random fields (LSRFs) $\{\bm{X}_{\bm{s}, A_{n}}: \bm{s} \in R_{n} \}$ in $\mathbb{R}^{p}$ observed at irregularly spaced locations in $R_{n} =[0,A_{n}]^{d} \subset \mathbb{R}^{d}$. We first derive the uniform convergence rate of general kernel estimators, followed by the asymptotic normality of an estimator for the mean function of the model. Moreover, we consider additive models to avoid the curse of dimensionality arising from the dependence of the convergence rate of estimators on the number of covariates. Subsequently, we derive the uniform convergence rate and joint asymptotic normality of the estimators for additive functions. We also introduce \textit{approximately $m_{n}$-dependent} RFs to provide examples of LSRFs. We find that these RFs include a wide class of L\'evy-driven moving average RFs. 
\medskip

\medskip

\end{abstract}


\maketitle


\section{Introduction}
In this study, we consider the following model: 
\begin{align}\label{NR-LSRF}
Y_{\bm{s}_{j}, A_{n}} &= m\left({\bm{s}_{j} \over A_{n}}, \bm{X}_{\bm{s}_{j}, A_{n}}\right) + \epsilon_{\bm{s}_{j},A_{n}},\ \bm{s}_{j} \in R_{n},\ j=1,\hdots,n,
\end{align}
where $E[\epsilon_{\bm{s}, A_{n}}|\bm{X}_{\bm{s},A_{n}}] = 0$ and $R_{n}= [0,A_{n}]^{d} \subset \mathbb{R}^{d}$ is a sampling region with $A_{n} \to \infty$ as $n \to \infty$. Here, $Y_{\bm{s}_{j}, A_{n}}$ and $\bm{X}_{\bm{s}_{j}, A_{n}}$ are random variables of dimensions $1$ and $p$, respectively. We assume that $\{\bm{X}_{\bm{s}, A_{n}} : \bm{s} \in R_{n}\}$ is a locally stationary random field on $R_{n}  \subset \mathbb{R}^{d}$ ($d \geq 2$). Locally stationary processes, as proposed by \cite{Da97}, are nonstationary time series that allow parameters of the time series to be time-dependent. They can be approximated by a stationary time series locally in time, which enables asymptotic theories to be established for the estimation of time-dependent characteristics. In time series analysis, locally stationary models are mainly considered in a parametric framework with time-varying coefficients. For example, we refer to \cite{DaNevo99}, \cite{DaSu06}, \cite{FrSaSu08}, \cite{HaLi10}, \cite{KoLi12}, \cite{Zh14} and \cite{Tr17}. Nonparametric methods for stationary and nonstationary time series models have also been developed. We refer to \cite{Ma96}, \cite{FaYa03} and \cite{Ha08} for stationary time series and \cite{Kr09}, \cite{ZhWu09}, \cite{Kr11}, \cite{Vo12}, \cite{ZhWu15}, \cite{Tr19}, \cite{Ku21a}, and \cite{Ku21b} for recent contributions on locally stationary (functional) time series. Furthermore, we refer to \cite{DaRiWu19} for general theory in the literature of locally stationary processes. 

Recently, statistical inference of spatial regression models for geostatistical data has drawn considerable attention in certain economic and scientific fields, such as spatial econometrics, ecology, and seismology. Nonparametric methods for spatial data have also been the focus of attention. Recent contributions include those by \cite{HaLuTr04}, \cite{MaSt08}, \cite{HaLuYu09}, \cite{Ro11}, \cite{Je12}, \cite{LuTj14}, \cite{Li16}, \cite{MaSeOu18}, and \cite{Ku19}, who studied nonparametric inference and estimation of mean functions of spatial regression models. Compared with the abovementioned studies regarding time series analysis, studies pertaining to nonparametric methods for locally stationary random fields are scarce even for nonparametric density estimation as well as nonparametric regression despite empirical interest in modeling spatially varying dependence structures. Moreover, for the analysis of spatial data on $\mathbb{R}^{d}$, it is generally assumed that the underlying model is a parametric Gaussian process. We refer to \cite{FuSiLiRu15} for the empirical motivation and discussion regarding the modeling of nonstationary random fields and \cite{StFu10} for discussion regarding the importance of nonparametric, nonstationary, and non-Gaussian spatial models. To bridge the gap between theory and practice, we first extend the concept of a locally stationary time series by \cite{Da97} to random fields and consider a locally stationary random field $\{\bm{X}_{\bm{s}, A_{n}}: \bm{s} \in R_{n}\}$ ($R_{n} \subset \mathbb{R}^{d}$) as a nonstationary random field that can be approximated by a stationary random field locally at each spatial point $\bm{s} \in R_{n}$. A detailed definition of locally stationary random fields is provided in Section \ref{Section-Settings}. Next, we provide a complete asymptotic theory for the nonparametric estimators of the model (\ref{NR-LSRF}). The model (\ref{NR-LSRF}) can be regarded as a natural extension of (i) nonparametric regression to the locally stationary time series considered in \cite{Vo12} and \cite{ZhWu15}, and (ii) nonparametric spatial regression investigated in \cite{HaLuTr04}, \cite{Ro11}, \cite{Je12}, \cite{Li16}, and \cite{Ku19} to locally stationary random fields, and (iii) linear regression models with spatially varying coefficients to a nonlinear framework. Therefore, the model includes a wide range of key models for nonstationary random fields.  

The objectives of this study are to (i) derive the uniform convergence rate of kernel estimators for the density function of $\bm{X}_{\bm{s},A_{n}}$ and the mean function $m$ in the model (\ref{NR-LSRF}) over compact sets; (ii) derive the asymptotic normality of the estimators at a specified point; and (iii) provide examples of locally stationary random fields on $\mathbb{R}^{d}$ with a detailed discussion of their properties. To attain the first and second objectives, we first derive the uniform convergence rate of the important general kernel estimators; the result is crucial for demonstrating our main results. As general estimators include a wide range of kernel-based estimators such as the Nadaraya-Watson estimators, the general results are of independent interest. Although these results are general, the estimators are affected by dimensionality because their convergence rate depends on the number of covariates. Hence, we consider additive models and derive the uniform convergence rate and joint asymptotic normality of kernel estimators for additive functions based on the backfitting method developed by \cite{MaLiNi99} and \cite{Vo12}. Our results are extensions of the results for time series in \cite{Vo12} to random fields with irregularly spaced observations, which include irregularly spaced time series as a special case. To the best of our knowledge, this is the first study where a complete asymptotic theory of density estimation, nonparametric regression and additive models for locally stationary random fields on $\mathbb{R}^{d}$ are developed with a detailed discussion on concrete examples of locally stationary random fields.

Technically, the proofs of our results are significantly different from those for equally distant time series (mixing sequence) or spatial data observed on lattice points, as our focus is on spatial dependence and irregularly spaced sampling points. In many scientific fields, such as ecology, meteorology, seismology, and spatial econometrics, sampling points are naturally irregular. In fact, measurement stations cannot be placed on a regular grid owing to physical constraints. The stochastic sampling design assumed in this study allows the sampling sites to have a possibly non-uniform density across the sampling region, thereby enabling the number of sampling sites to increase at  different rates with respect to the volume of the region $O(A_{n}^{d})$ . In this study, we work with both \textit{the pure increasing domain} case ($\lim_{n \to \infty}nA_{n}^{-d} = \kappa \in (0, \infty)$) and \textit{the mixed increasing domain} case ($\lim_{n \to \infty}nA_{n}^{-d} = \infty$). As recent contributions pertaining to the design of stochastic spatial sampling, we refer to \cite{Sh10} for time series and \cite{La03b}, \cite{LaZh06}, and \cite{BaLaNo15} 
for random fields. In \cite{LuTj14}, they show asymptotic normality of a kernel density estimator of a strictly stationary random field on $\mathbb{R}^{2}$ under non-stochastic irregularly spaced observations based on the technique developed in \cite{Bo82}. Our approach is quite different from theirs and their sampling framework seems not compatible to develop asymptotic theory for locally stationary random fields. More precisely, in the proofs of our main results, we combine the big and small block techniques of \cite{Be26} and the coupling technique for mixing sequences of \cite{Yu94} to construct a sequence of independent blocks in a finite sample size. However, their applications are nontrivial for addressing irregularly spaced data and spatial dependence simultaneously. Discussion on the mixing condition is provided in Section \ref{Section-Settings}.

To attain the third objective of our study, we discuss examples of locally stationary random fields on $\mathbb{R}^{d}$ that satisfy our mixing conditions and the other regularity conditions mentioned in Section \ref{Ex-LSRF}. For this, we introduce the concept of \textit{approximately $m_{n}$-dependent} locally stationary random fields ($m_{n} \to \infty$ as $n \to \infty$) and we extend continuous autoregressive and moving average (CARMA)-type random fields developed in \cite{BrMa17} and \cite{MaYu20} to locally stationary CARMA-type random fields. CARMA random fields are characterized by solutions of (fractional) stochastic partial differential equations (cf. \cite{Be19}) and are known as a rich class of models for spatial data (cf. \cite{BrMa17}, \cite{MaYa18} and \cite{MaYu20}). However, their mixing properties have not been investigated.  We can show that a wide class of L\'evy-driven moving average random fields, which includes the locally stationary CARMA-type random fields, comprise (approximately $m_{n}$-dependent) locally stationary random fields.  One of the key features of CARMA random fields is that they can represent non-Gaussian random fields as well as Gaussian random fields if the driving L\'evy random measures are purely non-Gaussian. On the other hand, the statistical models in most of existing papers for spatial data on $\mathbb{R}^{2}$ are based on Gaussian processes and for non-Gaussian processes, it is often difficult to check mixing conditions as considered in \cite{LaZh06} and \cite{BaLaNo15}, for example. Moreover, in many empirical applications, the assumption of Gaussianity in spatial models is not necessarily adequate (see \cite{StFu10} for example). As a result, this study also contributes to the flexible modeling of nonparametric, nonstationary and possibly non-Gaussian random fields on $\mathbb{R}^{d}$. Therefore, this paper not only extends the scope of empirical analysis for spatial and spatio-temporal data (discussed in Section 6) but also addresses open questions on the dependence structure of statistical models built on CARMA random fields.

The rest of this paper is organized as follows. In Section 2, we define locally stationary random fields and describe our sampling scheme to consider irregularly spaced data and mixing conditions for random fields. In Section 3, we present the uniform convergence rate of general kernel estimators as well as the estimators of the mean function $m$ in the model (\ref{NR-LSRF}) over compact sets. The asymptotic normality of the estimator of the mean function is also provided. In Section 4, we consider the additive models and provide the uniform convergence rate and joint asymptotic normality of the kernel estimators for the additive functions. In Section 5, we discuss examples of locally stationary random fields by introducing the concept of $m_{n}$-dependent random fields; we also provide examples of locally stationary L\'evy-driven moving average random fields. We also discuss some extensions of our results and applications to spatio-temporal data in Section 6. All the proofs are deferred to Appendix. 

\subsection{Notations}
For $\bm{x} = (x_{1},\hdots, x_{d})' \in \mathbb{R}^{d}$, let $|\bm{x}| = |x_{1}| + \cdots + |x_{d}|$ and $\|\bm{x}\| = (x_{1}^{2} + \cdots + x_{d}^{2})^{1/2}$ denote the $\ell^{1}$ and $\ell^{2}$ norm on $\mathbb{R}^{d}$. For $\bm{x} = (x_{1},\hdots, x_{d})', \bm{y} = (y_{1},\hdots, y_{d}) \in \mathbb{R}^{d}$, the notation $\bm{x} \leq \bm{y}$ means that $x_{j} \leq y_{j}$ for all $j=1,\hdots,d$. For any set $A \subset \mathbb{R}^{d}$, let $|A|$ denote the Lebesgue measure of $A$ and $[\![A]\!]$ denote the number of elements in $A$. For any positive sequences $a_{n}$ and $b_{n}$, we write $a_{n} \lesssim b_{n}$ if a constant $C >0$ independent of $n$ exists such that $a_{n} \leq Cb_{n}$ for all $n$,  $a_{n} \sim b_{n}$ if $a_{n} \lesssim b_{n}$ and $b_{n} \lesssim a_{n}$, and $a_{n} \ll b_{n}$ if $a_{n}/b_{n} \to 0$ as $n \to \infty$. For $a \in \R$ and $b > 0$, we use the shorthand notation $[a \pm b] = [a-b,a+b]$. We use the notation $\stackrel{d}{\to}$ to denote convergence in distribution. For random variables $X$ and $Y$, we write $X \stackrel{d}{=} Y$ if they have the same distribution. $N(\mu, \Sigma)$ denotes a (multivariate) normal distribution with mean $\mu$ and a covariance matrix $\Sigma$. Let $P_{\bm{S}}$ denote the joint probability distribution of the sequence of independent and identically distributed  (i.i.d.) random vectors $\{\bm{S}_{0,j}\}_{j \geq 1}$, and let $P_{\cdot|\bm{S}}$ denote the conditional probability distribution given $\{\bm{S}_{0,j}\}_{j \geq 1}$. Let $E_{\cdot|\bm{S}}$ and $\Var_{\cdot|\bm{S}}$ denote the conditional expectation and variance given $\{\bm{S}_{0,j}\}_{j \geq 1}$, respectively. $\mathcal{B}(\mathbb{R}^{d})$ denotes the Borel $\sigma$-field on $\mathbb{R}^{d}$.

\section{Settings}\label{Section-Settings}

In this section, we introduce locally stationary random fields on $\mathbb{R}^{d}$ that extend the concept of the locally stationary process on $\mathbb{R}$. Furthermore, we discuss the sampling design of irregularly spaced locations and the dependence structures for random fields.

\subsection{Local stationarity}
Intuitively, a random field $\{\bm{X}_{\bm{s}_{j}, A_{n}}: \bm{s} \in R_{n} \}$ ($A_{n} \to \infty$ as $n \to \infty$) is locally stationary if it behaves approximately stationary in local space. To ensure that it is locally stationary around each rescaled space point $\bm{u}$, a process $\{\bm{X}_{\bm{s}, A_{n}}\}$ can be approximated by a stationary random field $\{\bm{X}_{\bm{u}}(\bm{s}): \bm{s} \in \mathbb{R}^{d}\}$ stochastically. See \cite{DaSu06} for an example. This concept can be defined as follows. 

\begin{definition}\label{LSRF}
The process $\{\bm{X}_{\bm{s}, A_{n}} = (X_{\bm{s},A_{j}}^{1},\hdots,X_{\bm{s},A_{n}}^{p})': \bm{s} \in R_{n}\}$ is locally stationary if for each rescaled space point $\bm{u} \in [0,1]^{d}$, there exists an associated random field $\{\bm{X}_{\bm{u}}(\bm{s}) = (X_{\bm{u}}^{1}(\bm{s}),\hdots,X_{\bm{u}}^{p}(\bm{s}))': \bm{s} \in \mathbb{R}^{d}\}$ with the following properties: 
\begin{itemize}
\item[(i)] $\{\bm{X}_{\bm{u}}(\bm{s}): \bm{s} \in \mathbb{R}^{d}\}$ is strictly stationary with density $f_{\bm{X}_{\bm{u}}(\bm{s})}$.
\item[(ii)] It holds that 
\begin{align}\label{LSRF-ineq}
\|\bm{X}_{\bm{s},A_{n}} - \bm{X}_{\bm{u}}(\bm{s})\|_{1} \leq \left(\left\|{\bm{s} \over A_{n}} - \bm{u}\right\|_{2} + {1 \over A_{n}^{d}}\right)U_{\bm{s},A_{n}}(\bm{u})\ a.s.,
\end{align}
where $\{U_{\bm{s},A_{n}}(\bm{u})\}$ is a process of positive variables satisfying $E[(U_{\bm{s},A_{n}}(\bm{u}))^{\rho}]<C$ for some $\rho>0$, $C<\infty$ that is independent of $\bm{u}, \bm{s}$, and $A_{n}$, and where $\|\cdot\|_{1}$ and $\|\cdot\|_{2}$ denote arbitrary norms on $\mathbb{R}^{p}$ and $\mathbb{R}^{d}$, respectively. 
\end{itemize}
\end{definition}

Definition \ref{LSRF} is a natural extension of the concept of` local stationarity for time series introduced in \cite{Da97}. We discuss examples of locally stationary random fields in Section \ref{Ex-LSRF}. In particular, we demonstrate that a wide class of random fields, which includes locally stationary versions of L\'evy-driven moving average random fields, satisfy the Condition (\ref{LSRF-ineq}). We also refer to \cite{Pe18} and \cite{MaYa18} for other definitions of locally stationary spatial  and spatio-temporal processes by using their spectral representations, respectively.  

\subsection{Sampling design}
To account for irregularly spaced data, we consider the stochastic sampling design. First, we define the sampling region $R_{n}$. Let $\{A_{n}\}_{n \geq 1}$ be a sequence of positive numbers such that $A_{n} \to \infty$ as $n \to \infty$.  We consider the following set as the sampling region. 
\begin{align}\label{sampling-region}
R_{n} = [0, A_{n}]^{d}. 
\end{align}
Next, we introduce our (stochastic) sampling designs. Let $f_{\bm{S}}(\bm{s}_{0})$ be a continuous, everywhere positive probability density function on $R_{0} = [0,1]^{d}$, and let $\{\bm{S}_{0,j}\}_{j \geq 1}$ be a sequence of i.i.d. random vectors with probability density $f_{\bm{S}}(\bm{s}_{0})$ such that $\{\bm{S}_{0,j}\}_{j \geq 1}$ and $\{\bm{X}_{\bm{s},A}:\bm{s} \in R_{n} \}$ are defined on a common probability space $(\Omega, \mathcal{F}, P)$ and are independent. We assume that the sampling sites $\bm{s}_{1},\hdots, \bm{s}_{n}$ are obtained from realizations $\bm{s}_{0,1},\hdots, \bm{s}_{0,n}$ of random vectors $\bm{S}_{0,1},\hdots, \bm{S}_{0,n}$ by the following relation:
\[
\bm{s}_{j} = A_{n}\bm{s}_{0,j},\ j=1,\hdots, n. 
\]  
Herein, we assume that $nA_{n}^{-d} \to \infty$ as $n \to \infty$. We also assume the following conditions on the sampling scheme.

\begin{assumption}\label{Ass-S}
\begin{enumerate}
\item[(S1)] For any $\bm{\alpha} \in \mathbb{N}^{d}$ with $|\bm{\alpha}| = 1,2$, $\partial^{\bm{\alpha}}f_{\bm{S}}(\bm{s})$ exists and is continuous on $(0,1)^{d}$. 
\item[(S2)] $C_{0} \leq nA_{n}^{-d} \leq C_{1}n^{\eta_{1}}$ for some $0<C_{0} < C_{1}<\infty$ and small $\eta_{1} \in [0,1)$.
\item[(S3)] Let $A_{1,n}$ and $A_{2,n}$ be two positive numbers such that as $n \to \infty$, $A_{1,n}, A_{2,n} \to \infty$ and ${A_{1,n} \over A_{n}} + {A_{2,n} \over A_{1,n}} \leq C_{0}^{-1}n^{-\eta} \to 0$ for some $C_{0}>0$ and $\eta>0$.
\end{enumerate}
\end{assumption}

Condition (S2) implies that our sampling design allows both the pure increasing domain case ($\lim_{n \to \infty}nA_{n}^{-d} = \kappa \in (0,\infty)$) and the mixed increasing domain case ($\lim_{n \to \infty}nA_{n}^{-d} = \infty$). This implies that our study addresses the infill sampling criteria in the stochastic design case (cf. \cite{Cr93} and \cite{La03b}), which is of interest in geostatistical and environmental monitoring applications (cf. \cite{La03b} and \cite{LaKaCrHs99}). Moreover, in stochastic sampling design, the sampling density can be nonuniform. There is another approach for irregularly spaced sampling sites based on a homogeneous Poisson point process (cf. Chapter 8 in \cite{Cr93}). For a sample size $n$, the approach allows the sampling sites to have only a uniform distribution over the sampling region. Therefore, the stochastic sampling design in this study is flexible compared with that based on homogeneous Poisson point processes and is therefore useful for practical applications. It may be worthwhile to consider sampling designs based on point processes that allow nonuniform sampling sites such as Cox or Hawkes point processes; however, we did not consider them because we believe that our sampling design is sufficiently practical for numerous applications. Condition (S3) is considered to decompose the sampling region $R_{n}$ into \textit{big} and \textit{small blocks}. See also the proof of Proposition \ref{general-unif-rate} in Appendix A for details.

\begin{remark}
In practice, $A_{n}$ can be determined using the diameter of the sampling region. See \cite{HaPa94} and \cite{MaYa09} for examples. We can relax the assumption (\ref{sampling-region}) on $R_{n}$ to a more general situation, i.e., 
\[
R_{n} = \prod_{j=1}^{d}[0, A_{j,n}],
\]
where $A_{j,n}$ are sequences of positive constants with $A_{j,n} \to \infty$ as $n \to \infty$. To avoid complicated results, we assumed (\ref{sampling-region}). See also Section \ref{extension discuss} for further discussion on the general sampling region and extensions of our results to the general cases. 
\end{remark}

\subsection{Mixing condition}

Now we define $\beta$-mixing coefficients for a random field $\tilde{\bm{X}}$. Let $\sigma_{\tilde{\bm{X}}}(T) = \sigma(\{\tilde{\bm{X}}(\bm{s}): \bm{s} \in T\})$ be the $\sigma$-field generated by variables $\{\tilde{\bm{X}(}\bm{s}): \bm{s} \in T\}$, $T \subset \mathbb{R}^{d}$. For subsets $T_{1}$ and $T_{2}$ of $\mathbb{R}^{d}$, let 
\[
\bar{\beta}(T_{1}, T_{2}) = \sup {1 \over 2}\sum_{j=1}^{J}\sum_{k=1}^{K}|P(A_{j}\cap B_{k}) - P(A_{j})P(B_{k})|, 
\]
where the supremum is taken over all pairs of (finite) partitions $\{A_{1},\hdots,A_{J}\}$ and $\{B_{1},\hdots, B_{K}\}$ of $\mathbb{R}^{d}$ such that $A_{j} \in \sigma_{\tilde{\bm{X}}}(T_{1})$ and $B_{k} \in \sigma_{\tilde{\bm{X}}}(T_{2})$. Furthermore, let $d(T_{1},T_{2}) = \inf\{|\bm{x} - \bm{y}|: \bm{x} \in T_{1}, \bm{y} \in T_{2}\}$, where $|\bm{x}|= \sum_{j=1}^{d}|x_{j}|$ for $\bm{x} \in \mathbb{R}^{d}$, and let $\mathcal{R}(b)$ be the collection of all finite disjoint unions of cubes in $\mathbb{R}^{d}$ with a total volume not exceeding $b$. Subsequently, the $\beta$-mixing coefficients for the random field $\tilde{\bm{X}}$ can be defined as
\begin{align}\label{beta-mix-def}
\beta(a;b) = \sup \{\bar{\beta}(T_{1},T_{2}): d(T_{1},T_{2}) \geq a, T_{1}, T_{2} \in \mathcal{R}(b)\}. 
\end{align}
We assume that a non-increasing function $\beta_{1}$ with $\lim_{a \to \infty}\beta_{1}(a) = 0$ and a non-decreasing function $g_{1}$ exist such that the $\beta$-mixing coefficient $\beta(a;b)$ satisfies the following inequality: 
\[
\beta(a;b) \leq \beta_{1}(a)g_{1}(b),\ a>0, b>0,
\]
where $g_{1}$ may be unbounded for $d \geq 2$. 

In Section \ref{Ex-LSRF}, we can show that our results also hold for a locally stationary random field $\{\bm{X}_{\bm{s}, A_{n}}: \bm{s} \in R_{n}\}$ which is approximated by a $\beta$-mixing random field, i.e., 
\begin{align}
\bm{X}_{\bm{s},A_{n}} &= \tilde{\bm{X}}_{\bm{s},A_{n}} + \tilde{\bm{\epsilon}}_{\bm{s},A_{n}}, \label{beta-decomp-approx}
\end{align}
where $\tilde{\bm{X}}_{\bm{s},A_{n}}$ is a $\beta$-mixing random field with a $\beta$-mixing coefficient such that $\beta(a;b) = 0$ for $a \geq A_{2,n}$ (i.e. $A_{2,n}$-dependent) and $\tilde{\bm{\epsilon}}_{\bm{s},A_{n}}$ is a ``residual'' random field which is asymptotically negligible. See also Definition \ref{approx-m-def} in Section \ref{Ex-LSRF} for the meaning of asymptotic negligibility.

We can also define $\alpha$-mixing coefficients $\alpha(a;b)$ for a random field $\bm{X}$ similar to the way we defined $\beta(a;b)$ and it is known that $\beta$-mixing implies $\alpha$-mixing in general. However, it is difficult to compute $\beta$- (and even for $\alpha$-)mixing coefficients of a given random field on $\mathbb{R}^{d}$, which is a key different point from time series data. Moreover, almost all existing papers assume technical conditions on mixing coefficients which are difficult to check except for some Gaussian random fields, $m$-dependent random fields on $\mathbb{R}^{2}$ ($m>0$: fix) and linear random fields on $\mathbb{Z}^{2}$ (cf. \cite{HaLuTr04}, \cite{LaZh06}, \cite{LuTj14}, \cite{BaLaNo15} and references therein). In Section \ref{Ex-LSRF}, we introduce a concept of approximately $m_{n}$-dependent random fields. We can check that a wide class of random fields on $\mathbb{R}^{d}$ which includes non-Gaussian random fields as well as Gaussian random fields satisfies (\ref{beta-decomp-approx}) and we can show that our results hold for such a class of random fields. See Section \ref{LSLevyMARF} for details.

\begin{remark}
It is important to restrict the size of index sets $T_{1}$ and $T_{2}$ in the definition of $\beta(a;b)$. To see this, we define the $\beta$-mixing coefficients of a random field $\bm{X}$ as a natural extension of the $\beta$-mixing coefficients for the time series as follows: Let $\mathcal{O}_{1}$ and $\mathcal{O}_{2}$ be half-planes with boundaries $L_{1}$ and $L_{2}$, respectively. For each real number $a>0$, we define the following quantity
\begin{align*}
\beta(a) = \sup\left\{\bar{\beta}(\mathcal{O}_{1},\mathcal{O}_{2}) : d(\mathcal{O}_{1},\mathcal{O}_{2}) \geq a\right\},
\end{align*}
where $\sup$ is taken over all pairs of parallel lines $L_{1}$ and $L_{2}$ such that $d(L_{1},L_{2}) \geq a$. Subsequently, we obtain the following result:
\begin{theorem}[Theorem 1 in \cite{Br89}]
Suppose $\{\bm{X}(s): \bm{s} \in \mathbb{R}^{2}\}$ is a strictly stationary mixing random field, and $a>0$ is a real number. Then $\beta(a) = 1\ \text{or}\ 0$.
\end{theorem}
This implies that if a random field $\bm{X}$ is $\beta$-mixing with respect to the $\beta$-mixing coefficient $\beta(\cdot)$ ($\lim_{a \to \infty}\beta(a)=0$), then the random field $\bm{X}$ is ``$m$''-dependent, i.e., $\beta(a)=0$ for some $a>m$, where $m$ is a fixed positive constant. For practical purposes, this is extremely restrictive. Therefore, we adopted the definition (\ref{beta-mix-def}) for $\beta$-mixing random fields. We refer to \cite{Br93}, \cite{Do94}, and \cite{DeDoLaLeLoPr07} for details regarding the mixing coefficients for random fields. 
\end{remark}

\subsection{Discussion on mixing conditions}\label{beta-mixing RF discuss}
It is often assumed that $\bm{X}$ is $\alpha$-mixing and blocking techniques are applied to construct asymptotically independent blocks of observations and to show asymptotic normality of estimators based on the convergence of characteristic functions (cf. Proposition 2.6 in \cite{FaYa03}). See also \cite{La03b}, \cite{LaZh06} and \cite{BaLaNo15} for example. \cite{Vo12} used an exponential inequality for an equidistant $\alpha$-mixing sequence (Theorem 2.1 in \cite{Li96}) to derive uniform convergence rates of kernel estimators for locally stationary time series. We can show pointwise convergence results under $\alpha$-mixing conditions without changing proofs. Precisely, the central limit theorems (Theorems 3.2 and 4.2) below hold under the same Assumptions by replacing conditions on $\beta$-mixing coefficient $\beta(a;b)$ with conditions on $\alpha$-mixing coefficient $\alpha(a;b)$ that is defined as follows: Let $\tilde{\bm{X}}$ be a random field on $\mathbb{R}^{d}$. For any two subsets $T_{1}$ and $T_{2}$ of $\mathbb{R}^{d}$, let $\tilde{\alpha}(T_{1},T_{2}) = \sup\{|P(A \cap B) - P(A)P(B)|: A \in \sigma_{\tilde{\bm{X}}}(T_{1}), B \in \sigma_{\tilde{\bm{X}}}(T_{2})\}$. Then the $\alpha$-mixing coefficient of the random field $\tilde{\bm{X}}$ is defined as 
\begin{align}\label{al-mix-coef}
\alpha(a;b) = \sup\{\tilde{\alpha}(T_{1},T_{2}): d(T_{1}, T_{2}) \geq a, T_{2}, T_{2} \in \mathcal{R}(b)\}.
\end{align}
As with the $\beta$-mixing coefficients, we also assume that a non-increasing function $\alpha_{1}$ with $\lim_{a \to \infty}\alpha_{1}(a) = 0$ and a non-decreasing function $g'_{1}$ exist such that the $\alpha$-mixing coefficient $\alpha(a;b)$ satisfies the following inequality: 
\[
\alpha(a;b) \leq \alpha_{1}(a)g'_{1}(b),\ a>0, b>0,
\]
where $g'_{1}$ may be unbounded for $d \geq 2$. We refer to \cite{LaZh06} and \cite{BaLaNo15} for the definition and discussion of $\alpha$-mixing coefficients for random fields on $\mathbb{R}^{d}$.

On the other hand, the proofs of main results in this paper are based on a general blocking technique designed for irregularly spaced sampling sites. Moreover, to derive the uniform convergence rates of the kernel estimators, we need to care about the effect of non-equidistant sampling sites when applying a maximal inequality and it requires additional work compared with the case that sampling sites are equidistant. Indeed, in place of  using results for (regularly spaced) stationary sequence, which cannot be applied to the analysis of irregularly spaced non-stationary data, we construct ``exactly'' independent blocks of observations and apply results for independent data to the independent blocks since there is no practical guidance for introducing an order to spatial points as opposed to time series. Precisely, we first reduce the dependent data to not asymptotically but exactly independent blocks in finite sample by using the key result in \cite{Yu94}(Corollary 2.7) which does not require any dependence structure and regularly spaced sampling sites. Then we apply the Bernstein's inequality for independent and possibly not identically distributed random variables to the independent blocks. It also should be noted that according to Remarks (ii) after the proof Lemma 4.1 in \cite{Yu94}, the result on construction of independent blocks for $\beta$-mixing sequence would not hold for $\alpha$-mixing sequences. Therefore, we work with $\beta$-mixing sequence for uniform estimation of regression functions.


As a recent contribution in econometrics, \cite{ChChKa19} extended high-dimensional central limit theorems for independent data developed by \cite{ChChKa13, ChChKa17} to possibly nonstationary $\beta$-mixing time series under a nonstochastic sampling design using the technique in \cite{Yu94}. See also Section \ref{more discuss mixing} for further discussion on the mixing conditions.

\section{Main results}\label{Section_Main}

In this section, we consider general kernel estimators and derive their uniform convergence rates. Based on the result, we derive the uniform convergence rate and asymptotic normality of an estimator for the mean function in the model (\ref{NR-LSRF}).

\subsection{Kernel estimation for regression functions}\label{section-kernel}

We consider the following kernel estimator for $m(\bm{u},\bm{x})$ in the model (\ref{NR-LSRF}): 
\begin{align}\label{m-func-est}
\hat{m}(\bm{u},\bm{x}) &= {\sum_{j=1}^{n}\bar{K}_{h}(\bm{u}-\bm{s}_{j}/A_{n})\prod_{\ell=1}^{p}K_{h}(x_{\ell} - X_{\bm{s}_{j},A_{n}}^{\ell})Y_{\bm{s}_{j},A_{n}} \over \sum_{j=1}^{n}\bar{K}_{h}(\bm{u}-\bm{s}_{j}/A_{n})\prod_{\ell=1}^{p}K_{h}(x_{\ell} - X_{\bm{s}_{j},A_{n}}^{\ell})}.
\end{align}
Here, $\bm{X}_{\bm{s},A_{n}} = (X_{\bm{s},A_{n}}^{1},\hdots, X_{\bm{s},A_{n}}^{p})'$, $\bm{x} = (x_{1},\hdots,x_{p})'$ and $\bm{u} = (u_{1},\hdots,u_{p})' \in \mathbb{R}^{p}$. $K$ denotes a one-dimensional kernel function, and we used the notations $K_{h}(v) = K(v/h)$, $\bar{K}_{h}(\bm{u}) = \bar{K}(\bm{u}/h)$ where $\bar{K}(\bm{u}) = \prod_{j=1}^{d}K(u_{j})$ and $\bm{u}/h = (u_{1}/h,\hdots, u_{d}/h)'$. 

Before we state the main results, we summarize the assumptions made for the model (\ref{NR-LSRF}) and kernel functions. These assumptions are standard;  similar assumptions are made in \cite{Vo12} and \cite{ZhWu15}.
\begin{assumption}\label{Ass-M}
\begin{enumerate}
\item[(M1)] The process $\{\bm{X}_{\bm{s},A_{n}}\}$ is locally stationary. Hence, for each space point $\bm{u} \in [0,1]^{d}$, a strictly stationary random field $\{\bm{X}_{\bm{u}}(\bm{s})\}$ exists such that 
\[
\|\bm{X}_{\bm{s},A_{n}} - \bm{X}_{\bm{u}}(\bm{s})\| \leq \left(\left\|{\bm{s} \over A_{n}} - \bm{u}\right\| + {1 \over A_{n}^{d}}\right)U_{\bm{s}, A_{n}}(\bm{u})\ \text{a.s.}
\]
with $E[(U_{\bm{s}, A_{n}}(\bm{u}))^{\rho}]\leq C$ for some $\rho>0$. 
\item[(M2)] The density $f(\bm{u}, \bm{x}) = f_{\bm{X}_{\bm{u}}(\bm{s})}(\bm{x})$ of the variable $\bm{X}_{\bm{u}}(\bm{s})$ is smooth in $\bm{u}$. In particular, $f(\bm{u}, \bm{x})$ is partially differentiable with respect to (w.r.t.) $\bm{u} \in (0,1)^{d}$ for each $\bm{x} \in \mathbb{R}^{p}$, and the derivatives $\partial_{u_{i}}f(\bm{u}, \bm{x}) = {\partial \over \partial u_{i}}f(\bm{u}, \bm{x})$, $1 \leq i \leq d$ are continuous. 
\item[(M3)] The $\beta$-mixing coefficients of the array $\{\bm{X}_{\bm{s}, A_{n}}, \epsilon_{\bm{s}, A_{n}}\}$ satisfy $\beta(a;b) \leq \beta_{1}(a)g_{1}(b)$ with $\beta_{1}(a) \to 0$ as $a \to \infty$.
\item[(M4)] $f(\bm{u}, \bm{x})$ is partially differentiable w.r.t. $\bm{x}$ for each $\bm{u} \in [0,1]^{d}$. The derivatives $\partial_{x_{i}}f(\bm{u}, \bm{x}) := {\partial \over \partial x_{i}}f(\bm{u}, \bm{x})$, $1\leq i \leq p$ are continuous. 
\item[(M5)] $m(\bm{u}, \bm{x})$ is twice continuously partially differentiable with first derivatives $\partial_{u_{i}}m(\bm{u},\bm{x}) = {\partial \over \partial u_{i}}m(\bm{u},\bm{x})$, $\partial_{x_{i}}m(\bm{u},\bm{x}) = {\partial \over \partial x_{i}}m(\bm{u},\bm{x})$ and second derivatives $\partial_{u_{i}u_{j}}^{2}m(\bm{u},\bm{x}) = {\partial^{2} \over \partial u_{i}\partial u_{j}}m(\bm{u},\bm{x})$, $\partial_{u
_{i}x_{j}}^{2}m(\bm{u},\bm{x}) = {\partial^{2} \over \partial u_{i} \partial x_{j}}m(\bm{u},\bm{x})$, $\partial_{x_{i}x_{j}}^{2}m(\bm{u},\bm{x}) = {\partial^{2} \over \partial x_{i}\partial x_{j}}m(\bm{u},\bm{x})$. 
\end{enumerate}
\end{assumption}

\begin{remark}\label{comment-(M1)}
We can verify that a wide class of random fields satisfies Condition (M1). Indeed, our results can be applied to a class of locally stationary L\'evy-driven moving average random fields that include non-stationary CARMA random fields and CARMA random fields are known as a rich class of models for spatial data (cf. \cite{BrMa17}, \cite{MaYa18} and \cite{MaYu20}). See also Section 5 for detailed discussion on the properties of locally stationary L\'evy-driven moving average random fields.
\end{remark}

\begin{assumption}\label{Ass-KB}
\begin{enumerate}
\item[(KB1)] The kernel $K$ is symmetric around zero, bounded, and has a compact support, i.e., $K(v) = 0$ for all $|v|>C_{1}$ for some $C_{1}<\infty$. Moreover, $K$ is Lipschitz continuous, i.e., $|K(v_{1}) - K(v_{2})| \leq C_{2}|v_{1} - v_{2}|$ for some $C_{2}<\infty$ and all $v_{1},v_{2} \in \mathbb{R}$. 
\item[(KB2)] The bandwidth $h$ is assumed to converge to zero at least at a polynomial rate, that is, there exists a small $\xi_{1}>0$ such that $h \leq Cn^{-\xi_{1}}$ for some constant $0<C<\infty$. 
\end{enumerate}
\end{assumption}

\subsection{Uniform convergence rates for general kernel estimators}

As a first step to study the asymptotic properties of estimators (\ref{m-func-est}), we analyze the following general kernel estimator: 
\begin{align}\label{genKernel-est}
\hat{\psi}(\bm{u},\bm{x}) &= {1 \over nh^{p+d}}\sum_{j=1}^{n}\bar{K}_{h}\left(\bm{u} - {\bm{s}_{j} \over A_{n}}\right)\prod_{\ell=1}^{p}K_{h}\left(x_{\ell} - X_{\bm{s}_{j},A_{n}}^{\ell}\right)W_{\bm{s}_{j},A_{n}},
\end{align}
where $\bar{K}(\bm{u}) = \prod_{j=1}^{d}K(u_{j})$ and $\{W_{\bm{s}_{j}, A_{n}}\}$ is an array of one-dimensional random variables. Many kernel estimators, such as Nadaraya--Watson estimators, can be represented by (\ref{genKernel-est}). In this study, we use the results with $W_{\bm{s}, A_{n}}=1$ and $W_{\bm{s}, A_{n}} = \epsilon_{\bm{s}, A_{n}}$. 

Next, we derived the uniform convergence rate of $\hat{\psi}(\bm{u}, \bm{x}) - E_{\cdot|\bm{S}}[\hat{\psi}(\bm{u},\bm{x})]$. We assumed the following for the components in (\ref{genKernel-est}). Similar assumptions are made in \cite{Ha08}, \cite{Kr09}, and \cite{Vo12}.

\begin{assumption}\label{Ass-U}
\begin{enumerate}
\item[(U1)] It holds that $E[|W_{\bm{s}, A_{n}}|^{\zeta}] \leq C$ for some $\zeta > 2$ and $C <\infty$. 
\item[(U2)] The $\beta$-mixing coefficients of the array $\{\bm{X}_{\bm{s}, A_{n}}, W_{\bm{s}, A_{n}}\}$ satisfy $\beta(a;b) \leq \beta_{1}(a)g_{1}(b)$ with $\beta_{1}(a) \to 0$ as $a \to \infty$. 
\item[(U3)] Let $f_{\bm{X}_{\bm{s}, A_{n}}}$ 
be the density of $\bm{X}_{\bm{s}, A_{n}}$. 
For any compact set $S_{c} \subset \mathbb{R}^{p}$, a constant $C = C(S_{c})$ exists such that 
\begin{align*}
\sup_{\bm{s},A_{n}}\sup_{\bm{x} \in S_{c}}f_{\bm{X}_{\bm{s}, A_{n}}}(\bm{x}) \leq C\ \text{and}\ \sup_{\bm{s}, A_{n}}\sup_{\bm{x} \in S_{c}}E[|W_{\bm{s}, A_{n}}|^{\zeta}|\bm{X}_{\bm{s}, A_{n}}=\bm{x}] \leq C.
\end{align*} 
Moreover, for all distinct $\bm{s}_{1}, \bm{s}_{2} \in R_{n}$, 
\begin{align*}
&\sup_{\bm{s}_{1},\bm{s}_{2},A_{n}}\sup_{\bm{x}_{1}, \bm{x}_{2} \in S_{c}}E[|W_{\bm{s}_{1}, A_{n}}||W_{\bm{s}_{2}, A_{n}}||\bm{X}_{\bm{s}_{1},A_{n}}= \bm{x}_{1}, \bm{X}_{\bm{s}_{2},A_{n}} = \bm{x}_{2}] \leq C.
\end{align*}
\end{enumerate}
\end{assumption}

Furthermore, we assume the following regularity conditions. 
\begin{assumption}\label{Ass-R}
Let $a_{n} = \sqrt{{\log n} \over nh^{d+p}}$. As $n \to \infty$, 
\begin{enumerate}
\item[(R1)] $h^{-(d+p)}a_{n}^{d+p}A_{n}^{d}A_{1,n}^{-d}\beta(A_{2,n};A_{n}^{d}) \to 0$ and $A_{1,n}^{d}A_{n}^{-d}nh^{d+p}(\log n) \to 0$, 
\item[(R2)] ${n^{1/2}h^{(d+p)/2} \over A_{1,n}^{d}n^{1/\zeta}} \geq C_{0}n^{\eta}$ for some $0<C_{0}<\infty$ and $\eta>0$,
\item[(R3)] $A_{n}^{dr}h^{p} \to \infty$, 
\end{enumerate}
where $\zeta$ is a positive constant that appears in Assumption \ref{Ass-U}.   
\end{assumption}

Discussions on the assumptions regarding the mixing condition in Assumptions \ref{Ass-R} and \ref{Ass-R(add)} (introduced in Section \ref{asy-prop-m}) are given in Section \ref{Ex-LSRF} and Appendix \ref{Appendix-proof}.

\begin{proposition}\label{general-unif-rate}
Let $S_{c}$ be a compact subset of $\mathbb{R}^{p}$. Then, under Assumptions \ref{Ass-S}, \ref{Ass-KB}, \ref{Ass-U} and \ref{Ass-R}, the following results hold for $P_{\bm{S}}$ almost surely:  
\begin{align*}
\sup_{\bm{u} \in [0,1]^{d}, \bm{x} \in S_{c}}|\hat{\psi}(\bm{u}, \bm{x}) - E_{\cdot|\bm{S}}[\hat{\psi}(\bm{u}, \bm{x})]| &= O_{P_{\cdot|\bm{S}}}\left(\sqrt{\log n  \over nh^{d+p}}\right).
\end{align*}
\end{proposition}
\subsection{Asymptotic properties of $\hat{m}$}\label{asy-prop-m}

Next, we present our main results for the estimation of mean function $m$ for the model (\ref{NR-LSRF}).

\begin{theorem}\label{unif-rate-m}
Let $I_{h} = [C_{1}h, 1-C_{1}h]^{d}$ and let $S_{c}$ be a compact subset of $\mathbb{R}^{p}$. Suppose that $\inf_{\bm{u} \in [0,1]^{d}, \bm{x} \in S_{c}}f(\bm{u},\bm{x})>0$. Then, under Assumptions \ref{Ass-S}, \ref{Ass-M}, \ref{Ass-KB}, \ref{Ass-U} (with $W_{\bm{s}_{j},A_{n}}=1$ and $\epsilon_{\bm{s}_{j},A_{n}}$) and \ref{Ass-R}, the following result holds for $P_{\bm{S}}$ almost surely: 
\begin{align*}
\sup_{\bm{u} \in I_{h}, \bm{x} \in S_{c}}\left|\hat{m}(\bm{u},\bm{x}) - m(\bm{u},\bm{x})\right| &= O_{P_{\cdot|\bm{S}}}\left(\sqrt{{\log n  \over nh^{d+p}}} + h^{2} + {1 \over A_{n}^{dr}h^{p}}\right),
\end{align*}
where $r = \min\{1,\rho\}$.
\end{theorem}

The term $A_{n}^{-dr}h^{-p}$ in the convergence rate arises from the local stationarity of $\bm{X}_{\bm{s},A_{n}}$, i.e., the approximation error of $\bm{X}_{\bm{s},A_{n}}$ by a stationary random field $\bm{X}_{\bm{u}}(\bm{s})$. 

To demonstrate the asymptotic normality of the estimator, we additionally assume the following conditions: 
\begin{assumption}\label{Ass-U(add)}
\begin{enumerate}
\item[(Ua1)] (U1), (U2) and (U3) in Assumption \ref{Ass-U} hold. 
\item[(Ua2)] for all distinct $\bm{s}_{1}, \bm{s}_{2},\bm{s}_{3} \in R_{n}$, there exists a constant $C<\infty$ such that 
\begin{align*}
&\sup_{\bm{s}_{1},\bm{s}_{2}, \bm{s}_{3}, A_{n}}\sup_{\bm{x}_{1}, \bm{x}_{2} \bm{x}_{3} \in S_{c}}E[|W_{\bm{s}_{1}, A_{n}}||W_{\bm{s}_{2}, A_{n}}||W_{\bm{s}_{3}, A_{n}}||\bm{X}_{\bm{s}_{1},A_{n}}= \bm{x}_{1}, \bm{X}_{\bm{s}_{2},A_{n}} = \bm{x}_{2},\bm{X}_{\bm{s}_{3},A_{n}} = \bm{x}_{3}] \leq C.
\end{align*}
\end{enumerate}
\end{assumption}

\begin{assumption}\label{Ass-R(add)}
As $n \to \infty$, 
\begin{enumerate}
\item[(Ra1)] (R1) and (R2) in Assumption \ref{Ass-R} hold.
\item[(Ra2)] $A_{n}^{dr}h^{p+2}  \to \infty $.  
\item[(Ra3)] 
\[
\left({1 \over nh^{p+d}}\right)^{1/3}\left({A_{1,n} \over A_{n}}\right)^{2d/3}\left({A_{2,n} \over A_{1,n}}\right)^{2/3}g^{1/3}_{1}(A_{1,n}^{d})\sum_{k = 1}^{A_{n}/A_{1,n}}k^{d-1}\beta_{1}^{1/3}(kA_{1,n} + A_{2,n}) \to 0.
\] 
\end{enumerate}
\end{assumption}

The asymptotic normality of the kernel estimators can be established under $\alpha$-mixing that is weaker than $\beta$-mixing.  See also (\ref{al-mix-coef}) for the definition of $\alpha$-mixing coefficients for random fields. 
\begin{theorem}\label{general-CLT-m}
Suppose that $f(\bm{u}, \bm{x})>0$, $f_{\bm{S}}(\bm{u})>0$ and $\epsilon_{\bm{s}_{j},A_{n}} = \sigma\left({\bm{s}_{j} \over A_{n}}, \bm{x}\right)\epsilon_{j}$, where $\sigma(\cdot,\cdot)$ is continuous and $\{\epsilon_{j}\}_{j=1}^{n}$ is a sequence of i.i.d. random variables with mean zero and variance $1$. Moreover, suppose $nh^{d+p+4} \to c_{0}$ for a constant $c_{0}$. Then under Assumptions \ref{Ass-S}, \ref{Ass-M}, \ref{Ass-KB}, \ref{Ass-U(add)} and \ref{Ass-R(add)} with $W_{\bm{s}_{j},A_{n}}=1$ and $\epsilon_{\bm{s}_{j},A_{n}}$in Assumption 3.5, the following result holds for $P_{\bm{S}}$ almost surely: 
\begin{align*}
\sqrt{nh^{d+p}}(\hat{m}(\bm{u},\bm{x}) - m(\bm{u},\bm{x})) \stackrel{d}{\longrightarrow} N(B_{\bm{u},\bm{x}}, V_{\bm{u},\bm{x}}),
\end{align*}
where 
\begin{align*}
B_{\bm{u},\bm{x}} &= \sqrt{c_{0}}{\kappa_{2} \over 2}\left\{\sum_{i = 1}^{d}\left(2\partial_{u_{i}}m(\bm{u},\bm{x})\partial_{u_{i}}f(\bm{u},\bm{x}) + \partial_{u_{i}u_{i}}^{2}m(\bm{u},\bm{x})f(\bm{u},\bm{x})\right) \right. \\
&\left .\quad +  \sum_{k = 1}^{p}\left(2\partial_{x_{k}}m(\bm{u},\bm{x})\partial_{x_{k}}f(\bm{u},\bm{x}) + \partial_{x_{k}x_{k}}^{2}m(\bm{u},\bm{x})f(\bm{u},\bm{x})\right) \right\}
\end{align*}
and $V_{\bm{u},\bm{x}} = \kappa_{0}^{d+p}\sigma^{2}(\bm{u},\bm{x})/(f_{\bm{S}}(\bm{u})f(\bm{u},\bm{x}))$ with $\kappa_{0} = \int_{\mathbb{R}} K^{2}(x)dx$ and $\kappa_{2} = \int_{\mathbb{R}}x^{2}K(x)dx$. The same result holds true even if we replace the conditions on $\beta$-mixing coefficients in Assumptions \ref{Ass-M}, \ref{Ass-U(add)} and \ref{Ass-R(add)} with those on $\alpha$-mixing coefficients defined by (\ref{al-mix-coef}). 
\end{theorem}

The asymptotic variance $V_{\bm{u},\bm{x}}$ depends on the sampling density $f_{\bm{S}}(\bm{u})$ through $1/f_{S}(\bm{u})$. 
This dependence of variance on sampling density differs in stochastic and nonstochastic (or equidistant) sampling designs. Intuitively, the result implies that around ``hot-spots,'' where more sampling sites tend to be selected, the asymptotic variance of the estimator is smaller than that of the other spatial points. The same applies to Theorem \ref{CLT-m-add}. 

\begin{remark}\label{comment bandwidth-selection}
As shown in Theorem \ref{general-CLT-m}, the expressions of asymptotic bias and variance of the kernel estimator are very similar in structure to those from a standard random design for stationary time series and random fields. Therefore, we conjecture that the methods to choose the bandwidth in such a design can be adapted to our setting. In particular, using the formulas for the asymptotic bias and variance from Theorem \ref{general-CLT-m}, it should be possible to select the bandwidth via plug-in methods. 
\end{remark}

\begin{remark}\label{Remark-general-DP-condition}
Set $A_{n}^{d} = O(n^{1-\bar{\eta}_{1}})$ for some $\bar{\eta}_{1} \in [0,1)$, $A_{1,n} = O(A_{n}^{\gamma_{A_{1}}})$, $A_{2,n} = O(A_{n}^{\gamma_{A_{2}}})$ with $0< \gamma_{A_{2}}<\gamma_{A_{1}} < 1/3$ and $r = \min\{1,\rho\} = 1$. Assume that we can take a sufficiently large $\zeta>2$ such that ${2 \over \zeta}<(1-\bar{\eta}_{1})(1-3\gamma_{A_{1}})$. Then, Assumption \ref{Ass-R(add)} is satisfied for $d \geq 1$ and $p \geq 1$. See Remarks \ref{Rem-A4} and \ref{Rem-B1} for details. 
\end{remark}

\section{Additive models}\label{Section_Additive}

In the previous section, we considered general kernel estimators and discussed their asymptotic properties; however, the estimators are adversely affected by dimensionality. In particular, the convergence rate $O_{P_{\cdot|\bm{S}}}\left(\sqrt{\log n/nh^{d+p}}\right)$ of the estimators deteriorated as the dimension of the covariates $p$ increases. 
Hence, we consider additive models inspired by the idea presented in \cite{Vo12}, which is based on the smooth backfitting method developed by \cite{MaLiNi99}, and studied the asymptotic properties of the estimators of additive functions.

\subsection{Construction of estimators}

We place the following structural constraint on $m(\bm{u},\bm{x})$: 
\begin{align}
E[Y_{\bm{s},A_{n}}|\bm{X}_{\bm{s},A_{n}}=\bm{x}] &= m\left({\bm{s} \over A_{n}},\bm{x}\right) \nonumber \\
&= m_{0}\left({\bm{s} \over A_{n}}\right) + \sum_{\ell=1}^{p}m_{\ell}\left({\bm{s} \over A_{n}}, x_{\ell}\right). \label{AddNR-RF}
\end{align}
Model (\ref{AddNR-RF}) is a natural extension of the following linear regression models with spatially varying coefficients (cf. \cite{GeKiSiBa03}): 
\[
Y_{\bm{s},A_{n}} = \beta_{0}\left({\bm{s} \over A_{n}}\right) + \sum_{\ell=1}^{p}\beta_{\ell}\left({\bm{s} \over A_{n}}\right)X_{\bm{s},A_{n}}^{\ell} + \epsilon_{\bm{s},A_{n}}.
\]

\begin{remark}\label{comment spatial-coordinate}
In this paper, we focus on the circumvention of the curse of dimensionality that comes from the number of covariates since in some applications it seems not suitable to consider additive components over the spatial coordinates. For example, many environmental and climate data such as temperature and precipitation would have not spatial coordinate-wise but location specific features. Moreover, one of our motivation is extending the spatially varying linear regression models in \cite{GeKiSiBa03}, which is one of the influential papers in spatial statistics, to nonparametric, non-stationary and non-Gaussian settings. They apply their methods to the analysis of log selling price of single family homes and their regression model does not include additive components over spatial coordinates. \cite{BrMa17} also considers regression models that do not include additive components over spatial coordinates and estimating parameters of a CARMA random field with an application of their methods to land price data. Therefore, we believe our modeling would not be restrictive in many empirical applications.
\end{remark}

To identify the additive function of the model (\ref{AddNR-RF}) within a unit cube $[0,1]^{p}$, we impose the condition that 
\begin{align*}
\int m_{\ell}(\bm{u},x_{\ell})p_{\ell}(\bm{u},x_{\ell})dx_{\ell} = 0,\ \ell = 1,\hdots,p
\end{align*}
and all rescaled space points $\bm{u} \in [0,1]^{d}$. Here, the functions $p_{\ell}(\bm{u},x_{\ell}) = \int_{\mathbb{R}^{p}}p(\bm{u},\bm{x})d\bm{x}_{-\ell}$ are the marginals of the density
\begin{align*}
p(\bm{u},\bm{x}) &= {I(\bm{x} \in [0,1]^{p})f(\bm{u}, \bm{x}) \over P(X_{\bm{u}}(\bm{0}) \in [0,1]^{p})},
\end{align*}
where $f(\bm{u}, \cdot)$ is the density of the strictly stationary random field $\{X_{\bm{u}}(\bm{s})\}$. 

To estimate the functions $m_{0},\hdots,m_{p}$, we apply the strategy used in \cite{Vo12}, which is based on the smooth backfitting technique developed in \cite{MaLiNi99}. First, we introduce the auxiliary estimates.
\begin{align*}
\hat{p}(\bm{u},\bm{x}) &= {1 \over n_{[0,1]^{p}}}\sum_{j = 1}^{n}I(\bm{X}_{\bm{s}_{j},A_{n}} \in [0,1]^{p})\bar{K}_{h}\left(\bm{u}, {\bm{s}_{j} \over A_{n}}\right)\prod_{\ell = 1}^{p}K_{h}\left(x_{\ell}, X_{\bm{s}_{j}, A_{n}}^{\ell}\right),\\
\hat{m}(\bm{u}, \bm{x}) &= {1 \over n_{[0,1]^{p}}}\sum_{j = 1}^{n}I(\bm{X}_{\bm{s}_{j},A_{n}} \in [0,1]^{p})\bar{K}_{h}\left(\bm{u}, {\bm{s}_{j} \over A_{n}}\right)\prod_{\ell = 1}^{p}K_{h}\left(x_{\ell}, X_{\bm{s}_{j}, A_{n}}^{\ell}\right)Y_{\bm{s}_{j}, A_{n}}/\hat{p}(\bm{u}, \bm{x}),
\end{align*}
where $\bar{K}_{h}\left(\bm{u}, {\bm{s}_{j} \over A_{n}}\right) = \prod_{\ell = 1}^{d}K_{h}\left(u_{\ell}, {s_{j,\ell} \over A_{n}}\right)$, $\hat{p}(\bm{u}, \bm{x})$ is a kernel estimator for the function $f_{\bm{S}}(\bm{u})p(\bm{u},\bm{x})$, and $\hat{m}(\bm{u}, \bm{x})$ is a $(p+d)$-dimensional kernel estimator that estimates $m(\bm{u},\bm{x})$ for $\bm{x} \in [0,1]^{p}$. In the aforementioned definitions, 
\begin{align*}
n_{[0,1]^{p}} &= \sum_{j = 1}^{n}\bar{K}_{h}\left(\bm{u}, {\bm{s}_{j} \over A_{n}}\right)I(\bm{X}_{\bm{s}_{j}, A_{n}} \in [0,1]^{p})/\tilde{f}_{S}(\bm{u})
\end{align*}
is the number of observations in the unit cube $[0,1]^{p}$, where only space points close to $\bm{u}$ are considered. Furthermore, $\tilde{f}_{\bm{S}}(\bm{u}) = {1 \over n}\sum_{j=1}^{n}\bar{K}_{h}\left(\bm{u}, {\bm{s}_{j} \over A_{n}}\right)$, and 
\begin{align*}
K_{h}(v,w) &= I(v,w \in [0,1]){K_{h}(v-w) \over \int_{[0,1]}K_{h}(s-w)ds}
\end{align*}
is a modified kernel weight. This weight possesses the property that $\int_{[0, 1]}K_{h}(v,w)dv = 1$ for all $w \in [0,1]$, which is required to derive the asymptotic properties of the backfitting estimates. 

Given the estimators $\hat{p}$ and $\hat{m}$, we define the smooth backfitting estimates $\tilde{m}_{0}(\bm{u}), \tilde{m}_{1}(\bm{u},\cdot),\hdots,\tilde{m}_{d}(\bm{u},\cdot)$ of the functions $m_{0}(\bm{u}), m_{1}(\bm{u},\cdot),\hdots,m_{d}(\bm{u},\cdot)$ at the space point $\bm{u} \in [0,1]^{d}$ as the minimizer of the criterion
\begin{align}\label{Criterion-func}
\int_{[0,1]^{p}}\left\{\hat{m}(\bm{u},\bm{w}) - \left(g_{0} + \sum_{\ell=1}^{p}g_{\ell}(w_{\ell})\right)\right\}^{2}\hat{p}(\bm{u},\bm{w})d\bm{w},
\end{align}
where the minimization runs over all additive functions $g(\bm{x}) = g_{0} + \sum_{\ell=1}^{p}g_{\ell}(x^{\ell})$, whose components are normalized to satisfy 
\begin{align*}
\int g_{\ell}(w_{\ell})\hat{p}_{\ell}(\bm{u},w_{\ell})dw_{\ell} = 0,\ \ell = 1,\hdots,p.
\end{align*}
Here, $\hat{p}_{\ell}(\bm{u},x_{\ell}) = \int_{[0,1]^{p-1}}\hat{p}(\bm{u},\bm{x})d\bm{x}_{-\ell}$ is the marginal of the kernel density $\hat{p}(\bm{u},\cdot)$ at point $x_{\ell}$. According to (\ref{Criterion-func}), the estimate $\tilde{m}(\bm{u},\cdot) = \tilde{m}_{0}(\bm{u}) + \sum_{\ell = 1}^{p}\tilde{m}_{\ell}(\bm{u},\cdot)$ is an $L^{2}$-projection of the full dimensional kernel estimate $\hat{m}(\bm{u},\cdot)$ on to the subspace of additive functions, where the projection is performed with respect to the density estimate $\hat{p}(\bm{u},\cdot)$. 

By differentiation, we can demonstrate that the solution of (\ref{Criterion-func}) is characterized by the system of equations
\begin{align*}
\tilde{m}_{\ell}(\bm{u},x_{\ell}) &= \hat{m}_{\ell}(\bm{u},x_{\ell}) - \sum_{k \neq \ell}\int_{[0,1]}\tilde{m}_{k}(\bm{u},x_{k}){\hat{p}_{\ell,k}(\bm{u}, x_{\ell}, x_{k}) \over \hat{p}_{\ell}(\bm{u}, x_{\ell})}dx_{k} - \tilde{m}_{0}(\bm{u})
\end{align*}
together with 
\begin{align*}
\int \tilde{m}_{\ell}(\bm{u}, w_{\ell})\hat{p}_{\ell}(\bm{u},w_{\ell})dw_{\ell} = 0,\ \ell = 1,\hdots, p,
\end{align*}
where $\hat{p}_{\ell}$ and $\hat{p}_{\ell, k}$ are kernel density estimates, and $\hat{m}_{\ell}$ is a kernel estimator defined as
\begin{align*}
\hat{p}_{\ell}(\bm{u},x_{\ell}) &= {1 \over n_{[0,1]^{p}}}\sum_{j=1}^{n}I\left(\bm{X}_{\bm{s}_{j}, A_{n}} \in [0,1]^{p}\right)\bar{K}_{h}\left(\bm{u}, {\bm{s}_{j} \over A_{n}}\right)K_{h}\left(x_{\ell}, X_{\bm{s}_{j}, A_{n}}^{\ell}\right), \\
\hat{p}_{\ell, k}(\bm{u},x_{\ell}, x_{k}) &= {1 \over n_{[0,1]^{p}}}\sum_{j=1}^{n}I\left(\bm{X}_{\bm{s}_{j}, A_{n}} \in [0,1]^{p}\right)\bar{K}_{h}\left(\bm{u}, {\bm{s}_{j} \over A_{n}}\right)K_{h}\left(x_{\ell}, X_{\bm{s}_{j}, A_{n}}^{\ell}\right)K_{h}\left(x_{k}, X_{\bm{s}_{j}, A_{n}}^{k}\right), \\
\hat{m}_{\ell}(\bm{u},x_{\ell}) &= {1 \over n_{[0,1]^{p}}}\sum_{j=1}^{n}I\left(\bm{X}_{\bm{s}_{j}, A_{n}} \in [0,1]^{p}\right)\bar{K}_{h}\left(\bm{u}, {\bm{s}_{j} \over A_{n}}\right)K_{h}\left(x_{\ell}, X_{\bm{s}_{j}, A_{n}}^{\ell}\right)Y_{\bm{s}_{j}, A_{n}}/\hat{p}_{\ell}(\bm{u}, x_{\ell}). 
\end{align*}
Moreover, the estimate $\tilde{m}_{0}(\bm{u})$ of the model constant at space point $\bm{u}$ is expressed as 
\[
\tilde{m}_{0}(\bm{u}) =  {1 \over n_{[0,1]^{p}}}\sum_{j=1}^{n}I\left(\bm{X}_{\bm{s}_{j}, A_{n}} \in [0,1]^{p}\right)\bar{K}_{h}\left(\bm{u}, {\bm{s}_{j} \over A_{n}}\right)Y_{\bm{s}_{j}, A_{n}}/\bar{f}_{\bm{S}}(\bm{u}), 
\]
where $\bar{f}_{\bm{S}}(\bm{u}) = {1 \over n_{[0,1]^{p}}}\sum_{j=1}^{n}I\left(\bm{X}_{\bm{s}_{j}, A_{n}} \in [0,1]^{p}\right)\bar{K}_{h}\left(\bm{u}, {\bm{s}_{j} \over A_{n}}\right)$.

\subsection{Asymptotic properties of estimators}

Now we present the uniform convergence rate and joint asymptotic normality of the estimators. Before describing the results, we summarize the set of assumptions made.  

\begin{assumption}\label{Ass-Rb}
Let $a_{n} = \sqrt{{\log n} \over nh^{d+1}}$ and $r = \min\{\rho, 1\}$, where $\rho$ is a positive constant that appears in Assumption \ref{Ass-M}. As $n \to \infty$, 
\begin{enumerate}
\item[(Rb1)] $h^{-(d+1)}a_{n}^{d+1}A_{n}^{d}A_{1,n}^{-d}\beta(A_{2,n};A_{n}^{d}) \to 0$ and $A_{1,n}^{d}A_{n}^{-d}nh^{d+1}(\log n) \to 0$, 
\item[(Rb2)] ${n^{1/2}h^{(d+1)/2} \over A_{1,n}^{d}n^{1/\zeta}} \geq C_{0}n^{\eta}$ for some $C_{0}>0$ and $\eta>0$,
\item[(Rb3)] $A_{n}^{dr}h^{3} \to \infty$ and $A_{n}^{dr/(1+r)}h^{2} \to \infty$,
\item[(Rb4)] $nh^{d+4} \to \infty$. 
\item[(Rb5)] (Ra3) in Assumption \ref{Ass-R(add)} holds with $p=1$,
\end{enumerate}
where $\zeta$ is a positive constant that appears in Assumption \ref{Ass-U}.   
\end{assumption}

The following results are extension of Theorems 5.1 and 5.2 in \cite{Vo12} to locally stationary random fields with irregularly spaced observations. 

\begin{theorem}\label{unif-m-add}
Let $I_{h, 0} = [2C_{1}h, 1-2C_{1}h]$ and $I_{h} = [2C_{1}h, 1-2C_{1}h]^{d}$. Suppose that $\inf_{\bm{u} \in [0,1]^{d}, \bm{x} \in [0,1]^{p}}f(\bm{u},\bm{x})>0$ and $\inf_{\bm{u} \in [0,1]^{d}}f_{\bm{S}}(\bm{u})>0$. Then, under Assumptions \ref{Ass-S}, \ref{Ass-M}, \ref{Ass-KB}, \ref{Ass-U(add)} (with $W_{\bm{s}_{j},A_{n}}=1$ and $\epsilon_{\bm{s}_{j},A_{n}}$) and \ref{Ass-Rb}, the following result holds for $P_{\bm{S}}$ almost surely: 
\begin{align*}
\sup_{\bm{u} \in I_{h}, x_{\ell} \in I_{h,0}}\left|\tilde{m}_{\ell}(\bm{u}, x_{\ell}) - m_{\ell}(\bm{u}, x_{\ell})\right| = O_{P_{\cdot|\bm{S}}}\left(\sqrt{{\log n \over nh^{d+1}}} + h^{2}\right),\ \ell = 1,\hdots,p. 
\end{align*} 
\end{theorem}


As is the case with general regression function, the asymptotic normality of the kernel estimators for additive functions can be established under $\alpha$-mixing.  
\begin{theorem}\label{CLT-m-add}
Suppose that $\inf_{\bm{u} \in [0,1]^{d}, \bm{x} \in [0,1]^{p}}f(\bm{u},\bm{x})>0$ and $\inf_{\bm{u} \in [0,1]^{d}}f_{\bm{S}}(\bm{u})>0$. Moreover, suppose that $\epsilon_{\bm{s}_{j},A_{n}}$ given $\bm{X}_{\bm{s}_{j},A_{n}}$ are i.i.d. random variables and continuous functions $\sigma_{\ell}(\bm{u}, x_{\ell})$, $1 \leq \ell \leq p$ exist such that $\sigma^{2}_{\ell}({\bm{s}_{j} \over A_{n}}, x_{\ell}) = E_{\cdot|\bm{S}}[\epsilon_{\bm{s}_{j},A_{n}}^{2}|X_{\bm{s}_{j},A_{n}}^{\ell} = x_{\ell}]$ and that $n_{[0,1]^{p}}h^{d+5} - c_{0} = o_{P_{\cdot|S}}(1)$ for a constant $c_{0}$, $P_{\bm{S}}$ almost surely. Then, under Assumptions \ref{Ass-S}, \ref{Ass-M}, \ref{Ass-KB}, \ref{Ass-U(add)}  and \ref{Ass-Rb} with $W_{\bm{s}_{j},A_{n}}=1$ and $\epsilon_{\bm{s}_{j},A_{n}}$ in Assumption \ref{Ass-U(add)}, for any $\bm{u} \in (0,1)^{d}$, $x_{1},\hdots, x_{p} \in (0,1)$, the following result holds for $P_{\bm{S}}$ almost surely:
\begin{align*}
\sqrt{n_{[0,1]^{p}}h^{d+1}}\left(
\begin{array}{c}
\tilde{m}_{1}(\bm{u},x_{1}) - m_{1}(\bm{u},x_{1}) \\
\vdots \\
\tilde{m}_{p}(\bm{u},x_{p}) - m_{p}(\bm{u},x_{p})
\end{array}
\right) &\stackrel{d}{\longrightarrow} N(B_{\bm{u},\bm{x}}, V_{\bm{u},\bm{x}}).
\end{align*}
Here, $V_{\bm{u},\bm{x}} = \text{diag}(v_{1}(\bm{u}, x_{1}),\hdots,v_{p}(\bm{u}, x_{p}))$ is a $p \times p$ diagonal matrix with 
\[
v_{\ell}(\bm{u}, x_{\ell}) =\kappa_{0}^{d+1}\sigma_{\ell}^{2}(\bm{u},x_{\ell})/(f_{\bm{S}}(\bm{u})p_{\ell}(\bm{u},x_{\ell})),
\] 
where $\kappa_{0} = \int_{\mathbb{R}}K^{2}(x)dx$. $B_{\bm{u},\bm{x}}$ has of the form 
\[
B_{\bm{u},\bm{x}} = \sqrt{c_{0}}(\beta_{1}(\bm{u},x_{1})-\gamma_{1}(\bm{u}),\hdots,\beta_{p}(\bm{u},x_{p})-\gamma_{p}(\bm{u}))'.
\]
The functions $\beta_{\ell}(\bm{u}, \cdot)$ are defined as the minimizer of the problem
\[
\int_{\mathbb{R}^{p}}\left\{\beta(\bm{u},\bm{x}) - \left(b_{0} + \sum_{\ell=1}^{p}b_{\ell}(x_{\ell})\right)\right\}^{2}p(\bm{u},\bm{x})d\bm{x},
\]
where the minimization runs over all additive functions $b(x) = b_{0} + \sum_{\ell=1}^{p}b_{\ell}(x_{\ell})$ with $\int_{\mathbb{R}}b_{\ell}(x_{\ell})p_{\ell}(\bm{u},x_{\ell})dx_{\ell} = 0$; the function $\beta(\bm{u},\bm{x})$ is provided in Lemma \ref{lemma-C4} in Appendix. $\gamma_{\ell}(\bm{u})$ can be characterized by the equation $\int_{\mathbb{R}}\alpha_{\ell}(\bm{u},x_{\ell})\hat{p}_{\ell}(\bm{u},x_{\ell})dx_{\ell} = h^{2}\gamma_{\ell}(\bm{u}) + o_{P_{\cdot|\bm{S}}}(h^{2})$, where $\alpha_{\ell}(\bm{u},x_{\ell})$ are defined in Lemma \ref{lemma-C4}. 

The same result holds true even if we replace the conditions on $\beta$-mixing coefficients in Assumptions \ref{Ass-M}, \ref{Ass-U(add)} and \ref{Ass-Rb} with those on $\alpha$-mixing coefficients defined by (\ref{al-mix-coef}). 
\end{theorem}

\begin{remark}\label{Add-discuss}
Set $A_{n}^{d} = {n^{1-\bar{\eta}_{1}} \over (\log n)^{\bar{\eta}_{2}}}$ for some $\bar{\eta}_{1} \in [0,1)$ and $\bar{\eta}_{2} \geq 1$, $A_{1,n} = O(A_{n}^{\gamma_{A_{1}}})$ with $\gamma_{A_{1}} \in (0,1/3)$ and $r = 1$. Assume that we can take a sufficiently large $\zeta>2$. In this case, Assumption \ref{Ass-Rb} is satisfied for $d \geq 1$ and $p \geq 1$. 
\end{remark}

\section{Examples of locally stationary random fields on $\mathbb{R}^{d}$}\label{Ex-LSRF}
In this section, we present examples of locally stationary random fields. In particular, we discuss L\'evy-driven moving average (MA) random fields, which represent a wide class of random fields. One of the important features of this class is that it includes non-Gaussian random fields in addition to Gaussian random fields. Moreover, we provide sufficient conditions to demonstrate that a random field is locally stationary and can be approximated by a stationary L\'evy-driven MA random field. We focus on the case of $p=1$ to simplify our discussion. The results in this section can be extended to the multivariate case ($p \geq 2$). 
See Appendix \ref{Multi-LevyMA} herein for the definition of multivariate L\'evy-driven MA random fields on $\mathbb{R}^{d}$ and for a discussion on their properties. 

\subsection{L\'evy-driven MA random fields}\label{Section-MALevy}

\cite{BrMa17} generalized CARMA($p,q$) processes on $\mathbb{R}$ to CARMA ($p,q$) random fields on $\mathbb{R}^{d}$, which is a special class of L\'evy-driven MA random fields. See also \cite{Br00}, \cite{Br01}, \cite{MaSt07} and \cite{ScSt12} for examples of the CARMA process on $\mathbb{R}$ and \cite{MaYu20} for multivariate extension of CARMA random fields on $\mathbb{R}^{2}$ and their parametric inference. Let $L=\{L(A): A \in \mathcal{B}(\mathbb{R}^{d})\}$ be an infinitely divisible random measure on some probability space $(\Omega, \mathcal{A}, P)$, i.e., a random measure such that 
\begin{itemize}
\item[1.] for each sequence $(E_{m})_{m \in \mathbb{N}}$ of disjoint sets in $\mathcal{B}(\mathbb{R}^{d})$,
\begin{itemize}
\item[(a)] $L(\cup_{m=1}^{\infty}E_{m}) = \sum_{m=1}^{\infty}L(E_{m})$ a.s. whenever $\cup_{m=1}^{\infty}E_{m} \in \mathcal{B}(\mathbb{R}^{d})$,
\item[(b)] $(L(E_{m}))_{m \in \mathbb{N}}$ is a sequence of independent random variables. 
\end{itemize}
\item[2.] the random variable $L(A)$ has an infinitely divisible distribution for any $A \in \mathcal{B}(\mathbb{R}^{d})$. 
\end{itemize}
The characteristic function of $L(A)$ which will be denoted by $\varphi_{L(A)}(t)$, has a L\'evy--Khintchine representation of the form $\varphi_{L(A)}(t) = \exp\left(|A|\psi(t)\right)$ with 
\begin{align*}
\psi(t) &= it\gamma_{0} - {1 \over 2}t^{2}\sigma_{0} + \int_{\mathbb{R}}\left\{e^{itx}-1-itxI(x \in [-1,1])\right\}\nu_{0}(x)dx
\end{align*}
where $i = \sqrt{-1}$, $\gamma_{0} \in \mathbb{R}$, $0 \leq \sigma_{0} <\infty$, $\nu_{0}$ is a L\'evy density with $\int_{\mathbb{R}}\min\{1,x^{2}\}\nu_{0}(x)dx<\infty$, and $|A|$ is the Lebesgue measure of $A$. The triplet $(\gamma_{0}, \sigma_{0}, \nu_{0})$ is called the L\'evy characteristic of $L$; it uniquely determines the distribution of random measure $L$. We refer to \cite{Sa99} and \cite{Be96} as standard references on L\'evy processes and \cite{RaRo89} for details on the theory of infinitely divisible measures and fields. Let $a(z) = z^{p} + a_{1}z^{p-1} + \cdots+a_{p} = \prod_{i=1}^{p}(z - \lambda_{i})$ be a polynomial of degree $p$ with real coefficients and distinct negative zeros $\lambda_{1},\hdots,\lambda_{p}$, and let $b(z) = b_{0} + b_{1}z + \cdots + b_{q}z^{q} = \prod_{i=1}^{q}(z - \xi_{i})$ be a polynomial with real coefficients and real zeros $\xi_{1},\hdots, \xi_{q}$ such that $b_{q}=1$ and $0\leq q < p$ and $\lambda_{i}^{2} \neq \xi_{j}^{2}$ for all $i$ and $j$. Define $a(z) = \prod_{i=1}^{p}(z^{2} - \lambda_{i}^{2})$ and $b(z) = \prod_{i=1}^{q}(z^{2} - \xi_{i}^{2})$. A L\'evy-driven MA random field driven by an infinitely divisible random measure $L$, which we call L\'evy random measure, is defined by
\begin{align}\label{CARMA-RF}
X(\bm{s}) = \int_{\mathbb{R}^{d}}g(\bm{s}-\bm{v})L(d\bm{v})
\end{align}
for every $\bm{s} \in \mathbb{R}^{d}$. In particular, when $g(\cdot)$ is a kernel function of the form 
\begin{align}\label{CARMA_kernel_general}
g(\bm{s}) = \sum_{i=1}^{p}{b(\lambda_{i}) \over a'(\lambda_{i})}e^{\lambda_{i}\|\bm{s}\|},
\end{align}
where $a'$ denotes the derivative of the polynomial $a$, $X(\bm{s})$ is a univariate (isotropic) CARMA($p,q$) random field.

\begin{remark}[Connections to SPDEs]
\cite{Be19} characterizes (\ref{CARMA-RF}) as a solution of a (fractional) stochastic partial differential equation (SPDE). The author also extends the concept of CARMA($p,q$) random fields as a strictly stationary solution of an SPDE. The uniqueness of the solution is discussed in \cite{Be19b}.
\end{remark}

\subsection{Stationary distribution of L\'evy-driven MA random fields}

When the L\'evy-driven MA random field (\ref{CARMA-RF}) is strictly stationary, the characteristic function of the stationary distribution of $X$ is expressed as
\begin{align}\label{CF-CARMA}
\varphi_{X(\bm{0})}(t) &= E[e^{itX(\bm{0})}]= \exp\left(\int_{\mathbb{R}^{d}}H(tg(\bm{s}))d\bm{s}\right),
\end{align}
where 
\begin{align*}
\int_{\mathbb{R}^{d}}H(tg(\bm{s}))d\bm{s} &=  it\gamma_{1} - {1 \over 2}t^{2}\sigma_{1} + \int_{\mathbb{R}}\left\{e^{itx}-1-itxI(x \in [-1,1])\right\}\nu_{1}(x)dx,
\end{align*}
with 
\begin{align*}
\gamma_{1} &= \int_{\mathbb{R}^{d}}U(g(\bm{s}))d\bm{s},\ \sigma_{1} = \sigma_{0}\int_{\mathbb{R}^{d}}g^{2}(\bm{s})d\bm{s},\ \nu_{1}(x) = \int_{S_{g}}{1 \over |g(\bm{s})|}\nu_{0}\left({x \over g(\bm{s})}\right)d\bm{s}.
\end{align*}
Here, $S_{g} = \supp(g) = \{\bm{s} \in \mathbb{R}^{d}: g(\bm{s}) \neq 0\}$ denotes the support of $g$; the function $U$ is defined as follows:
\begin{align*}
U(u) &= u\left(\gamma_{0} + \int_{\mathbb{R}}x\left\{I(ux \in [-1,1]) - I(x \in [-1,1])\right\}\nu_{0}(x)dx\right).
\end{align*} 
The triplet $(\gamma_{1}, \sigma_{1}, \nu_{1})$ is again referred to as the L\'evy characteristic of $X(\bm{0})$ and determines the distribution of $X(\bm{0})$ uniquely. 
See \cite{KaRoSpWa19} for details. The representation of (\ref{CF-CARMA}) implies that the stationary distribution of $X(\bm{s})$ has a density function when the L\'evy random measure $L$ is Gaussian, i.e., $(\gamma_{0}, \sigma_{0}, \nu_{0}) = (\gamma_{0}, \sigma_{0}, 0)$. When $L$ is purely non-Gaussian, i.e., $(\gamma_{0}, \sigma_{0}, \nu_{0}) = (\gamma_{0},0,\nu_{0})$, we can demonstrate that $X(\bm{s})$ has a bounded continuous, infinitely differentiable density if 
\begin{align}\label{LS-Levy}
\nu_{0}(x) &= {c_{\beta} \over |x|^{1+\beta}}g_{0}(x),\ x \neq 0
\end{align} 
where $c_{\beta}>0$, $g_{0}$ is a positive, continuous, and bounded function on $\mathbb{R}$ with $\lim_{|x| \downarrow 0}g_{0}(x) = 1$ and $\beta \in (0,1)$. In fact, we can verify that a constant $\tilde{C}>0$ exists such that
\begin{align}\label{SC-density}
u^{2}\int_{\mathbb{R}^{d}}g^{2}(\bm{v})\int_{|x| \leq {1 \over |g(\bm{v})||u|}}x^{2}\nu_{0}(x)dxd\bm{v} \geq \tilde{C}|u|^{2-\alpha}
\end{align}  
for any $|u| \geq c_{0}>1$, where $\alpha = 2-\beta \in (0,2)$. This implies the existence of the density function of $X(\bm{s})$. See Lemma \ref{LS-density} for the proof of (\ref{SC-density}) and Lemma \ref{density-Levy-MARF} for the sufficient condition for the existence of the density of $\bm{X}(\bm{s})$ with a purely non-Gaussian L\'evy random measure in a multivariate case. A L\'evy process with a L\'evy density of the form (\ref{LS-Levy}) is called a locally stable L\'evy process and it can represent important L\'evy processes such as tempered stable and normal inverse Gaussian processes (cf. Section 2 in \cite{Ma19}, Section 5 in \cite{KaKu20} and \cite{KuKaSh21}).

\subsection{Locally stationary L\'evy-driven MA random fields}\label{LSLevyMARF}
Next, we present examples of locally stationary L\'evy-driven MA random fields.
Consider the following processes
\begin{align*}
X_{\bm{s},A_{n}} &= \int_{\mathbb{R}^{d}}g\left({\bm{s} \over A_{n}}, \|\bm{s} - \bm{v}\|\right)L(d\bm{v})\ \text{and}\ X_{\bm{u}}(\bm{s}) = \int_{\mathbb{R}^{d}}g\left(\bm{u}, \|\bm{s} - \bm{v}\|\right)L(d\bm{v}),
\end{align*}
where $g: [0,1]^{d}\times [0,\infty) \to \mathbb{R}$ is a bounded function such that $|g(\bm{u},\cdot) - g(\bm{v},\cdot)| \leq C\|\bm{u} - \bm{v}\|\bar{g}(\cdot)$ with $C<\infty$ and for any $\bm{u} \in [0,1]^{d}$, $\int_{\mathbb{R}^{d}}(|g(\bm{u},\|\bm{s}\|)| + |\bar{g}(\bm{s})|)d\bm{s}<\infty$ and $\int_{\mathbb{R}^{d}}(g^{2}(\bm{u},\|\bm{s}\|) +  \bar{g}^{2}(\bm{s}))d\bm{s}<\infty$. Notably, $X_{\bm{u}}(\bm{s})$ is a strictly stationary random field for each $\bm{u}$. If $E[|L(A)|^{q'}]<\infty$ for any $A \in \mathcal{B}(\mathbb{R}^{d})$ with a bounded Lebesgue measure $|A|$ and for $q' =1,2$, we have that
\[
E[X_{\bm{u}}(\bm{s})] = \mu_{0}\int_{\mathbb{R}^{d}}g\left(\bm{u}, \|\bm{s}\|\right)d\bm{s}\ \text{and}\ E\left[(\bm{X}_{\bm{u}}(\bm{s}))^{2}\right] = \bar{\sigma}_{0}^{2}\int_{\mathbb{R}^{d}}g^{2}\left(\bm{u}, \|\bm{s}\|\right)d\bm{s},
\]
where $\mu_{0} = -i\psi'(0)$ and $\bar{\sigma}_{0}^{2} = -\psi''(0)$. Let $m >0$. Define the function $\iota(\cdot : m) : [0,\infty) \to [0,1]$ as 
\[
\iota(x:m) = 
\begin{cases}
1 & \text{if $0\leq x \leq {m \over 2}$},\\
-{2 \over m}x+2 & \text{if ${m \over 2} < x \leq m$},\\
0 & \text{if $m<x$}
\end{cases}
\]
and consider the process $X_{\bm{u}}(\bm{s}:A_{2,n}) = \int_{\mathbb{R}^{d}}g\left(\bm{u}, \|\bm{s} - \bm{v}\|\right)\iota(\|\bm{s} - \bm{v}\|:A_{2,n})L(d\bm{v})$. Notably, $X_{\bm{u}}(\bm{s}:A_{2,n})$ is also an $2A_{2,n}$-dependent strictly stationary random field, i.e., the $\beta$-mixing coefficients $\beta(a;b) = \beta_{1}(a)g_{1}(b)$ of $X_{\bm{u}}(\bm{s}:A_{2,n})$ satisfy $\beta_{1}(a) = 0$ for $a \geq 2A_{2,n}$. Observe that 
\begin{align}
\left|X_{\bm{s},A_{n}} - X_{\bm{u}}(\bm{s}:A_{2,n})\right| &\leq \left|X_{\bm{s},A_{n}} - X_{\bm{u}}(\bm{s})\right| + \left|X_{\bm{u}}(\bm{s}) - X_{\bm{u}}(\bm{s}:A_{2,n})\right| \nonumber \\
&\leq \int_{\mathbb{R}^{d}}\left|g\left({\bm{s} \over A_{n}}, \|\bm{s} - \bm{v}\|\right) - g\left(\bm{u}, \|\bm{s} - \bm{v}\|\right)\right||L(d\bm{v})| \nonumber \\
&\quad + {1 \over A_{n}^{d}}\int_{\mathbb{R}^{d}}A_{n}^{d}|g\left(\bm{u}, \|\bm{s} - \bm{v}\|\right)|(1 - \iota(\|\bm{s} - \bm{v}\|:A_{2,n}))|L(d\bm{v})| \nonumber \\
&\leq \left\|{\bm{s} \over A_{n}} - \bm{u}\right\|\int_{\mathbb{R}^{d}}C\bar{g}(\|\bm{s} - \bm{v}\|)|L(d\bm{v})| \nonumber \\
&\quad  + {1 \over A_{n}^{d}}\int_{\mathbb{R}^{d}}A_{n}^{d}\underbrace{|g\left(\bm{u}, \|\bm{s} - \bm{v}\|\right)|(1 - \iota(\|\bm{s} - \bm{v}\|:A_{2,n}))}_{=: g_{\bm{u}}(\|\bm{s}-\bm{v}\|:A_{2,n})}|L(d\bm{v})| \nonumber \\
&\leq \left(\left\|{\bm{s} \over A_{n}} - \bm{u}\right\| + {1 \over A_{n}^{d}}\right)\int_{\mathbb{R}^{d}}(\underbrace{ C\bar{g}(\|\bm{s} - \bm{v}\|) +  A_{n}^{d}g_{\bm{u}}(\|\bm{s}-\bm{v}\|:A_{2,n})}_{=:g_{\bm{u},A_{n},A_{2,n}}(\|\bm{s} - \bm{v}\|)})|L(d\bm{v})| \nonumber \\
&=: \left(\left\|{\bm{s} \over A_{n}} - \bm{u}\right\| + {1 \over A_{n}^{d}}\right)U_{\bm{s},A_{n}}(\bm{u}). \label{m-dep-approx}
\end{align}
Here, we define $|L(A)|$ as the absolute value of the random variable $L(A)$ for any $A \in \mathcal{B}(\mathbb{R}^{d})$. If we assume that
\begin{align}\label{ass-decay}
\sup_{n \geq 1}\sup_{\bm{u} \in [0,1]^{d}}\int_{\mathbb{R}^{d}}(A_{n}^{d}g_{\bm{u}}(\|\bm{s}\|:A_{2,n}) + A_{n}^{2d}g_{\bm{u}}^{2}(\|\bm{s}\|:A_{2,n}))d\bm{s}<\infty,
\end{align}
we obtain
\begin{align}\label{U-moment}
E[U_{\bm{s}, A_{n}}^{2}(\bm{u})] &\leq  C\sup_{n \geq 1}\sup_{\bm{u} \in [0,1]^{d}}\left(\int_{\mathbb{R}^{d}}g_{\bm{u},A_{n},A_{2,n}}^{2}(\|\bm{s}\|)d\bm{s} + \left(\int_{\mathbb{R}^{d}}g_{\bm{u},A_{n},A_{2,n}}(\|\bm{s}\|)d\bm{s}\right)^{2}\right)<\infty.
\end{align}
See \cite{BrMa17} for details regarding the computation of the moments of L\'evy-driven MA random fields. Note that (\ref{m-dep-approx}) and (\ref{U-moment}) imply that $X_{\bm{s},A_{n}}$ is a locally stationary random field. 

\begin{remark}\label{remark-example}
We also note that (\ref{m-dep-approx}) and (\ref{U-moment}) imply that $X_{\bm{s},A_{n}}$ can be approximated by an $A_{2,n}$-dependent stationary random field under the condition (\ref{ass-decay}). In this case, the random field $X_{\bm{s},A_{n}}$ is an \textit{approximately $A_{2,n}$-dependent} locally stationary random field. Let $c_{1},\hdots, c_{p_{1}}$ be positive constants and $r_{1}(\cdot),\hdots, r_{p_{1}}(\cdot)$ be continuous functions on $[0,1]^{d}$ such that $|r_{k}(\bm{u}) - r_{k}(\bm{v})| \leq C\|\bm{u} - \bm{v}\|$, $1 \leq k \leq p_{1}$ with $C<\infty$. Consider the function
\begin{align}\label{exp-decay-kernel}
g(\bm{u},\|\bm{s}\|) &= \sum_{k=1}^{p_{1}}r_{k}\left(\bm{u}\right)e^{-c_{k}\|\bm{s}\|}.
\end{align}
We can demonstrate that a L\'evy-driven MA random field with the kernel function (\ref{exp-decay-kernel}) satisfies (\ref{m-dep-approx}); furthermore, locally stationary random fields include a wide class of CARMA-type random fields. For a discussion on general multivariate cases and examples that satisfy our regularity conditions, see also Appendix \ref{Multi-LevyMA} herein. 
\end{remark}

Next, we introduce the concept of approximately $m_{n}$-dependence for  locally stationary random fields. 
\begin{definition}[Approximately $m_{n}$-dependence for locally stationary random field]\label{approx-m-def}
Let $m_{n}$ be a sequence of positive constants with $m_{n} \to \infty$ as $n \to \infty$. We define a locally stationary random field $\bm{X}_{A_{n}} = \{\bm{X}_{\bm{s},A_{n}} = (X^{1}_{\bm{s},A_{n}},\hdots,X^{p}_{\bm{s},A_{n}})': \bm{s} \in R_{n}\}$ in $\mathbb{R}^{p}$ to be an approximately $m_{n}$-dependent random field if $\bm{X}_{A_{n}}$ can be represented as a sum of $m_{n}$-dependent random field $\{\bm{X}_{\bm{s},A_{n}:m_{n}}\}$ and the ``residual'' random field $\{\bm{\epsilon}_{\bm{s},A_{n}:m_{n}} = (\epsilon_{\bm{s},A_{n}:m_{n}}^{1},\hdots, \epsilon_{\bm{s},A_{n}:m_{n}}^{p})'\}$, which satisfies the following condition:  
\begin{enumerate}
\item[(Ma0)] For some $q>1$ such that ${q\zeta \over \zeta-1}$ is an even integer,
\begin{align*}
\max_{1 \leq j \leq n}\max_{1 \leq \ell \leq p}E_{\cdot|\bm{S}}\left[\left|\epsilon_{\bm{s}_{j},A_{n}:m_{n}}^{\ell}\right|^{{q\zeta \over \zeta-1}}\right]^{{\zeta-1 \over q\zeta}} \leq \gamma_{\epsilon}(m_{n}),\ \text{$P_{\bm{S}}$-a.s.}
\end{align*}
where $\zeta>2$ is the constant in Assumption \ref{Ass-U} and $\gamma_{\epsilon}(\cdot)$ is a decreasing function such that for some $\eta_{2} \in (0,1)$
\begin{align}\label{approx-decay}
{\gamma_{\epsilon}(m_{n})(n + A_{n}^{d}\log n) \over n^{(1-\eta_{2}) \over 2}h^{{d+p \over 2}+1}} \to 0\ \text{as $n \to \infty$}.
\end{align}
\end{enumerate}

\end{definition}

We can show that a univariate locally stationary CARMA-type random field with an exponential decay kernel such as (\ref{exp-decay-kernel}) and a L\'evy random measure that has finite moments up to the $q\zeta/(\zeta-1)$th moment, where $q\zeta/(\zeta-1)$ is an even integer, i.e., $E[|L([0,1]^{d})|^{q'}]<\infty$ for $1 \leq q' \leq q\zeta/(\zeta-1)$, satisfies Condition (Ma0) with
\begin{align}\label{CARMA-tail-decay}
\gamma_{\epsilon}(x) \leq C x^{{(d-1)(\zeta-1) \over q\zeta}}e^{-{c_{0}x \over 2}}
\end{align}
for a constant $C<\infty$, where $c_{0} = \min_{1 \leq \ell \leq p_{1}}c_{\ell}$. This implies (\ref{approx-decay}). In this case, the CARMA-type process is approximately $D(\log n)$-dependent for a sufficiently large $D>0$. See Lemma \ref{proof-CARMA-example} for the proof of (\ref{CARMA-tail-decay}). See also Appendix \ref{Multi-LevyMA} for a discussion on multivariate cases.

\begin{remark}
We introduced the notion of  approximately $m_{n}$-dependent random fields to give some examples of random fields that satisfy our assumptions. For this, it could be possible to adopt the $m$-dependence approximation technique in \cite{MaVoWu13} and \cite{Ma14}, for example. On the other hand, the proofs of our results are based on a general blocking technique designed for irregularly spaced sampling sites. Moreover, to derive the uniform convergence rate of the kernel estimator, we need to care about the effect of non-equidistant sampling sites when applying a maximal inequality, which requires additional work compared with the case that sampling sites are equidistant. See also Section \ref{beta-mixing RF discuss} for a discussion of mixing conditions. 
\end{remark}

Now we summarize our discussion in this subsection.
\begin{assumption}\label{Ass-M(add)}
\begin{enumerate}
\item[(Ma1)] The process $\{\bm{X}_{\bm{s},A_{n}}\}$ is an approximately $A_{2,n}$-dependent random field. Hence, $\bm{X}_{\bm{s},A_{n}}$ can be decomposed into a sum of $A_{2,n}$-dependent random fields $\{\bm{X}_{\bm{s},A_{n}:A_{2,n}}\}$ and residual random fields $\{\bm{\epsilon}_{\bm{s},A_{n}:A_{2,n}} \}$, which satisfies Condition (Ma0). 
\item[(Ma2)] The process $\{\bm{X}_{\bm{s},A_{n}}\}$ is an approximately $A_{2,n}$-dependent locally stationary random field. Therefore, for each space point $\bm{u} \in [0,1]^{d}$, an $A_{2,n}$-dependent strictly stationary random field $\{\bm{X}_{\bm{u}}(\bm{s}:A_{2,n})\}$ exists such that 
\[
\|\bm{X}_{\bm{s},A_{n}} - \bm{X}_{\bm{u}}(\bm{s}:A_{2,n})\| \leq \left(\left\|{\bm{s} \over A_{n}} - \bm{u}\right\| + {1 \over A_{n}^{d}}\right)U_{\bm{s}, A_{n}:A_{2,n}}(\bm{u})\ \text{a.s.}
\]
with $E[(U_{\bm{s}, A_{n}:A_{2,n}}(\bm{u}))^{\rho}]\leq C$ for some $\rho>0$.
\item[(Ma3)] (M2) and (M4) in Assumption \ref{Ass-M} hold by replacing $\bm{X}_{\bm{u}}(\bm{s})$ with $\bm{X}_{\bm{u}}(\bm{s}:A_{2,n})$.
\item[(Ma4)] (M3) and (M5) in Assumption \ref{Ass-M} hold.
\end{enumerate}
\end{assumption}

\begin{corollary}\label{col1}
Proposition \ref{general-unif-rate}, Theorems \ref{unif-rate-m}, \ref{general-CLT-m} holds even if we replace Assumption \ref{Ass-M} with Assumption \ref{Ass-M(add)}.
\end{corollary}

\begin{corollary}\label{col2}
Theorems \ref{unif-m-add} and \ref{CLT-m-add} hold even if we replace Assumption \ref{Ass-M} with Assumption \ref{Ass-M(add)} (with $p=1$ in (\ref{approx-decay})).
\end{corollary}

\begin{remark}
Let $\tilde{\beta}(a;b) = \tilde{\beta}_{1}(a)\tilde{g}_{1}(b)$ denote the $\beta$-mixing coefficients of $\bm{X}_{\bm{s},A_{n}:A_{2,n}}$. When the process $\{\bm{X}_{\bm{s},A_{n}}\}$ is an approximately $A_{2,n}$-dependent random field and an approximately $A_{2,n}$-dependent locally stationary random field, we can replace $\bm{X}_{\bm{s},A_{n}}$ with $\bm{X}_{\bm{s},A_{n}:A_{2,n}}$ in our analysis. In this case, because the $\beta$-mixing coefficients of $\bm{X}_{\bm{s},A_{n}:A_{2,n}}$ satisfy $\tilde{\beta}(A_{2,n};A_{n}^{d}) = \tilde{\beta}_{1}(A_{2,n})\tilde{g}_{1}(A_{n}^{d}) = 0$, conditions on $\beta$-mixing coefficients (R1), (Ra3), (Rb1) are satisfied automatically. See also Remark \ref{A2-depend-approx} in the Appendix. 
\end{remark}

\cite{BrMa17} developed CARMA random fields and they discuss in the introduction of their paper that even for CARMA random fields, which is a special class of L\'evy-driven moving average random fields, the random fields are already generates a much rich class of random fields. Therefore, L\'evy-driven moving average random fields are one of the most flexible and concrete class of random fields from both theoretical and practical point of view. Moreover, it is difficult to compute $\alpha$- or $\beta$-mixing coefficients for non-Gaussian random fields on $\mathbb{R}^{d}$ and there would be no general tools for the computation. The dependence structure of CARMA random fields on $\mathbb{R}^{d}$ have not been studied although there are some results on $\beta$-mixing properties of multivariate CARMA processes on $\mathbb{R}$ in \cite{ScSt12} for example. Indeed, existing papers on statistical analysis of CARMA random fields on $\mathbb{R}^{d}$ (cf. \cite{BrMa17} and \cite{MaYu20}) consider parametric estimation of second-order stationary CARMA random fields but their ($\alpha$- or $\beta$-) mixing properties are not addressed. Hence it is still an open question that CARMA random fields are $\beta$- (or $\alpha$-) mixing.   Then, our results in Section 5 are important contributions in the sense that our results in Sections 3 and 4 can be applied to L\'evy-driven moving average random fields that is a more flexible class of random fields than CARMA random fields by providing sufficient conditions for $m_{n}$-dependent approximation of locally stationary random fields and showing that locally stationary CARMA-type random fields can be approximated by $m_{n}$-dependent random fields.

\section{Discussion}
In this section we discuss some directions of possible extensions of our results and the model (\ref{NR-LSRF}) for applications to spatio-temporal data.

\subsection{Extensions}\label{extension discuss}

\subsubsection{On sampling region}

It is possible to extend the definition of the sampling region $R_{n}$ to a more general case that includes nonstandard shapes. For example, we can adopt the definition of sampling regions in \cite{La03b} as follows: First, we define the sampling region $R_{n}$. Let $R_{0}^{\ast} $ be an open connected subset of $(-2, 2]^{d}$ containing $[-1,1]^{d}$ and let $R_{0}$ be a Borel set satisfying $R_{0}^{\ast} \subset R_{0} \subset \overline{R}_{0}^{\ast}$, where for any set $S \subset \R^{d}$, $\overline{S}$ denotes its closure. Let $\{A_{n}\}_{n \geq 1}$ be a sequence of positive numbers such that $A_{n} \to \infty$ as $n \to \infty$ and consider $R_{n} = A_{n}R_{0}$ as a sampling region. We also assume that for any sequence of positive numbers $\{a_{n}\}_{n \geq 1}$ with $a_{n} \to 0$ as $n \to \infty$, the number of cubes of the form $a_{n}(\bm{i} +[0,1)^{d})$, $\bm{i} \in \mathbb{Z}^{d}$ with their lower left corner $a_{n}\bm{i}$ on the lattice $a_{n}\mathbb{Z}^{d}$ that intersect both $R_{0}$ and $R_{0}^{c}$ is $O(a_{n}^{-d+1})$ as $n \to \infty$. Moreover, let $f$ be a continuous, everywhere positive probability density function on $R_{0}$, and let $\{\bS_{0,i}\}_{i \geq 1}$ be a sequence of i.i.d. random vectors with density $f$. Assume that $\{\bS_{0,i}\}_{i \geq 1}$ and $\bm{X}_{\bm{s},A_{n}}$ are independent. 

The boundary condition on the prototype set $R_{0}$ holds in many practical situations, including many convex subsets in $\R^{d}$, such as spheres, ellipsoids, polyhedrons, as well as many nonconvex sets in $\R^{d}$. See also \cite{La03b} and Chapter 12 in \cite{La03a} for further discussion on the boundary condition. 

Under this setting on the sampling region, our results hold under the same assumptions and the proofs are the same. Furthermore, for any set $S \subset \mathbb{R}^{d}$, let $\text{int}(S) = S \backslash \overline{S}$ be the internal part of $S$ and for any $\delta>0$, let $S^{\delta} = \{\bm{s} \in \mathbb{R}^{d}: d(\bm{s}, S) \leq \delta\}$ where $d(\bm{s}, S) = \inf_{\bm{x} \in S}\|\bm{s} - \bm{x}\|$. Then by replacing $[0,1]^{d}$ and $(0,1)^{d}$ that appear in Assumptions \ref{Ass-S} and \ref{Ass-M} with $R_{0}$ and $\text{int}(R_{0})$, it is also possible to show uniform convergence results that correspond to Proposition \ref{general-unif-rate}, Theorem \ref{unif-rate-m} and Theorem \ref{unif-m-add} over $ (\bm{u}, \bm{x}) \in V_{0} \times S_{c}$, $V_{0,h} \times S_{c}$ and $V_{0,2h} \times I_{h,0}$, respectively. Here, $S_{c}$ is a compact subset of $\mathbb{R}^{p}$, and $V_{0,h}$, $V_{0,2h}$ and $V_{0}$ are compact subsets of $R_{0}$ such that $V_{0,2h} \subset V_{0,h}$ and $V_{0,2h}^{2C_{1}h} = V_{0,h}^{C_{1}h} = V_{0}$. We omit the proofs since they are almost the same as those of Proposition \ref{general-unif-rate}, Theorem \ref{unif-rate-m} and Theorem \ref{unif-m-add}.

\subsubsection{Constructing confidence bands for regression functions}

As an extension of Theorem \ref{general-CLT-m}, it is straightforward to show joint asymptotic normality of $\hat{m}$ over finite number of design points $\{(\bm{u}_{j}, \bm{x}_{j})\}_{j=1}^{L}$ and verify that $\hat{m}(\bm{u}_{j}, \bm{x}_{j})$ are asymptotically independent. Building on the result, it should be possible to construct simple confidence bands by plug-in methods and linear interpolations of  the following joint confidence intervals: Let $\xi_{1},\hdots, \xi_{L}$ be i.i.d. standard normal random variables, and let $q_{\tau}$ satisfy $P\left(\max_{1 \leq j \leq L}|\xi_{j}| > q_{\tau}\right) = \tau$ for $\tau \in (0,1)$ and take a bandwidth $h$ such that $h^{2} \ll 1/\sqrt{nh^{d+p}}$. 
Then, 
\[
C_{j}(1-\tau) = \left[\hat{m}(\bm{u}_{j}, \bm{x}_{j}) \pm \sqrt{{V_{\bm{u}_{j}, \bm{x}_{j}} \over nh^{d+p}}}q_{\tau}\right],\ j = 1,\hdots, L
\] 
are joint asymptotic $100(1-\tau)$\% confidence intervals of $m$. Here, $[a \pm b] = [a-b,a+b]$ for $a \in \R$ and $b > 0$. More generally, there could be two possible ways to construct confidence bands of the regression function. The first way is based on a Gumbel approximation as considered in \cite{ZhWu08}. The second way is based on intermediate (high-dimensional) Gaussian approximations as considered in \cite{HoLe12}. However, we believe that both approaches require additional substantial work and there would be no previous studies on the construction of uniform confidence bands for locally stationary time series or random fields. Therefore, we only explain the idea of both approaches formally to keep the tight focus of the present paper and leave the extension as a future research topic.  The idea of both approaches is as follows. Instead of constructing uniform confidence bands of the regression function over $I_{h}$ and $S_{c}$, we first consider a discretized version of $I_{h}$ and $S_{c}$ in Theorem \ref{unif-rate-m} such that $\{\bm{u}_{n,j}\}_{j=1}^{N_u}$, $\{\bm{x}_{n,k}\}_{k=1}^{N_x}$ where $\bm{u}_{n,j} \in I_{h}$ and $\bm{x}_{n,k} \in S_{c}$ and the design points $\{\bm{u}_{n,j}\}_{j=1}^{N_u}$, $\{\bm{x}_{n,k}\}_{k=1}^{N_x}$ are asymptotically dense in $I_{h}$ and $S_{c}$, i.e. $N_u, N_x \to \infty$, $\max_{1 \leq j_{1} \neq  j_{2} \leq N_u,}\|\bm{u}_{n,j_{1}} - \bm{u}_{n,j_{2}}\| \to 0$ and $\max_{1 \leq k_{1} \neq  k_{2} \leq N_x,}\|\bm{x}_{n,k_{1}} - \bm{x}_{n,k_{2}}\| \to 0$ as $n \to \infty$. Then we could construct ``asymptotically'' uniform confidence bands of the regression function by linear interpolations of joint confidence intervals over the design points combining Gumbel or intermediate Gaussian approximations.

\subsection{Further discussion on mixing conditions}\label{more discuss mixing}

Here we discuss further on the mixing conditions in this paper. See also Section \ref{beta-mixing RF discuss} for discussion on mixing conditions from another technical aspects. Consider the following decomposition of a random field $\bm{X}_{\bm{s},A_{n}}$:
\begin{align}\label{approx-beta-decomp-X-ep}
\bm{X}_{\bm{s}, A_{n}} = \tilde{\bm{X}}_{\bm{s}, A_{n}} + \tilde{\bm{\epsilon}}_{\bm{s}, A_{n}},
\end{align}
where $\tilde{\bm{X}}_{\bm{s}, A_{n}}$ is a locally stationary $\beta$-mixing (for uniform estimation) or $\alpha$-mixing (for asymptotic normality) random field that satisfies Assumptions \ref{Ass-M}, \ref{Ass-U(add)} and \ref{Ass-R(add)} (for general case) or Assumptions \ref{Ass-M}, \ref{Ass-U(add)} and \ref{Ass-Rb} (for additive models), and $\tilde{\bm{\epsilon}}_{\bm{s}, A_{n}} = (\tilde{\epsilon}^{1}_{\bm{s},A_{n}},\hdots,\tilde{\epsilon}^{p}_{\bm{s},A_{n}})'$ is a residual random field which is asymptotically negligible. The decomposition (\ref{approx-beta-decomp-X-ep}) implies that $\tilde{\bm{X}}_{\bm{s}, A_{n}}$ is not necessarily $m_{n}$-dependent. A similar argument in Remark \ref{A2-depend-approx} in Appendix \ref{Appendix-proof} yields that to show the asymptotic negligibility of $\tilde{\bm{\epsilon}}_{\bm{s}, A_{n}}$ for our results to hold, it is sufficient to verify the following condition: For some $q>1$ such that ${q\zeta \over \zeta-1}$ is an even integer,
\begin{align*}
\max_{1 \leq j \leq n}\max_{1 \leq \ell \leq p}E_{\cdot|\bm{S}}\left[\left|\tilde{\epsilon}_{\bm{s}_{j},A_{n}}^{\ell}\right|^{{q\zeta \over \zeta-1}}\right]^{{\zeta-1 \over q\zeta}} \leq \iota_{\tilde{\epsilon}}(n),\ \text{$P_{\bm{S}}$-a.s.}
\end{align*}
where $\zeta>2$ is the constant in Assumption \ref{Ass-U} and $\iota_{\tilde{\epsilon}}(\cdot)$ is a decreasing function such that for some $\eta_{2} \in (0,1)$
\begin{align}\label{approx-decay-Xe-beta}
{\iota_{\tilde{\epsilon}}(n)(n + A_{n}^{d}\log n) \over n^{(1-\eta_{2}) \over 2}h^{{d+p \over 2}+1}} \to 0\ \text{as $n \to \infty$}.
\end{align}
As far as (\ref{approx-decay-Xe-beta}) is satisfied, the random field $\bm{X}_{\bm{s}, A_{n}}$ itself is not necessarily $\beta$-mixing (or $\alpha$-mixing) in general and $\tilde{\bm{\epsilon}}_{\bm{s}, A_{n}}$ can accommodate complex dependence structures of $\bm{X}_{\bm{s}, A_{n}}$ that cannot be captured by $\beta$-mixing (or $\alpha$-mixing) random fields. As a special case, $\tilde{\bm{X}}_{\bm{s}, A_{n}}$ can be an approximately $m_{n}$-dependent locally stationary random field and in this case $\tilde{\bm{X}}_{\bm{s}, A_{n}}$ also can be decomposed into $m_{n}$-dependent random field $\tilde{\bm{X}}_{\bm{s}, A_{n}:m_{n}}$ and a residual random field $\tilde{\bm{\epsilon}}_{\bm{s}, A_{n}:m_{n}}$ that corresponds to the approximation error. Then $\bm{X}_{\bm{s}, A_{n}}$ can be represented as follows:
\[
\bm{X}_{\bm{s}, A_{n}} = \underbrace{\tilde{\bm{X}}_{\bm{s}, A_{n}:m_{n}}}_{\text{$m_{n}$-dependent}} + \underbrace{\tilde{\bm{\epsilon}}_{\bm{s}, A_{n}:m_{n}}}_{\text{approximation error}} + \underbrace{\tilde{\bm{\epsilon}}_{\bm{s}, A_{n}}}_{\text{complexly dependent}}.
\]  
The L\'evy-driven moving average random fields considered in Section 5 of the revised manuscript are examples when $\tilde{\bm{\epsilon}}_{\bm{s}, A_{n}} = \bm{0}$.

\subsection{Spatio-temporal modeling} 

Consider a spatially locally stationary spatio-temporal data $\{\bm{X}_{\bm{s}_{j},A_{n}}: \bm{s}_{j} \in R_{n}, \}$ where $\bm{X}_{\bm{s}_{j},A_{n}}=(X_{\bm{s}_{j},A_{n}}(t_{1}), \cdots, X_{\bm{s}_{j},A_{n}}(t_{p+1}))': \bm{s} \in R_{n}\}$, that is, at each location $\bm{s}_{j}$, $\bm{X}_{\bm{s}_{j},A_{n}}$ is a time series observed at time points $t_{1}<\hdots<t_{p+1}$ with $p \geq 1$ fixed. Our results can be applied to the following spatio-temporal model:
\begin{align*}
X_{\bm{s}_{j}, A_{n}}(t_{p+1}) &= m\left({\bm{s}_{j} \over A_{n}}, X_{\bm{s}_{j},A_{n}}(t_{1}), \cdots, X_{\bm{s}_{j},A_{n}}(t_{p})\right) + \epsilon_{\bm{s}_{j},A_{n}},\ j=1,\hdots, n.
\end{align*}
In our framework, the time points can be non-equidistant and the process can be temporally nonstationary since we consider finite number of time points. Additionally, our sampling design is related to that of \cite{AlRe19}. They study nonparametric estimation of a linear SPDE with a spatially-varying coefficient under the asymptotic regime with fixed time horizon and with the spatial resolution of the observations tending to zero. 

\subsection{More discussion on locally stationary L\'evy-driven MA random fields}

Our argument to show (\ref{m-dep-approx})-(\ref{U-moment}) is limited to non-Gaussian (or pure jump) MA random fields. However, the argument can be extended to the cases that allow the existence of non-trivial Gaussian part of the random measure $L$ as follows: Consider the following  L\'evy-driven MA random fields with possibly non-vanishing Gaussian part.
\begin{align*}
X_{\bm{s},A_{n}} &= \int_{\mathbb{R}^{d}}g\left({\bm{s} \over A_{n}}, \|\bm{s} - \bm{v}\|\right)L(d\bm{v})\ \text{and}\ X_{\bm{u}}(\bm{s}) = \int_{\mathbb{R}^{d}}g\left(\bm{u}, \|\bm{s} - \bm{v}\|\right)L(d\bm{v}),
\end{align*}
where $g: [0,1]^{d}\times [0,\infty) \to \mathbb{R}$ is a bounded function such that $|g(\bm{u},\cdot) - g(\bm{v},\cdot)| \leq C\|\bm{u} - \bm{v}\|\bar{g}(\cdot)$ with $C<\infty$ and for any $\bm{u} \in [0,1]^{d}$,
\[
\int_{\mathbb{R}^{d}}(|g(\bm{u},\|\bm{s}\|)| + |\bar{g}(\bm{s})|)d\bm{s}<\infty\ \text{and}\ \int_{\mathbb{R}^{d}}(g^{2}(\bm{u},\|\bm{s}\|) +  \bar{g}^{2}(\bm{s}))d\bm{s}<\infty.
\]
If $E[|L(A)|^{q'}]<\infty$ for any $A \in \mathcal{B}(\mathbb{R}^{d})$ with a bounded Lebesgue measure $|A|$ and for $q' =1,2$, we have that
\[
E[X_{\bm{u}}(\bm{s})] = \mu_{0}\int_{\mathbb{R}^{d}}g\left(\bm{u}, \|\bm{s}\|\right)d\bm{s}\ \text{and}\ E\left[(\bm{X}_{\bm{u}}(\bm{s}))^{2}\right] = \bar{\sigma}_{0}^{2}\int_{\mathbb{R}^{d}}g^{2}\left(\bm{u}, \|\bm{s}\|\right)d\bm{s},
\]
where $\mu_{0}\in \mathbb{R}$ and $\bar{\sigma}_{0}^{2} >0$. Let $m >0$. Define the function $\iota(\cdot : m) : [0,\infty) \to [0,1]$ as 
\[
\iota(x:m) = 
\begin{cases}
1 & \text{if $0\leq x \leq {m \over 2}$},\\
-{2 \over m}x+2 & \text{if ${m \over 2} < x \leq m$},\\
0 & \text{if $m<x$}
\end{cases}
\]
and consider the process $X_{\bm{u}}(\bm{s}:A_{2,n}) = \int_{\mathbb{R}^{d}}g\left(\bm{u}, \|\bm{s} - \bm{v}\|\right)\iota(\|\bm{s} - \bm{v}\|:A_{2,n})L(d\bm{v})$. 

Assume that $\|\bm{s}/A_n - \bm{u}\|>0$. Observe that 
\begin{align*}
\left|X_{\bm{s},A_{n}} - X_{\bm{u}}(\bm{s}:A_{2,n})\right| &\leq \left|X_{\bm{s},A_{n}} - X_{\bm{u}}(\bm{s})\right| + \left|X_{\bm{u}}(\bm{s}) - X_{\bm{u}}(\bm{s}:A_{2,n})\right| \nonumber \\
&= \left|\int_{\mathbb{R}^{d}}g\left({\bm{s} \over A_{n}}, \|\bm{s} - \bm{v}\|\right) - g\left(\bm{u}, \|\bm{s} - \bm{v}\|\right)L(d\bm{v})\right| \nonumber \\
&\quad + {1 \over A_{n}^{d}}\left|\underbrace{\int_{\mathbb{R}^{d}}A_{n}^{d}\overbrace{g\left(\bm{u}, \|\bm{s} - \bm{v}\|\right)(1 - \iota(\|\bm{s} - \bm{v}\|:A_{2,n}))}^{=: g_{\bm{u}}(\|\bm{s}-\bm{v}\|:A_{2,n})}L(d\bm{v})}_{=: U_{2,\bm{s},A_n}(\bm{u})}\right| \nonumber \\
&= \left\|{\bm{s} \over A_{n}} - \bm{u}\right\|\left|\underbrace{\int_{\mathbb{R}^{d}}{g\left({\bm{s} \over A_{n}}, \|\bm{s} - \bm{v}\|\right) - g\left(\bm{u}, \|\bm{s} - \bm{v}\|\right) \over \left\|{\bm{s} \over A_{n}} - \bm{u}\right\|}L(d\bm{v})}_{=:U_{1, \bm{s}, A_n}(\bm{u})}\right| + {1 \over A_n^d}|U_{2,\bm{s},A_n}(\bm{u})| \nonumber \\
&\leq \left(\left\|{\bm{s} \over A_{n}} - \bm{u}\right\| \! + \! {1 \over A_{n}^{d}}\right) \left\{\left|U_{1,\bm{s},A_n}(\bm{u})\right|   +\left|U_{2,\bm{s},A_n}(\bm{u}) \right| \right\}\nonumber \\
&=: \left(\left\|{\bm{s} \over A_{n}} - \bm{u}\right\| + {1 \over A_{n}^{d}}\right)U_{\bm{s},A_{n}}(\bm{u}). 
\end{align*}
Note that 
\begin{align*}
E\left[U_{1,\bm{s},A_n}(\bm{u})^2\right] &= \bar{\sigma}_0^2\int_{\mathbb{R}^{d}}\left({g\left({\bm{s} \over A_{n}}, \|\bm{z}\|\right) - g\left(\bm{u}, \|\bm{z}\|\right) \over \left\|{\bm{s} \over A_{n}} - \bm{u}\right\|}\right)^2d\bm{z} \leq C^2\bar{\sigma}^2_0\int_{\mathbb{R}^{d}}\bar{g}^2(\|\bm{z}\|)d\bm{z},\\
E\left[U_{2,\bm{s},A_n}(\bm{u})^2\right] &= \int_{\mathbb{R}^{d}}A_{n}^{2d}g_{\bm{u}}^2(\|\bm{z}\|:A_{2,n})d\bm{z}
\end{align*}
If we assume that
\begin{align*}
\sup_{n \geq 1}\sup_{\bm{u} \in [0,1]^{d}}\int_{\mathbb{R}^{d}}(A_{n}^{d}g_{\bm{u}}(\|\bm{s}\|:A_{2,n}) + A_{n}^{2d}g_{\bm{u}}^{2}(\|\bm{s}\|:A_{2,n}))d\bm{s}<\infty,
\end{align*}
then we obtain
\begin{align*}
E[U_{\bm{s}, A_{n}}^{2}(\bm{u})] &\leq  C\sup_{n \geq 1}\sup_{\bm{u} \in [0,1]^{d}}\!\left(\int_{\mathbb{R}^{d}}\!g_{\bm{u},A_{n},A_{2,n}}^{2}(\|\bm{s}\|)d\bm{s} + \left(\int_{\mathbb{R}^{d}}\!g_{\bm{u},A_{n},A_{2,n}}(\|\bm{s}\|)d\bm{s}\right)^{2}\right)\!<\infty.
\end{align*}

\newpage

\appendix

\section{Proofs}\label{Appendix-proof}

\subsection{Proofs for Section \ref{Section_Main}}

\begin{proof}[Proof of Proposition \ref{general-unif-rate}]
Define $B = \{(\bm{u},\bm{x}) \in \mathbb{R}^{d+p}: \bm{u} \in [0,1]^{d}, \bm{x} \in S_{c}\}$, $a_{n} = \sqrt{\log n/nh^{d+p}}$ and $\tau_{n} = \rho_{n}n^{1/\zeta}$ with $\rho_{n} = (\log n)^{\zeta_{0}}$ for some $\zeta_{0} >0$. Define 
\begin{align*}
\hat{\psi}_{1}(\bm{u},\bm{x}) &= {1 \over nh^{p+d}}\sum_{j=1}^{n}\bar{K}_{h}\left(\bm{u} - {\bm{s}_{j} \over A_{n}}\right)\prod_{\ell_{2}}^{p}K_{h}\left(x_{\ell_{2}} - X_{\bm{s}_{j},A_{n}}^{\ell_{2}}\right)W_{\bm{s}_{j},A_{n}}I(|W_{\bm{s}_{j}, A_{n}}| \leq \tau_{n}),\\
\hat{\psi}_{2}(\bm{u},\bm{x}) &= {1 \over nh^{p+d}}\sum_{j=1}^{n}\bar{K}_{h}\left(\bm{u} - {\bm{s}_{j} \over A_{n}}\right)\prod_{\ell_{2}}^{p}K_{h}\left(x_{\ell_{2}} - X_{\bm{s}_{j},A_{n}}^{\ell_{2}}\right)W_{\bm{s}_{j},A_{n}}I(|W_{\bm{s}_{j}, A_{n}}| > \tau_{n}).
\end{align*}
Note that $\hat{\psi}(\bm{u},\bm{x}) - E_{\cdot|\bm{S}}[\hat{\psi}(\bm{u},\bm{x})] = \hat{\psi}_{1}(\bm{u},\bm{x}) - E_{\cdot|\bm{S}}[\hat{\psi}_{1}(\bm{u},\bm{x})] + \hat{\psi}_{2}(\bm{u},\bm{x}) - E_{\cdot|\bm{S}}[\hat{\psi}_{2}(\bm{u},\bm{x})]$.

(Step1) First we consider the term $\hat{\psi}_{2}(\bm{u},\bm{x}) - E_{\cdot|\bm{S}}[\hat{\psi}_{2}(\bm{u},\bm{x})]$. 
\begin{align*}
P_{\cdot|\bm{S}}\left(\sup_{(\bm{u},\bm{x}) \in B}|\hat{\psi}_{2}(\bm{u},\bm{x})|>Ca_{n}\right) &\leq P_{\cdot|\bm{s}}\left(|W_{\bm{s}_{j}, A_{n}}| > \tau_{n}\ \text{for some $j=1,\hdots,n$}\right)\\
&\leq \tau_{n}^{-\zeta}\sum_{j=1}^{n}E_{\cdot|\bm{S}}[|W_{\bm{s}_{j}, A_{n}}|^{\zeta}] \leq Cn\tau_{n}^{-\zeta} = \rho_{n}^{-\zeta} \to 0\ P_{\bm{S}}-a.s.
\end{align*}

\begin{align*}
E_{\cdot|\bm{S}}\left[|\hat{\psi}_{2}(\bm{u},\bm{x})|\right] &\leq {1 \over nh^{d+p}}\sum_{j=1}^{n}\bar{K}_{h}\left(\bm{u} - {\bm{s}_{j} \over A_{n}}\right)\int_{\mathbb{R}^{d}}\prod_{\ell = 1}^{p}K_{h}(x_{\ell} - w_{\ell})\\
&\quad \times E_{\cdot|\bm{S}}[|W_{\bm{s}_{j}, A_{n}}|I(|W_{\bm{s}_{j}, A_{n}}|>\tau_{n})|X_{\bm{s}_{j},A_{n}} = \bm{w}]f_{\bm{X}_{\bm{s}_{j},A_{n}}}(\bm{w})d\bm{w}\\
&= {1 \over nh^{d}}\sum_{j=1}^{n}\bar{K}_{h}\left(\bm{u} - {\bm{s}_{j} \over A_{n}}\right)\int_{\mathbb{R}^{d}}\prod_{\ell = 1}^{p}K(\varphi_{\ell})\\
&\quad \times E_{\cdot|\bm{S}}[|W_{\bm{s}_{j}, A_{n}}|I(|W_{\bm{s}_{j}, A_{n}}|>\tau_{n})|X_{\bm{s}_{j},A_{n}} = \bm{x} - h\bm{\varphi}]f_{\bm{X}_{\bm{s}_{j},A_{n}}}(\bm{x} - h\bm{\varphi})d\bm{\varphi}\\
&= {1 \over nh^{d}}\sum_{j=1}^{n}\bar{K}_{h}\left(\bm{u} - {\bm{s}_{j} \over A_{n}}\right){1 \over \tau_{n}^{\zeta-1}}\int_{\mathbb{R}^{d}}\prod_{\ell = 1}^{p}K(\varphi_{\ell})\\
&\quad \times E_{\cdot|\bm{S}}[|W_{\bm{s}_{j}, A_{n}}|^{\zeta}I(|W_{\bm{s}_{j}, A_{n}}|>\tau_{n})|X_{\bm{s}_{j},A_{n}} = \bm{x} - h\bm{\varphi}]f_{\bm{X}_{\bm{s}_{j},A_{n}}}(\bm{x} - h\bm{\varphi})d\bm{\varphi}\\
&\leq {C \over \tau_{n}^{\zeta-1}}{1 \over nh^{d}}\sum_{j=1}^{n}\bar{K}_{h}\left(\bm{u} - {\bm{s}_{j} \over A_{n}}\right)\\
&= {C \over \tau_{n}^{\zeta-1}}\left(f_{\bm{S}}(\bm{u}) + O\left({\sqrt{\log n \over nh^{d}}} + h^{2}\right)\right) \leq {C \over \tau_{n}^{\zeta-1}} = C\rho_{n}^{-(\zeta-1)}n^{-(\zeta-1)/\zeta} \leq Ca_{n}\ P_{\bm{S}}-a.s.  
\end{align*}
In the last equation, we used Lemma \ref{Masry-thm}. As a result, 
\[
\sup_{(\bm{u},\bm{x}) \in B}|\hat{\psi}_{2}(\bm{u},\bm{x}) - E_{\cdot|\bm{S}}[\hat{\psi}_{2}(\bm{u},\bm{x})]| = O_{P_{\cdot|\bm{S}}}(a_{n}). 
\]

(Step2) Now we consider the term $\hat{\psi}_{1}(\bm{u},\bm{x}) - E_{\cdot|\bm{S}}[\hat{\psi}_{1}(\bm{u},\bm{x})]$. First we introduce some notations. For $\bm{\ell} = (\ell_{1}, \hdots, \ell_{d})' \in \mathbb{Z}^{d}$, let $\Gamma_{n}(\bm{\ell};\bm{0}) = (\bm{\ell} + (0,1]^{d})A_{3,n}$ where $A_{3,n} = A_{1,n} + A_{2,n}$ and divide $\Gamma_{n}(\bm{\ell};\bm{0})$ into $2^{d}$ hypercubes as follows: 
\[
\Gamma_{n}(\bm{\ell};\bm{\epsilon}) = \prod_{j=1}^{d}I_{j}(\epsilon_{j}),\ \bm{\epsilon} = (\epsilon_{1},\hdots, \epsilon_{d})' \in \{1,2\}^{d}, 
\]
where, for $1 \leq j \leq d$, 
\[
I_{j}(\epsilon_{j}) = 
\begin{cases}
(\ell_{j}\lambda_{3,n}, \ell_{j}\lambda_{3,n} + \lambda_{1,n}] & \text{if $\epsilon_{1} = 1$}, \\
(\ell_{j}\lambda_{3,n} + \lambda_{1,n}, (\ell_{j}+1)\lambda_{3,n}] & \text{if $\epsilon_{1} = 2$}.
\end{cases}
\]
Note that 
\begin{align}\label{partition volume}
|\Gamma_{n}(\bm{\ell};\bm{\epsilon})| = \lambda_{1,n}^{q(\bm{\epsilon})}\lambda_{2,n}^{d-q(\bm{\epsilon})}
\end{align}
for any $\bm{\ell} \in \mathbb{Z}^{d}$ and $\bm{\epsilon} \in \{1,2\}^{d}$, where $q(\bm{\epsilon}) = [\![\{1 \leq j \leq d: \epsilon_{j} = 1\}]\!]$.

Let $L_{n} = \{\bm{\ell} \in \mathbb{Z}^{d}: \Gamma_{n}(\bm{\ell},\bm{0}) \cap R_{n} \neq \emptyset\}$ denote the index set of all hypercubes $ \Gamma_{n}(\bm{\ell},\bm{0})$ that are contained in or boundary of $R_{n}$. Moreover, let $L_{1,n} = \{\bm{\ell} \in \mathbb{Z}^{d}: \Gamma_{n}(\bm{\ell},\bm{0}) \subset R_{n}\}$ denote the index set of all hypercubes $ \Gamma_{n}(\bm{\ell},\bm{0})$ that are contained in $R_{n}$, and let $L_{2,n} = \{\bm{\ell} \in \mathbb{Z}^{d}:  \Gamma_{n}(\bm{\ell},\bm{0}) \cap R_{n} \neq 0,  \Gamma_{n}(\bm{\ell},\bm{0}) \cap R_{n}^{c} \neq \emptyset \}$ be the index set of boundary hypercubes.

\begin{remark}
Let $\bm{\epsilon}_{0} = (1,\hdots, 1)'$. The partitions $\Gamma_{n}(\bm{\ell};\bm{\epsilon}_{0})$ correspond to ``big blocks'' and the partitions $\Gamma(\bm{\ell};\bm{\epsilon})$ for $\bm{\epsilon} \neq \bm{\epsilon}_{0}$ correspond to ``small blocks''. 
\end{remark}

Define 
\begin{align*}
Z'_{\bm{s},A_{n}}(\bm{u},\bm{x}) &= \bar{K}_{h}\left(\bm{u} - {\bm{s} \over A_{n}}\right)\left\{ \prod_{\ell=1}^{p}K_{h}\left(x_{\ell} - X_{\bm{s},A_{n}}^{\ell}\right)W_{\bm{s}_{j},A_{n}}I(|W_{\bm{s}, A_{n}}| \leq \tau_{n})\right. \\
&\left. \quad - E_{\cdot|\bm{S}}\left[\prod_{\ell=1}^{p}K_{h}\left(x_{\ell} - X_{\bm{s}_{j},A_{n}}^{\ell}\right)W_{\bm{s}_{j},A_{n}}I(|W_{\bm{s}_{j}, A_{n}}| \leq \tau_{n})\right] \right\}.
\end{align*}
Observe that
\begin{align}
&\sum_{j=1}^{n}Z'_{\bm{s}_{j},A_{n}}(\bm{u},\bm{x}) \nonumber  \\
&= \underbrace{\sum_{\ell \in L_{1,n}\cup L_{2,n}}Z_{A_{n}}^{'(\bm{\ell};\bm{\epsilon}_{0})}(\bm{u},\bm{x})}_{\text{big blocks}} + \underbrace{\sum_{\bm{\epsilon} \neq \bm{\epsilon}_{0}}\sum_{\ell \in L_{1,n}}Z_{A_{n}}^{'(\bm{\ell};\bm{\epsilon})}(\bm{u},\bm{x}) + \sum_{\bm{\epsilon} \neq \bm{\epsilon}_{0}}\sum_{\ell \in L_{2,n}}Z_{A_{n}}^{'(\bm{\ell};\bm{\epsilon})}(\bm{u},\bm{x})}_{\text{small blocks}}, \label{decomp1} 
\end{align}
where $Z^{'(\bm{\ell};\bm{\epsilon})}_{A_{n}}(\bm{u}, \bm{x}) = \sum_{j: \bm{s}_{j} \in \Gamma_{n}(\bm{\ell};\bm{\epsilon})}Z'_{\bm{s},A_{n}}(\bm{u},\bm{x})$. Let $\{\tilde{Z}^{'(\bm{\ell};\bm{\epsilon})}_{A_{n}}(\bm{u}, \bm{x})\}_{\bm{\ell} \in L_{1,n} \cup L_{2,n}}$ be independent random variables such that $\tilde{Z}^{'(\bm{\ell};\bm{\epsilon})}_{A_{n}}(\bm{u}, \bm{x}) \stackrel{d}{=}Z^{'(\bm{\ell};\bm{\epsilon})}_{A_{n}}(\bm{u}, \bm{x})$. Note that for distinct $\bm{\ell}_{1}$ and $\bm{\ell}_{2}$, $d(\Gamma_{n}(\bm{\ell}_{1};\bm{\epsilon}), \Gamma_{n}(\bm{\ell}_{2};\bm{\epsilon})) \geq A_{2,n}$. Applying Corollary 2.7 in \cite{Yu94} (Lemma \ref{indep_lemma} in this paper) with $m = \left({A_{n} \over A_{1,n}}\right)^{d}$ and $\beta(Q) = \beta(A_{2,n};A_{n}^{d})$, we have that 
\begin{align*}
&\sup_{t >0}\left|P_{\cdot|\bm{S}}\left(\left|\sum_{\bm{\ell} \in L_{1,n} \cup L_{2,n}}\!\!\!\!\!\!\! Z^{'(\bm{\ell};\bm{\epsilon}_{0})}_{A_{n}}(\bm{u}, \bm{x})\right|>t\right) - P_{\cdot|\bm{S}}\left(\left|\sum_{\bm{\ell} \in L_{1,n} \cup L_{2,n}}\!\!\!\!\!\!\! \tilde{Z}^{'(\bm{\ell};\bm{\epsilon}_{0})}_{A_{n}}(\bm{u}, \bm{x})\right|>t\right)\right|\\ 
&\quad \leq C\left({A_{n} \over A_{1,n}}\right)^{d}\!\!\! \beta(A_{2,n};A_{n}^{d})\ P_{\bm{S}}-a.s.
\end{align*} 
\begin{align*}
&\sup_{t >0}\left|P_{\cdot|\bm{S}}\left(\left|\sum_{\bm{\ell} \in L_{1,n} }Z^{'(\bm{\ell};\bm{\epsilon})}_{A_{n}}(\bm{u}, \bm{x})\right|>t\right) - P_{\cdot|\bm{S}}\left(\left|\sum_{\bm{\ell} \in L_{1,n}}\tilde{Z}^{'(\bm{\ell};\bm{\epsilon})}_{A_{n}}(\bm{u}, \bm{x})\right|>t\right)\right|\\ &\quad \leq C\left({A_{n} \over A_{1,n}}\right)^{d}\beta(A_{2,n};A_{n}^{d})\ P_{\bm{S}}-a.s.
\end{align*} 
\begin{align*}
&\sup_{t >0}\left|P_{\cdot|\bm{S}}\left(\left|\sum_{\bm{\ell} \in L_{2,n} }Z^{'(\bm{\ell};\bm{\epsilon})}_{A_{n}}(\bm{u}, \bm{x})\right|>t\right) - P_{\cdot|\bm{S}}\left(\left|\sum_{\bm{\ell} \in L_{2,n}}\tilde{Z}^{'(\bm{\ell};\bm{\epsilon})}_{A_{n}}(\bm{u}, \bm{x})\right|>t\right)\right|\\
& \leq C\left({A_{n} \over A_{1,n}}\right)^{d}\beta(A_{2,n};A_{n}^{d})\ P_{\bm{S}}-a.s.
\end{align*} 
Since $A_{n}^{d}A_{1,n}^{-d}\beta(A_{2,n};A_{n}^{d}) \to 0$ as $n \to \infty$, these results imply that
\begin{align*}
\sum_{\bm{\ell} \in L_{1,n} \cup L_{2,n}}Z^{'(\bm{\ell};\bm{\epsilon}_{0})}_{A_{n}}(\bm{u}, \bm{x}) &= O\left(\sum_{\bm{\ell} \in L_{1,n} \cup L_{2,n}}\tilde{Z}^{'(\bm{\ell};\bm{\epsilon}_{0})}_{A_{n}}(\bm{u}, \bm{x})\right)\ P_{\bm{S}}-a.s.\\
\sum_{\bm{\ell} \in L_{1,n}}Z^{'(\bm{\ell};\bm{\epsilon})}_{A_{n}}(\bm{u}, \bm{x}) &= O\left(\sum_{\bm{\ell} \in L_{1,n}}\tilde{Z}^{'(\bm{\ell};\bm{\epsilon})}_{A_{n}}(\bm{u}, \bm{x})\right)\ P_{\bm{S}}-a.s.\\
\sum_{\bm{\ell} \in L_{2,n}}Z^{'(\bm{\ell};\bm{\epsilon})}_{A_{n}}(\bm{u}, \bm{x}) &= O\left(\sum_{\bm{\ell} \in L_{2,n}}\tilde{Z}^{'(\bm{\ell};\bm{\epsilon})}_{A_{n}}(\bm{u}, \bm{x})\right)\ P_{\bm{S}}-a.s.
\end{align*}

Now we show $\sup_{(\bm{u},\bm{x}) \in B}\left|\hat{\psi}_{1}(\bm{u},\bm{x}) - E_{\cdot|\bm{S}}[\hat{\psi}_{1}(\bm{u},\bm{x})]\right| = O_{P_{\cdot|\bm{S}}}\left(a_{n}\right)$. Cover the region $B$ with $N \leq Ch^{-(d+p)}a_{n}^{-(d+p)}$ balls $B_{n} = \{(\bm{u}, \bm{x}) \in \mathbb{R}^{d+p}: \|(\bm{u}, \bm{x}) - (\bm{u}_{n}, \bm{x}_{n})\|_{\infty} \leq a_{n}h\}$ and use $(\bm{u}_{n}, \bm{x}_{n})$ to denote the mid point of $B_{n}$, where $\|\bm{x}_{1} - \bm{x}_{2}\|_{\infty}:= \max_{1 \leq j \leq d}|x_{1,j} - x_{2,j}|$. In addition, let $K^{\ast}(\bm{w},\bm{v}) = C\prod_{k=1}^{d}\tilde{K}(w_{k})\prod_{j=1}^{p}I(|v_{j}| \leq 2C_{1})$ for $(\bm{w}, \bm{v}) \in \mathbb{R}^{d+p}$, where $\tilde{K}$ satisfies Assumption \ref{Ass-KB} (KB1) with $\tilde{K}^{\ast}(\bm{w}, \bm{v}) \geq C_{0}\prod_{k=1}^{d}I(|w_{k}| \leq 2C_{1})\prod_{j=1}^{p}I(|v_{j}| \leq 2C_{1})$. Note that for $(\bm{u}, \bm{x}) \in B_{n}$ and sufficiently large n,
\begin{align*}
&\left|\bar{K}_{h}\left(\bm{u} - {\bm{s} \over A_{n}}\right)\prod_{\ell=1}^{p}K_{h}\left(x_{\ell} - X_{\bm{s},A_{n}}^{\ell}\right) - \bar{K}_{h}\left(\bm{u}_{n} - {\bm{s} \over A_{n}}\right)\prod_{\ell=1}^{p}K_{h}\left(x_{\ell,n} - X_{\bm{s},A_{n}}^{\ell}\right)\right|\\
&\quad \leq a_{n}K_{h}^{\ast}\left(\bm{u}_{n} - {\bm{s} \over A_{n}}, \bm{x}_{n} - \bm{X}_{\bm{s},A_{n}}\right)
\end{align*} 
with $K_{h}^{\ast}(v) = K^{\ast}(v/h)$. For $\bm{\ell} \in L_{1,n} \cup L_{2,n}$ and $\bm{\epsilon} \in \{1,2\}^{d}$, define $Z_{A_{n}}^{''(\bm{\ell};\bm{\epsilon})}(\bm{u},\bm{x})$ by replacing
\begin{align*}
\bar{K}_{h}\left(\bm{u} - {\bm{s} \over A_{n}}\right)\prod_{\ell_{2}}^{p}K_{h}\left(x_{\ell_{2}} - X_{\bm{s},A_{n}}^{\ell_{2}}\right)W_{\bm{s}_{j},A_{n}}I(|W_{\bm{s}, A_{n}}| \leq \tau_{n})
\end{align*}
with 
\begin{align*}
K_{h}^{\ast}\left(\bm{u}_{n} - {\bm{s}_{j} \over A_{n}}, \bm{x}_{n} - \bm{X}_{\bm{s}_{j},A_{n}}\right)|W_{\bm{s}_{j},A_{n}}|I(|W_{\bm{s}_{j}, A_{n}}| \leq \tau_{n})
\end{align*}
in the definition of $Z'_{\bm{s}_{j}, A_{n}}(\bm{u},\bm{x})$ and define
\begin{align*}
\bar{\psi}_{1}(\bm{u},\bm{x}) &= {1 \over nh^{p+d}}\sum_{j=1}^{n}K_{h}^{\ast}\left(\bm{u}_{n} - {\bm{s}_{j} \over A_{n}}, \bm{x}_{n} - \bm{X}_{\bm{s}_{j},A_{n}}\right)|W_{\bm{s}_{j},A_{n}}|I(|W_{\bm{s}_{j}, A_{n}}| \leq \tau_{n}).
\end{align*}  
Note that $E_{\cdot|\bm{S}}\left[\left|\bar{\psi}_{1}(\bm{u},\bm{x})\right|\right]\leq M<\infty$ for some sufficiently large $M$, $P_{\bm{S}}$-a.s. Then we obtain
\begin{align*}
&\sup_{(\bm{u},\bm{x}) \in B}\left|\hat{\psi}_{1}(\bm{u},\bm{x}) - E_{\cdot|\bm{S}}[\hat{\psi}_{1}(\bm{u},\bm{x})]\right|\\
&\leq \left|\hat{\psi}_{1}(\bm{u}_{n},\bm{x}_{n}) - E_{\cdot|\bm{S}}[\hat{\psi}_{1}(\bm{u}_{n},\bm{x}_{n})]\right| + a_{n}\left(\left|\bar{\psi}_{1}(\bm{u}_{n},\bm{x}_{n})\right| + E_{\cdot|\bm{S}}\left[\left|\bar{\psi}_{1}(\bm{u}_{n},\bm{x}_{n})\right|\right]\right)\\
&\leq \left|\hat{\psi}_{1}(\bm{u}_{n},\bm{x}_{n}) - E_{\cdot|\bm{S}}[\hat{\psi}_{1}(\bm{u}_{n},\bm{x}_{n})]\right| + \left|\bar{\psi}_{1}(\bm{u}_{n},\bm{x}_{n}) - E_{\cdot|\bm{S}}[\bar{\psi}_{1}(\bm{u}_{n},\bm{x}_{n})]\right| + 2Ma_{n}\\
&\leq \left|\sum_{\bm{\ell} \in L_{1,n} \cup L_{2,n}}Z^{'(\bm{\ell};\bm{\epsilon}_{0})}_{A_{n}}(\bm{u}_{n}, \bm{x}_{n})\right| + \sum_{\bm{\epsilon} \neq \bm{\epsilon}_{0}}\left|\sum_{\bm{\ell} \in L_{1,n}}Z^{'(\bm{\ell};\bm{\epsilon})}_{A_{n}}(\bm{u}_{n}, \bm{x}_{n})\right| + \sum_{\bm{\epsilon} \neq \bm{\epsilon}_{0}}\left|\sum_{\bm{\ell} \in L_{2,n}}Z^{'(\bm{\ell};\bm{\epsilon})}_{A_{n}}(\bm{u}_{n}, \bm{x}_{n})\right| \\
&+ \left|\sum_{\bm{\ell} \in L_{1,n} \cup L_{2,n}}Z^{''(\bm{\ell};\bm{\epsilon}_{0})}_{A_{n}}(\bm{u}_{n}, \bm{x}_{n})\right| + \sum_{\bm{\epsilon} \neq \bm{\epsilon}_{0}}\left|\sum_{\bm{\ell} \in L_{1,n}}Z^{''(\bm{\ell};\bm{\epsilon})}_{A_{n}}(\bm{u}_{n}, \bm{x}_{n})\right| + \sum_{\bm{\epsilon} \neq \bm{\epsilon}_{0}}\left|\sum_{\bm{\ell} \in L_{2,n}}Z^{''(\bm{\ell};\bm{\epsilon})}_{A_{n}}(\bm{u}_{n}, \bm{x}_{n})\right|\\
&+ 2Ma_{n}. 
\end{align*}
Moreover, let $\{\tilde{Z}_{A_{n}}^{''(\bm{\ell};\bm{\epsilon})}(\bm{u},\bm{x})\}_{\bm{\ell} \in L_{1,n}\cup L_{2,n}}$ be independent random variables where $\tilde{Z}_{A_{n}}^{''(\bm{\ell};\bm{\epsilon})}(\bm{u},\bm{x}) \stackrel{d}{=} Z_{A_{n}}^{''(\bm{\ell};\bm{\epsilon})}(\bm{u},\bm{x})$. Then applying Corollary 2.7 in \cite{Yu94} to $\{\tilde{Z}_{A_{n}}^{'(\bm{\ell};\bm{\epsilon})}(\bm{u},\bm{x})\}_{\bm{\ell} \in L_{1,n}\cup L_{2,n}}$ and \\
$\{\tilde{Z}_{A_{n}}^{''(\bm{\ell};\bm{\epsilon})}(\bm{u},\bm{x})\}_{\bm{\ell} \in L_{1,n}\cup L_{2,n}}$, we have that 
\begin{align*}
&P_{\cdot|\bm{S}}\left(\sup_{(\bm{u},\bm{x}) \in B}\left|\hat{\psi}_{1}(\bm{u},\bm{x}) - E_{\cdot|\bm{S}}[\hat{\psi}_{1}(\bm{u},\bm{x})]\right|>2^{d+1}Ma_{n}\right)\\
&\leq N\max_{1 \leq k \leq N}P_{\cdot|\bm{S}}\left(\sup_{(\bm{u},\bm{x}) \in B_{k}}\left|\hat{\psi}_{1}(\bm{u},\bm{x}) - E_{\cdot|\bm{S}}[\hat{\psi}_{1}(\bm{u},\bm{x})]\right|>2^{d+1}Ma_{n}\right)\\
&\leq \sum_{\bm{\epsilon}\in \{1,2\}^{d}}\hat{Q}_{n}(\bm{\epsilon}) + \sum_{\bm{\epsilon}\in \{1,2\}^{d}}\bar{Q}_{n}(\bm{\epsilon}) + 2^{d+1}N\left({A_{n} \over A_{1,n}}\right)^{d}\beta(A_{2,n};A_{n}^{d}),
\end{align*}
where 
\begin{align*}
\hat{Q}_{n}(\bm{\epsilon}_{0}) &= N\max_{1 \leq k \leq N}P_{\cdot|\bm{S}}\left(\left|\sum_{\bm{\ell} \in L_{1,n} \cup L_{2,n}}\tilde{Z}^{'(\bm{\ell};\bm{\epsilon}_{0})}_{A_{n}}(\bm{u}_{k}, \bm{x}_{k})\right|>Ma_{n}nh^{d+p}\right),\\
\bar{Q}_{n}(\bm{\epsilon}_{0}) &= N\max_{1 \leq k \leq N}P_{\cdot|\bm{S}}\left(\left|\sum_{\bm{\ell} \in L_{1,n} \cup L_{2,n}}\tilde{Z}^{''(\bm{\ell};\bm{\epsilon}_{0})}_{A_{n}}(\bm{u}_{k}, \bm{x}_{k})\right|>Ma_{n}nh^{d+p}\right),
\end{align*}
and for $\bm{\epsilon} \neq \bm{\epsilon}_{0}$, 
\begin{align*}
\hat{Q}_{n}(\bm{\epsilon}) &= N\max_{1 \leq k \leq N}P_{\cdot|\bm{S}}\left(\left|\sum_{\bm{\ell} \in L_{1,n} \cup L_{2,n}}\tilde{Z}^{'(\bm{\ell};\bm{\epsilon})}_{A_{n}}(\bm{u}_{k}, \bm{x}_{k})\right|>Ma_{n}nh^{d+p}\right),\\
\bar{Q}_{n}(\bm{\epsilon}) &= N\max_{1 \leq k \leq N}P_{\cdot|\bm{S}}\left(\left|\sum_{\bm{\ell} \in L_{1,n} \cup L_{2,n}}\tilde{Z}^{''(\bm{\ell};\bm{\epsilon})}_{A_{n}}(\bm{u}_{k}, \bm{x}_{k})\right|>Ma_{n}nh^{d+p}\right).
\end{align*}
Since the proof is similar, we restrict our attention to $\hat{Q}_{n}(\bm{\epsilon})$, $\bm{\epsilon} \neq \bm{\epsilon}_{0}$. Note that
\begin{align*}
P_{\cdot|\bm{S}}\left(\left|\sum_{\bm{\ell} \in L_{1,n} \cup L_{2,n}}\tilde{Z}^{'(\bm{\ell};\bm{\epsilon})}_{A_{n}}(\bm{u}_{k}, \bm{x}_{k})\right|>Ma_{n}nh^{d+p}\right) \leq 2P_{\cdot|\bm{S}}\left(\sum_{\bm{\ell} \in L_{1,n} \cup L_{2,n}}\tilde{Z}^{'(\bm{\ell};\bm{\epsilon})}_{A_{n}}(\bm{u}_{k}, \bm{x}_{k})>Ma_{n}nh^{d+p}\right)
\end{align*}
Since $\tilde{Z}^{'(\bm{\ell};\bm{\epsilon})}_{A_{n}}(\bm{u}_{k}, \bm{x}_{k})$ are zero-mean random variables with 
\begin{align}
\left|\tilde{Z}^{'(\bm{\ell};\bm{\epsilon})}_{A_{n}}(\bm{u}_{k}, \bm{x}_{k})\right| &\leq CA_{1,n}^{d-1}A_{2,n}(\log n)\tau_{n},\ P_{\bm{S}}-a.s.\ (\text{from Lemma \ref{n summands}})\nonumber \\
E_{\cdot|\bm{S}}\left[\left(\tilde{Z}^{'(\bm{\ell};\bm{\epsilon})}_{A_{n}}(\bm{u}_{k}, \bm{x}_{k})\right)^{2}\right] &\leq Ch^{d+p}A_{1,n}^{d-1}A_{2,n}(\log n),\ P_{\bm{S}}-a.s., \label{bound-L2}
\end{align}
Lemma \ref{Bernstein} yields that 
\begin{align*}
&P_{\cdot|\bm{S}}\left(\sum_{\bm{\ell} \in L_{1,n} \cup L_{2,n}}\tilde{Z}^{'(\bm{\ell};\bm{\epsilon})}_{A_{n}}(\bm{u}_{k}, \bm{x}_{k})>Ma_{n}nh^{d+p}\right)\\
&\quad \leq \exp \left(-{{Mnh^{d+p}\log n \over 2} \over \left({A_{n} \over A_{1,n}}\right)^{d}A_{1,n}^{d-1}A_{2,n}h^{d+p}(\log n) + {M^{1/2}n^{1/2}h^{(d+p)/2}(\log n)^{1/2}A_{1,n}^{d-1}A_{2,n}\tau_{n} \over 3}}\right).
\end{align*}
For (\ref{bound-L2}), see Lemma \ref{decomp}. Observe that 
\begin{align*}
{nh^{d+p}\log n \over \left({A_{n} \over A_{1,n}}\right)^{d}A_{1,n}^{d-1}A_{2,n}h^{d+p}(\log n)} &= nA_{n}^{-d}\left({A_{1,n} \over A_{2,n}}\right) \gtrsim {A_{1,n} \over A_{2,n}} \gtrsim n^{\eta},\\
{nh^{d+p}\log n \over n^{1/2}h^{(d+p)/2}(\log n)^{1/2}A_{1,n}^{d-1}A_{2,n}\tau_{n}} &= {n^{1/2}h^{(d+p)/2}(\log n)^{1/2} \over A_{1,n}^{d}\left({A_{2,n} \over A_{1,n}}\right)\rho_{n}n^{1/\zeta}} \geq C_{0}n^{\eta/2}.
\end{align*}
Taking $M>0$ sufficiently large, this shows the desired result. 
\end{proof}

\begin{remark}\label{RemarkA1}
Let $\eta_{1} \in [0,1), \gamma_{2}, \gamma_{A_{1}}, \gamma_{A_{2}} \in (0,1)$ with $\gamma_{A_{1}}>\gamma_{A_{2}}$. Define
\begin{align*}
A_{n}^{d} = n^{1-\eta_{1}}, nh^{p+d} = n^{\gamma_{2}}, A_{1,n} = A_{n}^{\gamma_{A_{1}}}, A_{2,n} = A_{n}^{\gamma_{A_{2}}}. 
\end{align*}
Note that
\begin{align*}
{n^{1/2}h^{(d+p)/2}(\log n)^{1/2} \over A_{1,n}^{d}\left({A_{2,n} \over A_{1,n}}\right)\rho_{n}n^{1/\zeta}} &\geq {n^{1/2}h^{(d+p)/2}(\log n)^{1/2} \over A_{1,n}^{d}\rho_{n}n^{1/\zeta}} =  {n^{\gamma_{2}/2}(\log n)^{1/2} \over \rho_{n}n^{(1-\eta_{1})\gamma_{A_{1}} + 1/\zeta}}\\
& = {(\log n)^{1/2} \over \rho_{n}}n^{{\gamma_{2} \over 2} - (1-\eta_{1})\gamma_{A_{1}} - {1 \over \zeta}} 
\end{align*}
For $n^{{\gamma_{2} \over 2} - (1-\eta_{1})\gamma_{A_{1}} - {1 \over \zeta}} \gtrsim n^{\eta}$ for some $\eta>0$, we need
\begin{align}\label{RC4}
(1-\eta_{1})\gamma_{A_{1}}  + {1 \over \zeta} < {\gamma_{2} \over 2}.
\end{align}
\end{remark}

\begin{remark}\label{A2-depend-approx}
Assume that $\bm{X}_{\bm{s},A_{n}}$ satisfies Conditions (Ma1) and (Ma2) in Assumption \ref{Ass-M(add)}. Observe that
\begin{align*}
Z_{A_{n}}^{(\bm{\ell};\bm{\epsilon})}(\bm{u},\bm{x}) &= Z_{1,A_{n}}^{(\bm{\ell};\bm{\epsilon})}(\bm{u},\bm{x}: A_{2,n}) + Z_{2,A_{n}}^{(\bm{\ell};\bm{\epsilon})}(\bm{u},\bm{x}: A_{2,n})\\
&= \sum_{j: \bm{s}_{j} \in \Gamma_{n}(\bm{\ell};\bm{\epsilon}) \cap R_{n}}Z_{1,\bm{s}_{j},A_{n}}(\bm{u},\bm{x}: A_{2,n}) + \sum_{j: \bm{s}_{j} \in \Gamma_{n}(\bm{\ell};\bm{\epsilon}) \cap R_{n}}Z_{2,\bm{s}_{j},A_{n}}(\bm{u},\bm{x}: A_{2,n}),
\end{align*}
where 
\begin{align*}
Z_{1,\bm{s}_{j},A_{n}}(\bm{u},\bm{x}: A_{2,n}) &= \bar{K}_{h}\left(\bm{u} - {\bm{s} \over A_{n}}\right)\left\{ \prod_{\ell=1}^{p}K_{h}\left(x_{\ell} - X_{\bm{s}_{j}, A_{n}:A_{2,n}}^{\ell}\right)W_{\bm{s},A_{n}} \right.\\
&\left. \quad \quad - E_{\cdot|\bm{S}}\left[\prod_{\ell=1}^{p}K_{h}\left(x_{\ell} - X_{\bm{s}_{j}, A_{n}:A_{2,n}}^{\ell}\right)W_{\bm{s},A_{n}}\right] \right\}
\end{align*}
and $Z_{2,\bm{s}_{j},A_{n}}(\bm{u},\bm{x}: A_{2,n}) = \bar{Z}_{\bm{s}_{j},A_{n}}(\bm{u},\bm{x}) - Z_{1,\bm{s}_{j},A_{n}}(\bm{u},\bm{x}: A_{2,n})$. Note that
\begin{align*}
&|E_{\cdot|\bm{S}}\left[Z_{2,\bm{s}_{j},A_{n}}(\bm{u},\bm{x}: A_{2,n})\right]|\\
&\leq 2\bar{K}_{h}\left(\bm{u} - {\bm{s}_{j} \over A_{n}}\right)E_{\cdot|\bm{S}}\left[\left|\prod_{\ell=1}^{p}K_{h}\left(x_{\ell} - X_{\bm{s}_{j}, A_{n}:A_{2,n}}^{\ell}\right) - \prod_{\ell=1}^{p}K_{h}\left(x_{\ell} - X_{\bm{s}, A_{n}}^{\ell}\right)\right||W_{\bm{s}_{j},A_{n}}|\right]\\
&\leq 2C\sum_{\ell=1}^{p}E_{\cdot|\bm{S}}\left[\left|K_{h}\left(x_{\ell} - X_{\bm{s}_{j}, A_{n}:A_{2,n}}^{\ell}\right) - K_{h}\left(x_{\ell} - X_{\bm{s}, A_{n}}^{\ell}\right)\right||W_{\bm{s}_{j}, A_{n}}|\right]\\
& \leq 2CpE_{\cdot|\bm{S}}\left[\max_{1 \leq \ell \leq p}\left|{X_{\bm{s}, A_{n}}^{\ell} - X_{\bm{s}_{j}, A_{n}:A_{2,n}}^{\ell} \over h}\right||W_{\bm{s}_{j}, A_{n}}|\right] = {2Cp \over h}E_{\cdot|\bm{S}}\left[\max_{1 \leq \ell \leq p}\left|\epsilon_{\bm{s}_{j},A_{n}:A_{2,n}}\right||W_{\bm{s}_{j}, A_{n}}|\right]\\
& \leq {2Cp \over h}E_{\cdot|\bm{S}}\left[\max_{1 \leq \ell \leq p}\left|\epsilon_{\bm{s}_{j},A_{n}:A_{2,n}}\right|^{{\zeta \over \zeta-1}}\right]^{1-1/\zeta}E_{\cdot|\bm{S}}\left[|W_{\bm{s}_{j}, A_{n}}|^{\zeta}\right]^{1/\zeta}\\
&\leq {2Cp^{1+1/q} \over h}\max_{1 \leq \ell \leq p}E_{\cdot|\bm{S}}\left[\left|\epsilon_{\bm{s}_{j},A_{n}:A_{2,n}}\right|^{{q\zeta \over \zeta-1}}\right]^{{\zeta-1 \over q\zeta}} \leq 2Cp^{1+1/q}h^{-1}\gamma_{\epsilon}(A_{2,n}).
\end{align*}
Here, we used Lipschitz continuity of $K$ in the third inequality and Lemma 2.2.2 in \cite{vaWe96} in the fifth inequality. Applying Markov's inequality and Lemma \ref{n summands}, this yields that 
\begin{align}\label{Z2-resid}
&P_{\cdot|\bm{S}}\left(\left|\sum_{\ell \in L_{1,n} \cup L_{2,n}}Z_{2,A_{n}}^{(\bm{\ell};\bm{\epsilon})}(\bm{u},\bm{x}: A_{2,n})\right| > \sqrt{n^{1-\eta_{2}}h^{d+p}}\right) \nonumber \\
&\leq {\sum_{\ell \in L_{1,n} \cup L_{2,n}}\sum_{j: \bm{s}_{j} \in \Gamma_{n}(\bm{\ell};\bm{\epsilon}) \cap R_{n}}E_{\cdot|\bm{S}}\left[\left|Z_{2,\bm{s}_{j},A_{n}}(\bm{u},\bm{x}: A_{2,n})\right|\right] \over \sqrt{n^{1-\eta_{2}}h^{d+p}}} \nonumber \\
&\leq C{(A_{n}/A_{1,n})^{d}A_{1,n}^{d}(nA_{n}^{-d} + \log n)\gamma_{\epsilon}(A_{2,n}) \over n^{{1-\eta_{2} \over 2}}h^{{d+p \over 2}+1}} \to 0\ \text{as $n \to \infty$}. 
\end{align}
Define $Z_{1,A_{n}}^{'(\bm{\ell;\bm{\epsilon}})}(\bm{u}, \bm{x})$ by replacing $X_{\bm{s}_{j},A_{n}}^{\ell}$ in definition of $Z_{A_{n}}^{'(\bm{\ell;\bm{\epsilon}})}(\bm{u}, \bm{x})$ with $X_{\bm{s}_{j},A_{n}:A_{2,n}}$ and define $Z_{2,A_{n}}^{'(\bm{\ell;\bm{\epsilon}})}(\bm{u}, \bm{x}) = Z_{A_{n}}^{'(\bm{\ell;\bm{\epsilon}})}(\bm{u}, \bm{x}) - Z_{1,A_{n}}^{'(\bm{\ell;\bm{\epsilon}})}(\bm{u}, \bm{x})$. Observe that
\begin{align*}
&P_{\cdot|\bm{S}}\left(\left|\sum_{\bm{\ell} \in L_{1,n} \cup L_{2,n}}Z^{'(\bm{\ell};\bm{\epsilon})}_{A_{n}}(\bm{u}, \bm{x})\right|>t\right)\\
&= P_{\cdot|\bm{S}}\left(\left\{\left|\sum_{\bm{\ell} \in L_{1,n} \cup L_{2,n}}Z^{'(\bm{\ell};\bm{\epsilon})}_{A_{n}}(\bm{u}, \bm{x})\right|>t\right\} \cap \left\{\left|\sum_{\bm{\ell} \in L_{1,n} \cup L_{2,n}}Z^{'(\bm{\ell};\bm{\epsilon})}_{2,A_{n}}(\bm{u}, \bm{x})\right|>t_{0}\right\}\right)\\
&\quad + P_{\cdot|\bm{S}}\left(\left\{\left|\sum_{\bm{\ell} \in L_{1,n} \cup L_{2,n}}Z^{'(\bm{\ell};\bm{\epsilon})}_{1,A_{n}}(\bm{u}, \bm{x}) + \sum_{\bm{\ell} \in L_{1,n} \cup L_{2,n}}Z^{'(\bm{\ell};\bm{\epsilon})}_{2,A_{n}}(\bm{u}, \bm{x})\right|>t\right\}  \cap \left\{\left|\sum_{\bm{\ell} \in L_{1,n} \cup L_{2,n}}Z^{'(\bm{\ell};\bm{\epsilon})}_{2,A_{n}}(\bm{u}, \bm{x})\right|\leq t_{0}\right\}\right)\\
&\leq P_{\cdot|\bm{S}}\left(\left|\sum_{\bm{\ell} \in L_{1,n} \cup L_{2,n}}Z^{'(\bm{\ell};\bm{\epsilon})}_{2,A_{n}}(\bm{u}, \bm{x})\right|>t_{0}\right) + P_{\cdot|\bm{S}}\left(\left|\sum_{\bm{\ell} \in L_{1,n} \cup L_{2,n}}Z^{'(\bm{\ell};\bm{\epsilon})}_{1,A_{n}}(\bm{u}, \bm{x})\right|>t - t_{0}\right)
\end{align*}
Set $t = Ma_{n}nh^{d+p}$ and $t_{0} = \sqrt{n^{1-\eta_{2}}h^{d+p}} \lesssim \sqrt{nh^{d+p}}$. Let $\tilde{\beta}(a;b)$ be the $\beta$-mixing coefficients of $X_{\bm{s},A_{n}:A_{2,n}}$. Applying (\ref{Z2-resid}) and Lemma \ref{indep_lemma} (Corollary 2.7 in \cite{Yu94}), we have that 
\begin{align*}
&P_{\cdot|\bm{S}}\left(\left|\sum_{\bm{\ell} \in L_{1,n} \cup L_{2,n}}Z^{'(\bm{\ell};\bm{\epsilon})}_{A_{n}}(\bm{u}, \bm{x})\right|>Ma_{n}nh^{d+p}\right) \leq P_{\cdot|\bm{S}}\left(\left|\sum_{\bm{\ell} \in L_{1,n} \cup L_{2,n}}Z^{'(\bm{\ell};\bm{\epsilon})}_{1,A_{n}}(\bm{u}, \bm{x})\right|>2Ma_{n}nh^{d+p}\right) + o(1)\\
&\leq P_{\cdot|\bm{S}}\left(\left|\sum_{\bm{\ell} \in L_{1,n} \cup L_{2,n}}\tilde{Z}^{'(\bm{\ell};\bm{\epsilon})}_{1,A_{n}}(\bm{u}, \bm{x})\right|>2Ma_{n}nh^{d+p}\right) + \left({A_{n} \over A_{1,n}}\right)^{d}\underbrace{\tilde{\beta}(A_{2,n};A_{n}^{d})}_{= 0} + o(1),
\end{align*}
where $\{\tilde{Z}^{'(\bm{\ell};\bm{\epsilon})}_{1,A_{n}}(\bm{u}, \bm{x})\}$ are independent random variables with $\tilde{Z}^{'(\bm{\ell};\bm{\epsilon})}_{1,A_{n}}(\bm{u}, \bm{x}) \stackrel{d}{=} Z^{'(\bm{\ell};\bm{\epsilon})}_{1,A_{n}}(\bm{u}, \bm{x})$. This implies that under (Ma1) and (Ma2), we can replace $\bm{X}_{\bm{s},A_{n}}$ with $\bm{X}_{\bm{s},A_{n}:A_{2,n}}$ in our analysis.
  
\end{remark}

\begin{proof}[Proof of Theorem \ref{unif-rate-m}]
Observe that 
\begin{align*}
\hat{m}(\bm{u}, \bm{x}) - m(\bm{u},\bm{x}) &= {1 \over \hat{f}(\bm{u}, \bm{x})}\left(\hat{g}_{1}(\bm{u},\bm{x}) + \hat{g}_{2}(\bm{u},\bm{x}) - m(\bm{u},\bm{x})\hat{f}(\bm{u},\bm{x})\right),
\end{align*}
where
\begin{align*}
\hat{f}(\bm{u},\bm{x}) &= {1 \over nh^{d+p}}\sum_{j=1}^{n}\bar{K}_{h}\left(\bm{u} - {\bm{s}_{j} \over A_{n}}\right)\prod_{\ell=1}^{p}K_{h}\left(x_{\ell} - X_{\bm{s}_{j},A_{n}}^{\ell}\right),\\
\hat{g}_{1}(\bm{u},\bm{x}) &= {1 \over nh^{d+p}}\sum_{j=1}^{n}\bar{K}_{h}\left(\bm{u} - {\bm{s}_{j} \over A_{n}}\right)\prod_{\ell=1}^{p}K_{h}\left(x_{\ell} - X_{\bm{s}_{j},A_{n}}^{\ell}\right)\epsilon_{\bm{s}_{j}, A_{n}},\\
\hat{g}_{2}(\bm{u},\bm{x}) &= {1 \over nh^{d+p}}\sum_{j=1}^{n}\bar{K}_{h}\left(\bm{u} - {\bm{s}_{j} \over A_{n}}\right)\prod_{\ell=1}^{p}K_{h}\left(x_{\ell} - X_{\bm{s}_{j},A_{n}}^{\ell}\right)m\left({\bm{s}_{j} \over A_{n}} , \bm{X}_{\bm{s}_{j},A_{n}}\right). 
\end{align*}

(Step1) First we give a sketch of the proof. In Steps and 2, we show the following four results: 
\begin{itemize}
\item[(i)] $\sup_{\bm{u} \in [0,1]^{d}, \bm{x} \in S_{c}}\left|\hat{g}_{1}(\bm{u},\bm{x})\right| = O_{P_{\cdot|\bm{S}}}\left(\sqrt{(\log n)/nh^{d+p}}\right)$, $P_{\bm{S}}$-a.s. 
\item[(ii)] 
\begin{align*}
&\sup_{\bm{u} \in [0,1]^{d}, \bm{x} \in S_{c}}\left|\hat{g}_{2}(\bm{u},\bm{x}) - m(\bm{u},\bm{x})\hat{f}(\bm{u},\bm{x}) - E_{\cdot|\bm{S}}\left[\hat{g}_{2}(\bm{u},\bm{x}) - m(\bm{u},\bm{x})\hat{f}(\bm{u},\bm{x})\right]\right|\\
&\quad = O_{P_{\cdot|\bm{S}}}\left(\sqrt{(\log n)/nh^{d+p}}\right),\  P_{\bm{S}}-a.s.
\end{align*}
\item[(iii)] Let $\kappa_{2} = \int_{\mathbb{R}}x^{2}K(x)dx$. 
\begin{align*}
&\sup_{\bm{u} \in I_{h}, \bm{x} \in S_{c}}\left|E_{\cdot|\bm{S}}\left[\hat{g}_{2}(\bm{u},\bm{x}) - m(\bm{u},\bm{x})\hat{f}(\bm{u},\bm{x})\right]\right|\\
&\quad = h^{2}{\kappa_{2} \over 2}f_{\bm{S}}(\bm{u})\left\{\sum_{i = 1}^{d}\left(2\partial_{u_{i}}m(\bm{u},\bm{x})\partial_{u_{i}}f(\bm{u},\bm{x}) + \partial_{u_{i}u_{i}}^{2}m(\bm{u},\bm{x})f(\bm{u},\bm{x})\right) \right. \\
&\left .\quad \quad +  \sum_{k = 1}^{p}\left(2\partial_{x_{k}}m(\bm{u},\bm{x})\partial_{x_{k}}f(\bm{u},\bm{x}) + \partial_{x_{k}x_{k}}^{2}m(\bm{u},\bm{x})f(\bm{u},\bm{x})\right) \right\} + O\left({1 \over A_{n}^{dr}h^{p}}\right) +  o(h^{2}),\ 
P_{\bm{S}}-a.s.
\end{align*}
\item[(iv)] $\sup_{\bm{u} \in I_{h}, \bm{x} \in S_{c}}\left|\hat{f}(\bm{u},\bm{x}) - f(\bm{u},\bm{x})\right| = o_{P_{\cdot|\bm{S}}}(1)$, $P_{\bm{S}}$-a.s.
\end{itemize}
(i) can be shown by applying Proposition \ref{general-unif-rate} with $W_{\bm{s}_{j},A_{n}} = \epsilon_{\bm{s}_{j},A}$. (ii) can be shown by applying Proposition \ref{general-unif-rate} to $\hat{g}_{2}(\bm{u},\bm{x}) - m(\bm{u},\bm{x})\hat{f}(\bm{u},\bm{x})$. For the proof of (iv), we decompose $\hat{f}(\bm{u},\bm{x}) - f(\bm{u},\bm{x})$ into a variance part $\hat{f}(\bm{u},\bm{x}) - E_{\cdot|\bm{S}}[\hat{f}(\bm{u},\bm{x})]$ and a bias part $E_{\cdot|\bm{S}}[\hat{f}(\bm{u},\bm{x})] - f(\bm{u},\bm{x})$. Applying Proposition \ref{general-unif-rate} with $W_{\bm{s}_{j},A_{n}} = 1$, we have that the variance part is $o_{P_{\cdot|\bm{S}}}(1)$ uniformly in $\bm{u}$ and $\bm{x} \in S_{c}$. The bias part can be evaluated by similar arguments used to prove (iii). Combining the results (i), (ii) and (iii), we have that
\begin{align*}
&\sup_{\bm{u} \in I_{h}, \bm{x} \in S_{c}}\left|\hat{m}(\bm{u},\bm{x}) - m(\bm{u},\bm{x})\right|\\
&\leq {1 \over \inf_{\bm{u} \in I_{h}, \bm{x} \in S_{c}}\hat{f}(\bm{u},\bm{x})}\left(\sup_{\bm{u} \in I_{h}, \bm{x} \in S_{c}}\left|\hat{g}_{1}(\bm{u},\bm{x})\right| + \sup_{\bm{u} \in I_{h}, \bm{x} \in S_{c}}\left|\hat{g}_{2}(\bm{u},\bm{x}) - m(\bm{u},\bm{x})\hat{f}(\bm{u},\bm{x})\right|\right)\\
&\leq {1 \over \inf_{\bm{u} \in I_{h}, \bm{x} \in S_{c}}\hat{f}(\bm{u}, \bm{x})}O_{P_{\cdot|\bm{S}}}\left(\sqrt{\log n \over nh^{d+p}} + {1 \over A_{n}^{dr}h^{p}} + h^{2}\right).
\end{align*}
The result (iv) and $\inf_{\bm{u} \in [0,1]^{d},\bm{x} \in S_{c}}f(\bm{u},\bm{x})>0$ imply that 
\begin{align*}
{1 \over \inf_{\bm{u} \in I_{h}, \bm{x} \in S_{c}}\hat{f}(\bm{u},\bm{x})} &\lesssim {1 \over \inf_{\bm{u} \in [0,1]^{d}, \bm{x} \in S_{c}}f(\bm{u}, \bm{x}) - \inf_{\bm{u} \in I_{h}, \bm{x} \in S_{c}}|\hat{f}(\bm{u},\bm{x}) - f(\bm{u},\bm{x})|}= O_{P_{\cdot|\bm{S}}}(1).
\end{align*} 
Therefore, we complete the proof. 

(Step2) In this step, we show (iii). Let $K_{0}: \mathbb{R} \to \mathbb{R}$ be a Lipschitz continuous function with support $[-qC_{1},qC_{1}]$ for some $q>1$. Assume that $K_{0}(x) = 1$ for all $x \in [-C_{1},C_{1}]$ and write $K_{0,h}(x) = K_{0}(x/h)$. Observe that 
\begin{align*}
E_{\cdot|\bm{S}}\left[\hat{g}_{2}(\bm{u},\bm{x}) - m(\bm{u},\bm{x})\hat{f}(\bm{u},\bm{x}))\right] &= \sum_{i=1}^{4}Q_{i}(\bm{u},\bm{x}),
\end{align*}
where
\begin{align*}
Q_{i}(\bm{u},\bm{x}) &= {1 \over nh^{d+p}}\sum_{j=1}^{n}\bar{K}_{h}\left(\bm{u} - {\bm{s}_{j} \over A_{n}}\right)q_{i}(\bm{u},\bm{x})
\end{align*}
and 
\begin{align*}
q_{1}(\bm{u},\bm{x}) &= E_{\cdot|\bm{S}}\left[ \prod_{\ell=1}^{p}K_{0,h}(x_{\ell} - X_{\bm{s}_{j},A_{n}}^{\ell})\left\{\prod_{\ell=1}^{p}K_{h}(x_{\ell} - X_{\bm{s}_{j},A_{n}}^{\ell}) - \prod_{\ell=1}^{p}K_{h}\left(x_{\ell} - X_{{\bm{s}_{j} \over A_{n}}}^{\ell}(\bm{s}_{j})\right) \right\} \right. \\
&\left .\quad \quad \times \left\{m\left({\bm{s}_{j} \over A_{n}}, \bm{X}_{\bm{s}_{j},A_{n}}\right) - m(\bm{u},\bm{x})\right\} \right],
\end{align*}
\begin{align*}
q_{2}(\bm{u},\bm{x}) &= E_{\cdot|\bm{S}}\left[ \prod_{\ell=1}^{p}K_{0,h}(x_{\ell} - X_{\bm{s}_{j},A_{n}}^{\ell})\prod_{\ell=1}^{p}K_{h}\left(x_{\ell} - X_{{\bm{s}_{j} \over A_{n}}}^{\ell}(\bm{s}_{j})\right) \left\{m\left({\bm{s}_{j} \over A_{n}}, \bm{X}_{\bm{s}_{j},A_{n}}\right) - m\left({\bm{s}_{j} \over A_{n}},\bm{X}_{{\bm{s}_{j} \over A_{n}}}(\bm{s}_{j})\right)\right\} \right],
\end{align*}
\begin{align*}
q_{3}(\bm{u},\bm{x}) &= E_{\cdot|\bm{S}}\left[\left\{\prod_{\ell=1}^{p}K_{0,h}(x_{\ell} - X_{\bm{s}_{j},A_{n}}^{\ell}) - \prod_{\ell=1}^{p}K_{0,h}\left(x_{\ell} - X_{{\bm{s}_{j} \over A_{n}}}^{\ell}(\bm{s}_{j})\right) \right\} \right. \\
&\left. \quad \quad \prod_{\ell=1}^{p}K_{h}\left(x_{\ell} - X_{{\bm{s}_{j} \over A_{n}}}^{\ell}(\bm{s}_{j})\right) \left\{m\left({\bm{s}_{j} \over A_{n}}, \bm{X}_{{\bm{s}_{j} \over A_{n}}}(\bm{s}_{j})\right) - m(\bm{u},\bm{x})\right\}\right],
\end{align*}
\begin{align*}
q_{4}(\bm{u},\bm{x}) &= E_{\cdot|\bm{S}}\left[\prod_{\ell=1}^{p}K_{h}\left(x_{\ell} - X_{{\bm{s}_{j} \over A_{n}}}^{\ell}(\bm{s}_{j})\right)\left\{m\left({\bm{s}_{j} \over A_{n}}, \bm{X}_{{\bm{s}_{j} \over A_{n}}}(\bm{s}_{j})\right) - m(\bm{u},\bm{x})\right\} \right].
\end{align*}
We first consider $Q_{1}(\bm{u},\bm{x})$. Since the kernel $K$ is bounded, we can use the telescoping argument to get that
\begin{align*}
\left|\prod_{\ell=1}^{p}K_{h}(x_{\ell} - X_{\bm{s}_{j},A_{n}}^{\ell}) - \prod_{\ell=1}^{p}K_{h}\left(x_{\ell} - X_{{\bm{s}_{j} \over A_{n}}}^{\ell}(\bm{s}_{j})\right)\right| \leq C\sum_{\ell = 1}^{p}\left|K_{h}(x_{\ell} - X_{\bm{s}_{j},A_{n}}^{\ell}) - K_{h}\left(x_{\ell} - X_{{\bm{s}_{j} \over A_{n}}}^{\ell}(\bm{s}_{j})\right)\right|.
\end{align*}
Once again using the boundedness of $K$, we can find a constant $C<\infty$ such that 
\begin{align*}
\left|K_{h}(x_{\ell} - X_{\bm{s}_{j},A_{n}}^{\ell}) - K_{h}\left(x_{\ell} - X_{{\bm{s}_{j} \over A_{n}}}^{\ell}(\bm{s}_{j})\right)\right| &\leq \left|K_{h}(x_{\ell} - X_{\bm{s}_{j},A_{n}}^{\ell}) - K_{h}\left(x_{\ell} - X_{{\bm{s}_{j} \over A_{n}}}^{\ell}(\bm{s}_{j})\right)\right|^{r},
\end{align*}
where $r = \min\{\rho,1\}$. Therefore, 
\begin{align*}
\left|\prod_{\ell=1}^{p}K_{h}(x_{\ell} - X_{\bm{s}_{j},A_{n}}^{\ell}) - \prod_{\ell=1}^{p}K_{h}\left(x_{\ell} - X_{{\bm{s}_{j} \over A_{n}}}^{\ell}(\bm{s}_{j})\right)\right| \leq C\sum_{\ell = 1}^{p}\left|K_{h}(x_{\ell} - X_{\bm{s}_{j},A_{n}}^{\ell}) - K_{h}\left(x_{\ell} - X_{{\bm{s}_{j} \over A_{n}}}^{\ell}(\bm{s}_{j})\right)\right|^{r}.
\end{align*}
Applying this inequality, we have that
\begin{align*}
Q_{1}(\bm{u},\bm{x}) &\leq {C \over nh^{d+p}}\sum_{j=1}^{n}\bar{K}_{h}\left(\bm{u} - {\bm{s}_{j} \over A_{n}}\right)E_{\cdot|\bm{S}}\left[\sum_{\ell = 1}^{p}\left|K_{h}(x_{\ell} - X_{\bm{s}_{j},A_{n}}^{\ell}) - K_{h}\left(x_{\ell} - X_{{\bm{s}_{j} \over A_{n}}}^{\ell}(\bm{s}_{j})\right)\right|^{r} \right. \\
&\left. \quad \times \prod_{\ell=1}^{p}K_{0,h}(x_{\ell} - X_{\bm{s}_{j},A_{n}}^{\ell}) \left|m\left({\bm{s}_{j} \over A_{n}}, \bm{X}_{\bm{s}_{j},A_{n}}\right) - m(\bm{u},\bm{x})\right|\right].
\end{align*}
Note that $\prod_{\ell=1}^{p}K_{0,h}(x_{\ell} - X_{\bm{s}_{j},A_{n}}^{\ell}) \left|m\left({\bm{s}_{j} \over A_{n}}, \bm{X}_{\bm{s}_{j},A_{n}}\right) - m(\bm{u},\bm{x})\right| \leq Ch$. Since $K$ is Lipschitz, $\left|X_{\bm{s}_{j},A_{n}}^{\ell} - X_{{\bm{s}_{j} \over A_{n}}}^{\ell}(\bm{s}_{j})\right| \leq {C \over A_{n}^{d}}U_{\bm{s}_{j},A_{n}}(\bm{s}_{j}/A_{n})$ and the variable $U_{\bm{s}_{j},A_{n}}(\bm{s}_{j}/A_{n})$ have finite $r$-th moment, we have that 
\begin{align*}
&Q_{1}(\bm{u},\bm{x})\\
&\quad \leq {C \over nh^{d+p-1}}\sum_{j=1}^{n}\bar{K}_{h}\left(\bm{u} - {\bm{s}_{j} \over A_{n}}\right)E_{\cdot|\bm{S}}\left[\sum_{\ell = 1}^{p}\left|K_{h}(x_{\ell} - X_{\bm{s}_{j},A_{n}}^{\ell}) - K_{h}\left(x_{\ell} - X_{{\bm{s}_{j} \over A_{n}}}^{\ell}(\bm{s}_{j})\right)\right|^{r}\right]\\
&\quad \leq {C \over nh^{d+p-1}}\sum_{j=1}^{n}\bar{K}_{h}\left(\bm{u} - {\bm{s}_{j} \over A_{n}}\right)E_{\cdot|\bm{S}}\left[\sum_{k=1}^{p}\left|{1 \over A_{n}^{d}h}U_{\bm{s}_{j},A_{n}}\left({\bm{s}_{j} \over A_{n}}\right)\right|^{r}\right] \leq {C \over A_{n}^{dr}h^{p-1+r}}\ (\text{from Lemma \ref{Masry-thm}})
\end{align*}
uniformly in $\bm{u}$ and $\bm{x}$. Using similar arguments, we can also show that 
\begin{align*}
\sup_{\bm{u} \in I_{h}, \bm{x} \in S_{c}}|Q_{2}(\bm{u},\bm{x})| &\leq {C \over A_{n}^{dr}h^{p}},\ \sup_{\bm{u} \in I_{h}, \bm{x} \in S_{c}}|Q_{3}(\bm{u},\bm{x})| \leq {C \over A_{n}^{dr}h^{p-1+r}}. 
\end{align*}
Finally, applying Lemmas \ref{Masry-thm2} and \ref{Masry-thm3} and using the assumptions on the smoothness of $m$ and $f$, we have that
\begin{align*}
Q_{4}(\bm{u},\bm{x}) &= h^{2}{\kappa_{2} \over 2}f_{\bm{S}}(\bm{u})\left\{\sum_{i = 1}^{d}\left(2\partial_{u_{i}}m(\bm{u},\bm{x})\partial_{u_{i}}f(\bm{u},\bm{x}) + \partial_{u_{i}u_{i}}^{2}m(\bm{u},\bm{x})f(\bm{u},\bm{x})\right) \right. \\
&\left .\quad +  \sum_{k = 1}^{p}\left(2\partial_{x_{k}}m(\bm{u},\bm{x})\partial_{x_{k}}f(\bm{u},\bm{x}) + \partial_{x_{k}x_{k}}^{2}m(\bm{u},\bm{x})f(\bm{u},\bm{x})\right) \right\} + o(h^{2})
\end{align*}
uniformly in $\bm{u}$ and $\bm{x}$. Combining the results on $Q_{i}(\bm{u},\bm{x})$, $1 \leq i \leq 4$ yields (iii). 
\end{proof}

\begin{remark}\label{Rem-A4}
Let $h \sim n^{1/(d+p+4)}$ and $A_{n}^{d} = n^{1-\eta_{1}}$ for some $\eta_{1} \in [0,1)$. 
\begin{align*}
{1 \over A_{n}^{dr}h^{p}} \lesssim h^{2} \Leftrightarrow 1 \lesssim A_{n}^{dr}h^{p+2} = n^{r(1-\eta_{1}) - {p+2 \over d+p+4}}(\log n)^{{r\eta_{2} \over d}} \Leftarrow 1-\eta_{1}\geq {p+2 \over r(d+p+4)}.
\end{align*}
If $r = 1$, 
\begin{align}\label{RC-unif}
1-\eta_{1}\geq {p+2 \over d+p+4} \Leftrightarrow {d+2 \over d+p+4} \geq \eta_{1}.
\end{align}
(\ref{RC-unif}) is satisfied for $d \geq 1$ and $p \geq 1$.
\end{remark}

\begin{proof}[Proof of Theorem \ref{general-CLT-m}]
Observe that
\begin{align*}
\sqrt{nh^{d+p}}(\hat{m}(\bm{u},\bm{x}) - m(\bm{u},\bm{x})) = {\sqrt{nh^{d+p}} \over \hat{f}(\bm{u},\bm{x})}\left(\hat{g}_{1}(\bm{u},\bm{x}) + \hat{g}_{2}(\bm{u},\bm{x}) - m(\bm{u},\bm{x})\hat{f}(\bm{u},\bm{x}))\right).
\end{align*}
Define
\begin{align*}
B(\bm{u},\bm{x}) &= \sqrt{nh^{p+d}}\left(\hat{g}_{2}(\bm{u},\bm{x}) - m(\bm{u},\bm{x})\hat{f}(\bm{u},\bm{x}))\right),\ V(\bm{u},\bm{x}) =  \sqrt{nh^{d+p}}\hat{g}_{1}(\bm{u},\bm{x}).
\end{align*}
$B(\bm{u},\bm{x})$ converge in $P_{\cdot|\bm{S}}$-probability to $B_{\bm{u},\bm{x}}$ defined in Theorem \ref{general-CLT-m}. This follows from (iii) in the proof of Theorem \ref{unif-rate-m} and the fact that $B(\bm{u},\bm{x}) - E_{\cdot|\bm{S}}[B(\bm{u},\bm{x})] = o_{P_{\cdot|\bm{S}}}(1)$. To prove the latter, it is sufficient to prove $\Var_{\cdot|\bm{S}}(B(\bm{u},\bm{x})) = o(1)$, $P_{\bm{S}}$-a.s., which can be shown by similar arguments used in the proof Lemma \ref{asy-variance}. Using the i.i.d. assumption on $\{\epsilon_{j}\}$, $\bar{K}(\bm{u}) = \prod_{\ell=1}^{d}K(u_{\ell})$ and Lemma \ref{Masry-thm2}, the asymptotic variance in (\ref{Asy-N}) can be computed as follows:
\begin{align*}
&\Var_{\cdot|\bm{S}}\left(V(\bm{u},\bm{x})\right)\\
&= \Var_{\cdot|\bm{S}}\left({1 \over \sqrt{nh^{d+p}}}\sum_{j=1}^{n}\bar{K}_{h}\left(\bm{u} - {\bm{s}_{j} \over A_{n}}\right)\prod_{\ell=1}^{p}K_{h}\left(x_{\ell} - X_{\bm{s}_{j},A_{n}}^{\ell}\right)\epsilon_{\bm{s}_{j}, A_{n}}\right)\\
&= {1 \over nh^{d+p}}\sum_{j=1}^{n}\bar{K}_{h}^{2}\left(\bm{u} - {\bm{s}_{j} \over A_{n}}\right)\int_{\mathbb{R}^{p}}\prod_{\ell=1}^{p}K_{h}\left(x_{\ell} - w_{\ell}\right)E_{\cdot|\bm{S}}\left[\epsilon_{\bm{s}_{j}, A_{n}}^{2}|\bm{X}_{\bm{s}_{j},A_{n}} = \bm{w}\right]f_{\bm{X}_{\bm{s}_{j},A_{n}}}(\bm{w})d\bm{w}\\
&= {1 \over nh^{d}}\sum_{j=1}^{n}\bar{K}_{h}^{2}\left(\bm{u} - {\bm{s}_{j} \over A_{n}}\right)\int_{\mathbb{R}^{p}}\prod_{\ell=1}^{p}K\left(\varphi_{\ell}\right)\sigma^{2}\left({\bm{s}_{j} \over A_{n}}, \bm{x}-h\bm{\varphi}\right)f_{\bm{X}_{\bm{s}_{j},A_{n}}}(\bm{x}-h\bm{\varphi})d\bm{\varphi}\\
&= {1 \over nh^{d}}\sum_{j=1}^{n}\bar{K}_{h}^{2}\left(\bm{u} - {\bm{s}_{j} \over A_{n}}\right)\left(\kappa_{0}^{p}\sigma^{2}\left({\bm{s}_{j} \over A_{n}}, \bm{x}\right)f\left({\bm{s}_{j} \over A_{n}},\bm{x}\right)\right) + o(1)\\
&= \kappa_{0}^{d+p}f_{\bm{S}}(\bm{u})\sigma^{2}\left(\bm{u}, \bm{x}\right)f\left(\bm{u},\bm{x}\right) + o(1)\ P_{\bm{S}}-a.s.
\end{align*}
Moreover, $V(\bm{u},\bm{x})$ is asymptotically normal. In particular, 
\begin{align}\label{Asy-N}
V(\bm{u},\bm{x}) \stackrel{d}{\to} N(0, \kappa_{0}^{d+p}f_{\bm{S}}(\bm{u})\sigma^{2}(\bm{u},\bm{x})f(\bm{u},\bm{x})).
\end{align}
We can show (\ref{Asy-N}) by applying blocking arguments. Decompose $V(\bm{u},\bm{x})$ into some big-blocks and small-blocks as in (\ref{decomp1}). We can neglect the small blocks by applying Lemma \ref{asy-variance} and use mixing conditions to replace the big blocks by independent random variables. This allows us to apply a Lyapunov's condition for the central limit theorem for sum of independent random variables to get the result. We omit the details as the proof is similar to that of Theorem 3.1 in \cite{La03b} under the standard strictly stationary $\alpha$-mixing settings for random fields. 
\end{proof}

\subsection{Proofs for Section \ref{Section_Additive}}

\begin{proof}[Proof of Theorems \ref{unif-m-add} and \ref{CLT-m-add}]
Since we can prove the desired result by applying almost the same strategy in the proof of Theorems 5.1 and 5.2 in \cite{Vo12}, we omit the proof. We note that it suffices to check conditions (A1)-(A6), (A8) and (A9) in \cite{MaLiNi99} to obtain the desired results. We can check those conditions by applying almost the same argument in the proof of Theorems 5.1 and 5.2 in \cite{Vo12} and applying Proposition \ref{general-unif-rate}, Theorems \ref{unif-rate-m}, \ref{general-CLT-m} and Lemmas \ref{lemma-C2}-\ref{lemma-C4} in this paper, which correspond to Lemmas C.2-C.4 in \cite{Vo12}. 
\end{proof}



\section{Auxiliary lemmas}

Define 
\begin{align}
\bar{Z}_{\bm{s},A_{n}}(\bm{u},\bm{x}) &= \bar{K}_{h}\left(\bm{u} - {\bm{s} \over A_{n}}\right)\left\{ \prod_{\ell_{2}=1}^{p}K_{h}\left(x_{\ell_{2}} - X_{\bm{s},A_{n}}^{\ell_{2}}\right)W_{\bm{s},A_{n}} - E_{\cdot|\bm{S}}\left[\prod_{\ell_{2}=1}^{p}K_{h}\left(x_{\ell_{2}} - X_{\bm{s},A_{n}}^{\ell_{2}}\right)W_{\bm{s},A_{n}}\right] \right\},
\end{align}

\begin{align*}
Z_{A_{n}}(\bm{u},\bm{x}) &= \sum_{j=1}^{n}\bar{Z}_{\bm{s}_{j},A_{n}}(\bm{u},\bm{x}) \nonumber \\
&= \sum_{\ell \in L_{1,n}\cup L_{2,n}}Z_{A_{n}}^{(\bm{\ell};\bm{\epsilon}_{0})}(\bm{u},\bm{x}) + \sum_{\bm{\epsilon} \neq \bm{\epsilon}_{0}}\sum_{\ell \in L_{1,n}}Z_{A_{n}}^{(\bm{\ell};\bm{\epsilon})}(\bm{u},\bm{x}) + \sum_{\bm{\epsilon} \neq \bm{\epsilon}_{0}}\sum_{\ell \in L_{2,n}}Z_{A_{n}}^{(\bm{\ell};\bm{\epsilon})}(\bm{u},\bm{x}), 
\end{align*}
where $Z_{A_{n}}^{(\bm{\ell};\bm{\epsilon})}(\bm{u},\bm{x}) = \sum_{j: \bm{s}_{j} \in \Gamma_{n}(\bm{\ell};\bm{\epsilon}) \cap R_{n}}\bar{Z}_{\bm{s}_{j},A_{n}}(\bm{u},\bm{x})$.

\begin{lemma}\label{decomp}
Under Assumptions \ref{Ass-S}, \ref{Ass-KB}, \ref{Ass-U}  and \ref{Ass-R},
\begin{align*}
E_{\cdot|\bm{S}}\left[\left(Z_{A_{n}}^{(\bm{\ell};\bm{\epsilon})}(\bm{u},\bm{x})\right)^{2}\right] &\leq CA_{1,n}^{d-1}A_{2,n}(nA_{n}^{-d} + \log n)h^{p+d}\ P_{\bm{S}}-a.s.
\end{align*}
\end{lemma}

\begin{proof}
Observe that 
\begin{align*}
E_{\cdot|\bm{S}}\left[\left(Z_{A_{n}}^{(\bm{\ell};\bm{\epsilon})}(\bm{u},\bm{x})\right)^{2}\right] &= \sum_{j: \bm{s}_{j} \in \Gamma_{n}(\bm{\ell};\bm{\epsilon}) \cap R_{n}}\!\!\!\!\!\!\!\!\!\!E_{\cdot|\bm{S}}\left[\bar{Z}_{\bm{s}_{j},A_{n}}^{2}(\bm{u},\bm{x})\right] + \sum_{j_{1} \neq j_{2}: \bm{s}_{j_{1}}, \bm{s}_{j_{2}} \in \Gamma_{n}(\bm{\ell};\bm{\epsilon}) \cap R_{n}}\!\!\!\!\!\!\!\!\!\!\!\!\!\!E_{\cdot|\bm{S}}\left[\bar{Z}_{\bm{s}_{j_{1}},A_{n}}(\bm{u},\bm{x})\bar{Z}_{\bm{s}_{j_{2}},A_{n}}(\bm{u},\bm{x})\right].  
\end{align*}
Note that
\begin{align*}
&E_{\cdot|\bm{S}}\left[\bar{Z}_{\bm{s}_{j},A_{n}}^{2}(\bm{u},\bm{x})\right]\\ 
&= \bar{K}_{h}^{2}\left(\bm{u} - {\bm{s}_{j} \over A_{n}}\right)\left( E_{\cdot|\bm{S}}\left[\prod_{\ell_{2}=1}^{p}K_{h}^{2}\left(x_{\ell_{2}} - X_{\bm{s}_{j},A_{n}}^{\ell_{2}}\right)W_{\bm{s}_{j},A_{n}}^{2} \right]   - \left(E_{\cdot|\bm{S}}\left[\prod_{\ell_{2}=1}^{p}K_{h}\left(x_{\ell_{2}} - X_{\bm{s}_{j},A_{n}}^{\ell_{2}}\right)W_{\bm{s}_{j},A_{n}}\right]\right)^{2}\right)\\
&\leq \bar{K}_{h}^{2}\left(\bm{u} - {\bm{s}_{j} \over A_{n}}\right)\left\{E_{\cdot|\bm{S}}\left[\prod_{\ell_{2}=1}^{p}K_{h}^{2}\left(x_{\ell_{2}} - X_{\bm{s}_{j},A_{n}}^{\ell_{2}}\right)W_{\bm{s}_{j},A_{n}}^{2} \right] + \left(E_{\cdot|\bm{S}}\left[\prod_{\ell_{2}=1}^{p}K_{h}\left(x_{\ell_{2}} - X_{\bm{s}_{j},A_{n}}^{\ell_{2}}\right)\left|W_{\bm{s}_{j},A_{n}}\right|\right]\right)^{2}\right\}.
\end{align*}
Observe that 
\begin{align*}
&E_{\cdot|\bm{S}}\left[\prod_{\ell_{2}=1}^{p}K_{h}\left(x_{\ell_{2}} - X_{\bm{s},A_{n}}^{\ell_{2}}\right)\left|W_{\bm{s},A_{n}}\right|\right]\\
&= \int_{\mathbb{R}^{p}} \prod_{\ell_{2}=1}^{p}K_{h}\left(x_{\ell_{2}} - X_{\bm{s},A_{n}}^{\ell_{2}}\right)E_{\cdot|\bm{S}}\left[\left|W_{\bm{s},A_{n}}\right||\bm{X}_{\bm{s},A_{n}} = \bm{w}\right]f_{\bm{X}_{\bm{s},A_{n}}}(\bm{w})d\bm{w}\\
&= h^{p}\int_{\mathbb{R}^{p}} \prod_{\ell_{2}=1}^{p}K_{h}\left(\varphi_{\ell_{2}}\right)E_{\cdot|\bm{S}}\left[\left|W_{\bm{s},A_{n}}\right||\bm{X}_{\bm{s},A_{n}} = \bm{x} - h\bm{\varphi}_{\ell_{2}}\right]f_{\bm{X}_{\bm{s},A_{n}}}(\bm{x} - h\bm{\varphi}_{\ell_{2}})d\bm{\varphi} \leq C\|K\|_{\infty}^{p}h^{p},
\end{align*}
where $\|K\|_{\infty} = \sup_{x \in \mathbb{R}}|K(x)|$. Likewise, 
\begin{align*}
E_{\cdot|\bm{S}}\left[\prod_{\ell_{2}=1}^{p}K_{h}^{2}\left(x_{\ell_{2}} - X_{\bm{s}_{j},A_{n}}^{\ell_{2}}\right)W_{\bm{s}_{j},A_{n}}^{2} \right] &\leq C\|K\|_{\infty}^{2p}h^{p}.
\end{align*}
Then
\begin{align}
E_{\cdot|\bm{S}}\left[\bar{Z}_{\bm{s}_{j},A_{n}}^{2}(\bm{u},\bm{x})\right] &\leq C(h^{p} + h^{2p})\|K\|_{\infty}^{2p}\bar{K}_{h}^{2}\left(\bm{u} - {\bm{s}_{j} \over A_{n}}\right) \leq Ch^{p}\|K\|_{\infty}^{2p}\bar{K}_{h}^{2}\left(\bm{u} - {\bm{s}_{j} \over A_{n}}\right)\ P_{\bm{S}}-a.s. \label{Z1}
\end{align}
Likewise,
\begin{align}\label{Z2}
\left|E_{\cdot|\bm{S}}\left[\bar{Z}_{\bm{s}_{j_{1}},A_{n}}(\bm{u},\bm{x})\bar{Z}_{\bm{s}_{j_{2}},A_{n}}(\bm{u},\bm{x})\right]\right| &\leq Ch^{2p}\|K\|_{\infty}^{2p}\bar{K}_{h}\left(\bm{u} - {\bm{s}_{j_{1}} \over A_{n}}\right)\bar{K}_{h}\left(\bm{u} - {\bm{s}_{j_{2}} \over A_{n}}\right)\ P_{\bm{S}}-a.s.
\end{align}
Then Lemmas \ref{Masry-thm} and \ref{n summands} imply that 
\begin{align*}
\sum_{j: \bm{s}_{j} \in \Gamma_{n}(\bm{\ell};\bm{\epsilon}) \cap R_{n}}\bar{K}_{h}^{2}\left(\bm{u} - {\bm{s}_{j} \over A_{n}}\right) &\leq C\sum_{j: \bm{s}_{j} \in \Gamma_{n}(\bm{\ell};\bm{\epsilon}) \cap R_{n}}\bar{K}_{h}\left(\bm{u} - {\bm{s}_{j} \over A_{n}}\right) \leq Ch^{d}[\![\{j: \bm{s}_{j} \in \Gamma_{n}(\bm{\ell};\bm{\epsilon}) \cap R_{n}\}]\!]\\
& \leq Ch^{d}A_{1,n}^{d-1}A_{2,n}(nA^{-d}+\log n),\ P_{\bm{S}}-a.s.,
\end{align*}
\begin{align*}
&\sum_{j_{1} \neq j_{2}: \bm{s}_{j_{1}}, \bm{s}_{j_{2}} \in \Gamma_{n}(\bm{\ell};\bm{\epsilon}) \cap R_{n}}\bar{K}_{h}\left(\bm{u} - {\bm{s}_{j_{1}} \over A_{n}}\right)\bar{K}_{h}\left(\bm{u} - {\bm{s}_{j_{2}} \over A_{n}}\right)\\
&\leq  \left(\sum_{j: \bm{s}_{j} \in \Gamma_{n}(\bm{\ell};\bm{\epsilon}) \cap R_{n}}\bar{K}_{h}\left(\bm{u} - {\bm{s}_{j} \over A_{n}}\right)\right)^{2} \leq Ch^{2d}[\![\{j: \bm{s}_{j} \in \Gamma_{n}(\bm{\ell};\bm{\epsilon}) \cap R_{n}\}]\!]^{2}\\
&\leq  Ch^{2d}A_{1,n}^{2(d-1)}A_{2,n}^{2}(nA^{-d}+\log n)^{2},\ P_{\bm{S}}-a.s.
\end{align*}
Since $A_{1,n}^{d-1}A_{2,n}(nA_{n}^{-d} + \log n)h^{p+d} \leq A_{1,n}^{d}(nA_{n}^{-d} + \log n)h^{p+d} = o(1)$,  (\ref{Z1}) and (\ref{Z2}) yield that 
\begin{align*}
E_{\cdot|\bm{S}}\left[\left(Z_{A_{n}}^{(\bm{\ell};\bm{\epsilon})}(\bm{u},\bm{x})\right)^{2}\right] &\leq C\left\{A_{1,n}^{d-1}A_{2,n}(nA_{n}^{-d} + \log n)h^{p+d} \right. \\
&\left. \quad + A_{1,n}^{2(d-1)}A_{2,n}^{2}(n^{2}A_{n}^{-2d} + \log^{2} n)h^{2(d+p)}\right\}\\
&\leq CA_{1,n}^{d-1}A_{2,n}(nA_{n}^{-d} + \log n)h^{d+p}, P_{\bm{S}}-a.s.
\end{align*}
\end{proof}

\begin{remark}\label{Rem-B1}
As we defined in Remark \ref{RemarkA1}, consider
\begin{align*}
A_{n}^{d} = n^{1-\eta_{1}}, nh^{p+d} = n^{\gamma_{2}}, A_{1,n} = A_{n}^{\gamma_{A_{1}}}, A_{2,n} = A_{n}^{\gamma_{A_{2}}}
\end{align*}
Note that $A_{1,n}^{d-1}A_{2,n}A_{n}^{-d}nh^{d+p} \leq A_{1,n}^{d}A_{n}^{-d}nh^{d+p}   = n^{-(1-\eta_{1})(1-\gamma_{A_{1}}) + \gamma_{2}}$.
For $A_{1,n}^{d}A_{n}^{-d}nh^{d+p} \lesssim n^{-c}$ for some $c>0$, we need 
\begin{align}\label{RC2}
\gamma_{2} < (1-\eta_{1})(1 -\gamma_{A_{1}}).
\end{align}
Note that $A_{1,n}^{d-1}A_{2,n}h^{d+p} =  n^{-1}A_{1,n}^{d-1}A_{2,n}nh^{p+d} \leq n^{-1}A_{1,n}^{d}nh^{p+d} = n^{(1-\eta_{1})\gamma_{A_{1}} + \gamma_{2} - 1}$. For $A_{1,n}^{d}h^{d+p} = o((\log n)^{-1})$, we need
\begin{align}\label{RC3}
(1-\eta_{1})\gamma_{A_{1}} + \gamma_{2} < 1.
\end{align}
(\ref{RC3}) and (\ref{RC2}) implies that 
\begin{align}\label{RC5}
\gamma_{2} < \min\{1 - (1-\eta_{1})\gamma_{A_{1}}, (1-\eta_{1})(1 - \gamma_{A_{1}})\} = (1-\eta_{1})(1 - \gamma_{A_{1}}).
\end{align}
(\ref{RC4}) and (\ref{RC5}) imply that $2(1-\eta_{1})\gamma_{A_{1}} + {2 \over \zeta} < \gamma_{2} < (1-\eta_{1})(1-\gamma_{A_{1}})$. For this, we need 
\begin{align*}
2(1-\eta_{1})\gamma_{A_{1}} + {2 \over \zeta} < (1-\eta_{1})(1 -\gamma_{A_{1}}) \Leftrightarrow {2 \over \zeta} < (1-\eta_{1})(1-3\gamma_{A_{1}}). 
\end{align*}
\end{remark}

\begin{lemma}\label{asy-variance}
Under Assumptions \ref{Ass-S}, \ref{Ass-KB}, \ref{Ass-U(add)}  and \ref{Ass-R(add)},
\begin{align}
{1 \over nh^{d+p}}\Var_{\cdot|\bm{S}}\left(\sum_{\ell \in L_{1,n}}Z_{A_{n}}^{(\bm{\ell};\bm{\epsilon})}(\bm{u},\bm{x})\right) &= o(1),\ P_{\bm{S}}-a.s. \label{V1} \\ 
{1 \over nh^{d+p}}\Var_{\cdot|\bm{S}}\left(\sum_{\ell \in L_{2,n}}Z_{A_{n}}^{(\bm{\ell};\bm{\epsilon})}(\bm{u},\bm{x})\right) &= o(1),\ P_{\bm{S}}-a.s. \label{V2}
\end{align}
\end{lemma}
\begin{proof}
Since the proof is similar, we only show (\ref{V1}). Note that 
\begin{align*}
{1 \over nh^{d+p}}\Var_{\cdot|\bm{S}}\left(\sum_{\bm{\ell} \in L_{1,n}}Z_{A_{n}}^{(\bm{\ell};\bm{\epsilon})}(\bm{u},\bm{x})\right) &= {1 \over nh^{d+p}}\sum_{\bm{\ell} \in L_{1,n}}E_{\cdot|\bm{S}}\left[\left(Z_{A_{n}}^{(\bm{\ell};\bm{\epsilon})}(\bm{u},\bm{x})\right)^{2}\right]\\
&\quad + {1 \over nh^{d+p}}\sum_{\bm{\ell}_{1}, \bm{\ell}_{2} \in L_{1,n}, \bm{\ell}_{1} \neq \bm{\ell}_{2}}\!\!\!\!\!\!E_{\cdot|\bm{S}}\left[Z_{A_{n}}^{(\bm{\ell}_{1};\bm{\epsilon})}(\bm{u},\bm{x})Z_{A_{n}}^{(\bm{\ell}_{2};\bm{\epsilon})}(\bm{u},\bm{x})\right]\\
&=: I_{1} + I_{2}.
\end{align*}
As a result of Lemma \ref{decomp}, 
\begin{align*}
I_{1} &\leq Cn^{-1}h^{-(d+p)}\left({A_{n} \over A_{1,n}}\right)^{d}A_{1,n}^{d-1}A_{2,n}(nA_{n}^{-d}+\log n)h^{d+p} = C{A_{2,n} \over A_{1,n}}(\log n) = o(1).
\end{align*}
Applying Theorem 1.1 in \cite{Ri13}, we have that 
\begin{align*}
&E_{\cdot|\bm{S}}\left[Z_{A_{n}}^{(\bm{\ell}_{1};\bm{\epsilon})}(\bm{u},\bm{x})Z_{A_{n}}^{(\bm{\ell}_{2};\bm{\epsilon})}(\bm{u},\bm{x})\right]\\
&\leq E_{\cdot|\bm{S}}\left[\left|Z_{A_{n}}^{(\bm{\ell}_{1};\bm{\epsilon})}(\bm{u},\bm{x})\right|^{3}\right]^{1/3}E_{\cdot|\bm{S}}\left[\left|Z_{A_{n}}^{(\bm{\ell}_{2};\bm{\epsilon})}(\bm{u},\bm{x})\right|^{3}\right]^{1/3}\beta_{1}^{1/3}(d(\bm{\ell}_{1}, \bm{\ell}_{2})A_{2,n})g^{1/3}_{1}(A_{1,n}^{d}),
\end{align*}
where $d(\bm{\ell}_{1}, \bm{\ell}_{2}) = \min_{1 \leq j \leq d}|\ell_{1}^{j} - \ell_{2}^{j}|$. A similar argument to show (\ref{Z1}) and (\ref{Z2}) yield that 
\begin{align*}
E_{\cdot|\bm{S}}\left[\left|Z_{A_{n}}^{(\bm{\ell}_{1};\bm{\epsilon})}(\bm{u},\bm{x})\right|^{3}\right] &\leq CA_{1,n}^{d-1}A_{2,n}(nA_{n}^{-d} + \log n)h^{d+p}.
\end{align*}
Therefore, similar arguments in the proof of Theorem 3.1 in \cite{La03b} yield
\begin{align*}
I_{2} &\leq C{(A_{1,n}^{d-1}A_{2,n}(nA_{n}^{-d}+\log n)h^{p+d})^{2/3} \over nh^{d+p}}\!\!\!\!\! \sum_{\bm{\ell}_{1}, \bm{\ell}_{2} \in L_{1,n}, \bm{\ell}_{1} \neq \bm{\ell}_{2}}\!\!\!\!\!\!\!\!\!\!\!\beta_{1}^{1/3}((|\bm{\ell}_{1} - \bm{\ell}_{2}|-d)_{+}A_{3,n}+A_{2,n})g^{1/3}_{1}(A_{1,n}^{d})\\
&\leq C\left\{\left({1 \over nh^{d+p}}\right)^{1/3}\left({A_{1,n} \over A_{n}}\right)^{2d/3}\left({A_{2,n} \over A_{1,n}}\right)^{2/3} + {A_{1,n}^{(d-1)/3}A_{2,n}^{1/3}(\log n)^{1/3} \over nh^{(d+p)/3}}\right\}\\
&\quad \times g^{1/3}_{1}(A_{1,n}^{d})\left\{\beta_{1}^{1/3}(A_{2,n}) + \sum_{k = 1}^{A_{n}/A_{1,n}}k^{d-1}\beta_{1}^{1/3}(kA_{3,n}+A_{2,n})\right\} = o(1)
\end{align*}
where $|\bm{\ell}_{1} - \bm{\ell}_{2}| = \sum_{j=1}^{d}|\ell_{1,j} - \ell_{2,j}|$. 
\end{proof}

\begin{lemma}\label{lemma-C1}
Define $\tilde{n}_{0} = E_{\cdot|\bm{S}}[\tilde{n}_{[0,1]^{p}}]$ where
\begin{align*}
\tilde{n}_{[0,1]^{p}} &= \sum_{j = 1}^{n}\bar{K}_{h}\left(\bm{u}, {\bm{s}_{j} \over A_{n}}\right)I(\bm{X}_{\bm{s}_{j}, A_{n}} \in [0,1]^{p}).
\end{align*}
Suppose Assumptions \ref{Ass-S}, \ref{Ass-M}, \ref{Ass-KB}, \ref{Ass-U(add)} (with $W_{\bm{s}, A_{n}} = 1$ and $\epsilon_{\bm{s}, A_{n}}$) and \ref{Ass-Rb} hold. Then uniformly for $u \in I_{h}$, 
\begin{align}\label{T0-ratio-rate}
{\tilde{n}_{0} \over n} &= f_{\bm{S}}(\bm{u})P(\bm{X}_{\bm{u}}(\bm{0}) \in [0,1]^{p}) + O\left(A_{n}^{-{\rho \over 1+\rho}}\right) + o(h)\ P_{\bm{S}}-a.s.
\end{align}
and 
\begin{align}\label{T0-normalize-rate}
{\tilde{n}_{[0,1]^{p}} - \tilde{n}_{0} \over \tilde{n}_{0}} = O_{P_{\cdot|\bm{S}}}\left(\sqrt{\log n \over nh^{d}}\right)\ P_{\bm{S}}-a.s. 
\end{align}
\end{lemma}
\begin{proof}
(Step1) In this step, we show (\ref{T0-ratio-rate})
Define $U_{\bm{s},A_{n}} = U_{\bm{s},A_{n}}(\bm{s}/A_{n})$. Recall that $\|\bm{X}_{\bm{s},A_{n}} - \bm{X}_{{\bm{s} \over A_{n}}}(\bm{s})\| \leq {1 \over A_{n}^{d}}U_{\bm{s},A_{n}}$ almost surely with $E[U_{\bm{s},A_{n}}^{\rho}]<C\leq \infty$ for some $\rho>0$. Observe that for sufficiently large $C < \infty$, 
\begin{align*}
E[I(\bm{X}_{\bm{s},A_{n}} \in [0,1]^{p})] &= E[I(\bm{X}_{\bm{s},A_{n}} \in [0,1]^{p}, \|\bm{X}_{\bm{s},A_{n}} - \bm{X}_{{\bm{s} \over A_{n}}}(\bm{s}/A_{n})\| \leq A_{n}^{-d}U_{\bm{s},A_{n}})]\\
&\begin{cases}
\geq E[I(\bm{X}_{{\bm{s} \over A_{n}}}(\bm{s}) \in [CA_{n}^{-d}U_{\bm{s},A_{n}}, 1 - CA_{n}^{-d}U_{\bm{s},A_{n}}]^{p})] \\
\leq E[I(\bm{X}_{{\bm{s} \over A_{n}}}(\bm{s}) \in [-CA_{n}^{-d}U_{\bm{s},A_{n}}, 1 + CA_{n}^{-d}U_{\bm{s},A_{n}}]^{p})].
\end{cases}
\end{align*}
Define 
\begin{align*}
B_{L} &= {1 \over n}\sum_{j=1}^{n}\bar{K}_{h}\left(\bm{u}, {\bm{s}_{j} \over A_{n}}\right)E_{\cdot|\bm{S}}\left[I(\bm{X}_{{\bm{s}_{j} \over A_{n}}}(\bm{s}_{j}) \in [CA_{n}^{-d}U_{\bm{s}_{j},A_{n}}, 1 - CA_{n}^{-d}U_{\bm{s}_{j},A_{n}}]^{p})\right], \\
B_{U} &= {1 \over n}\sum_{j=1}^{n}\bar{K}_{h}\left(\bm{u}, {\bm{s}_{j} \over A_{n}}\right)E_{\cdot|\bm{S}}\left[I(\bm{X}_{{\bm{s}_{j} \over A_{n}}}(\bm{s}_{j}) \in [-CA_{n}^{-d}U_{\bm{s},A_{n}}, 1 + CA_{n}^{-d}U_{\bm{s}_{j},A_{n}}]^{p})\right].
\end{align*}
From the definitions of $B_{L}$ and $B_{U}$, $B_{L} \leq {\tilde{n}_{0} \over n} \leq B_{U}$. Let $q \in (0,1)$ and write $B_{U} = B_{U,1} + B_{U,2}$, where 
\begin{align*}
&B_{U,1}\\
&= {1 \over n}\sum_{j=1}^{n}\bar{K}_{h}\left(\bm{u}, {\bm{s}_{j} \over A_{n}}\right)E_{\cdot|\bm{S}}\left[I(\bm{X}_{{\bm{s}_{j} \over A_{n}}}(\bm{s}_{j}) \in [-CA_{n}^{-d}U_{\bm{s}_{j},A_{n}}, 1 + CA_{n}^{-d}U_{\bm{s}_{j},A_{n}}]^{p}, U_{\bm{s}_{j},A_{n}} \leq A_{n}^{dq})\right]
\end{align*}
and $B_{U,2} = B_{U} - B_{U,1}$. Applying Lemma \ref{Masry-thm3}, we have that 
\begin{align*}
B_{U,1} &\leq {1 \over n}\sum_{j=1}^{n}\bar{K}_{h}\left(\bm{u}, {\bm{s}_{j} \over A_{n}}\right)E_{\cdot|\bm{S}}\left[I(\bm{X}_{{\bm{s}_{j} \over A_{n}}}(\bm{s}) \in [-CA_{n}^{d(q-1)}, 1 + CA_{n}^{d(q-1)}]^{p})\right]\\
&= \int_{\mathbb{R}^{p}}I(\bm{x} \in [-CA_{n}^{d(q-1)}, 1 + CA_{n}^{d(q-1)}]^{p}){1 \over n}\sum_{j=1}^{n}\bar{K}_{h}\left(\bm{u}, {\bm{s}_{j} \over A_{n}}\right)f\left({\bm{s}_{j} \over A_{n}}, \bm{x}\right)d\bm{x}\\
&= f_{\bm{S}}(\bm{u})\int_{\mathbb{R}^{p}}I(\bm{x} \in [-CA_{n}^{d(q-1)}, 1 + CA_{n}^{d(q-1)}]^{p})f\left(\bm{u}, \bm{x}\right)d\bm{x} + o(h)\\
&= f_{\bm{S}}(\bm{u})\int_{\mathbb{R}^{p}}I(\bm{x} \in [0, 1]^{p})f\left(\bm{u}, \bm{x}\right)d\bm{x} + O(A_{n}^{-d(1-q)}) + o(h),\ P_{\bm{S}}-a.s.
\end{align*}
uniformly over $I_{h}$. Moreover, applying Lemma \ref{Masry-thm3}, we also have that 
\begin{align*}
B_{U,2} &\leq {1 \over n}\sum_{j=1}^{n}\bar{K}_{h}\left(\bm{u}, {\bm{s}_{j} \over A_{n}}\right)E_{\cdot|\bm{S}}\left[I(U_{\bm{s}_{j},A_{n}} > A_{n}^{dq})\right]\\
&\leq {1 \over n}\sum_{j=1}^{n}\bar{K}_{h}\left(\bm{u}, {\bm{s}_{j} \over A_{n}}\right)E_{\cdot|\bm{S}}\left[(U_{\bm{s}_{j},A_{n}}/A_{n}^{dq})^{\rho}\right] \leq {C \over A_{n}^{dq\rho}}\ P_{\bm{S}}-a.s.
\end{align*}
Set $q = {1 \over 1+\rho}$. Then we have that 
\begin{align}\label{BU-bound}
B_{U} &\leq f_{\bm{S}}(\bm{u})P(\bm{X}_{\bm{u}}(\bm{0})\in [0,1]^{p}) + O(A_{n}^{-{d\rho \over 1+\rho}}) + o(h)
\end{align}
uniformly over $I_{h}$, $P_{\bm{S}}$-a.s. Likewise, 
\begin{align}\label{BL-bound}
B_{U} &\geq f_{\bm{S}}(\bm{u})P(\bm{X}_{\bm{u}}(\bm{0}) \in [0,1]^{p}) - O(A_{n}^{-{d\rho \over 1+\rho}}) - o(h)
\end{align}
uniformly over $I_{h}$, $P_{\bm{S}}$-a.s. (\ref{BU-bound}) and (\ref{BL-bound}) yields (\ref{T0-ratio-rate}). 

(Step2) Now we show (\ref{T0-normalize-rate}). Applying Proposition \ref{general-unif-rate} and (\ref{T0-ratio-rate}), we have that
\begin{align*}
{\tilde{n}_{[0,1]^{p}} - \tilde{n}_{0} \over \tilde{n}_{0}} &= {n \over \tilde{n}_{0}}\times {\tilde{n}_{[0,1]^{p}} - \tilde{n}_{0} \over n} = \underbrace{O(1)}_{(\ref{T0-ratio-rate})} \times \underbrace{O_{P_{\cdot|\bm{S}}}\left(\sqrt{\log n \over nh^{d}}\right)}_{\text{Proposition \ref{general-unif-rate}}} = O_{P_{\cdot|\bm{S}}}\left(\sqrt{\log n \over nh^{d}}\right). 
\end{align*}
\end{proof}
Define $n_{0} = E_{\cdot|\bm{S}}[n_{[0,1]^{p}}]$. Lemmas \ref{lemma-C1} and \ref{Masry-thm} imply that 
\begin{align*}
{n_{0} \over n} &= {\tilde{n}_{0} \over n\tilde{f}_{\bm{S}}(\bm{u})} = \underbrace{{\tilde{n}_{0} \over n}}_{=O(1)}\underbrace{\left({f_{\bm{S}}(\bm{u}) - \tilde{f}_{\bm{S}}(\bm{u}) \over f_{\bm{S}}(\bm{u})\tilde{f}_{\bm{S}}(\bm{u})}\right)}_{=o(h)} + {\tilde{n}_{0} \over nf_{\bm{S}}(\bm{u})}\\
&= P(\bm{X}_{\bm{u}}(\bm{0}) \in [0,1]^{p}) + O(A_{n}^{-{d\rho \over 1+\rho}}) + o(h)
\end{align*}
and 
\begin{align*}
{n_{[0,1]^{p}} - n_{0} \over n_{0}} &= {(\tilde{n}_{[0,1]^{p}} - \tilde{n}_{0})\tilde{f}_{\bm{S}}(\bm{u}) \over \tilde{n}_{0}\tilde{f}_{\bm{S}}(\bm{u})} = {\tilde{n}_{[0,1]^{p}} - \tilde{n}_{0} \over \tilde{n}_{0}} = O_{P_{\cdot|\bm{S}}}\left(\sqrt{{\log n \over nh^{d}}}\right).  
\end{align*}

\begin{lemma}\label{lemma-C2}
Let $\kappa_{0}(w) = \int_{\mathbb{R}}K_{h}(w,v)dv$. Suppose Assumptions \ref{Ass-S}, \ref{Ass-M}, \ref{Ass-KB}, \ref{Ass-U(add)} (with $W_{\bm{s}, A_{n}} = 1$ and $\epsilon_{\bm{s}, A_{n}}$) and \ref{Ass-Rb} hold. Then
\begin{align*}
\sup_{\bm{u} \in I_{h}, x_{\ell} \in I_{h,0}}\left|\hat{p}_{\ell}(\bm{u}, x_{\ell}) - p_{\ell}(\bm{u},x_{\ell})\right| &= O_{P_{\cdot|\bm{S}}}\left(\sqrt{ \log n \over nh^{d+1}}\right) + O\left({1 \over A_{n}^{dr}h^{p+r}}\right) + o(h), \\
\sup_{\bm{u} \in I_{h}, x_{\ell} \in [0,1]}\left|\hat{p}_{\ell}(\bm{u}, x_{\ell}) - \kappa_{0}(x_{\ell})p_{\ell}(\bm{u},x_{\ell})\right| &= O_{P_{\cdot|\bm{S}}}\left(\sqrt{ \log n \over nh^{d+1}}\right) + O\left({1 \over A_{n}^{dr}h^{p+r}}\right) + O(h), \\
\sup_{\bm{u} \in I_{h}, x_{\ell}, x_{k} \in I_{h,0}}\left|\hat{p}_{\ell}(\bm{u}, x_{\ell}, x_{k}) - p_{\ell}(\bm{u},x_{\ell}, x_{k})\right| &= O_{P_{\cdot|\bm{S}}}\left(\sqrt{ \log n \over nh^{d+2}}\right) + O\left({1 \over A_{n}^{dr}h^{p+r}}\right) + o(h), 
\end{align*}
and 
\begin{align*}
&\sup_{\bm{u} \in I_{h}, x_{\ell}, x_{k} \in [0,1]}\left|\hat{p}_{\ell}(\bm{u}, x_{\ell}, x_{k}) - \kappa_{0}(x_{\ell})\kappa_{0}(x_{k})p_{\ell}(\bm{u},x_{\ell}, x_{k})\right| = O_{P_{\cdot|\bm{S}}}\left(\sqrt{ \log n \over nh^{d+2}}\right) + O\left({1 \over A_{n}^{dr}h^{p+r}}\right) + O(h)
\end{align*}
\end{lemma}
\begin{proof}
Since the proof is similar, we only give the proof for $\hat{p}_{\ell}(\bm{u},x_{\ell})$. Define
\begin{align*}
\check{p}_{\ell}(\bm{u},x_{\ell}) &= {1 \over n_{0}}\sum_{j=1}^{n}I(\bm{X}_{\bm{s}_{j},A_{n}} \in [0,1]^{p})\bar{K}_{h}\left(\bm{u}, {\bm{s}_{j} \over A_{n}}\right)K_{h}(x_{\ell}, X_{\bm{s}_{j}, A_{n}}^{\ell}). 
\end{align*}
Applying Lemma \ref{lemma-C1}, we have that
\begin{align*}
\hat{p}_{\ell}(\bm{u},x_{\ell}) &= \left(1 + {n_{[0,1]^{p}} - n_{0} \over n_{0}}\right)^{-1}\check{p}_{\ell}(\bm{u},x_{\ell}) = \left(1 - {n_{[0,1]^{p}} - n_{0} \over n_{0}} + O_{P_{\cdot|\bm{S}}}\left(\left({n_{[0,1]^{p}} - n_{0} \over n_{0}}\right)^{2}\right)\right)\check{p}_{\ell}(\bm{u},x_{\ell})\\
&= \check{p}_{\ell}(\bm{u},x_{\ell}) + O_{P_{\cdot|\bm{S}}}\left(\sqrt{\log n \over nh^{d}}\right)
\end{align*}
uniformly for $\bm{u} \in I_{h}$ and $x_{\ell} \in [0,1]$. Applying similar arguments in the proof of Theorem \ref{unif-rate-m} to $\check{p}_{\ell}(\bm{u},x_{\ell})$, we obtain the desired result. 
\end{proof}

Decompose $\hat{m}_{\ell}(\bm{u}, \bm{x}_{\ell}) = \hat{m}_{1,\ell}(\bm{u}, \bm{x}_{\ell}) + \hat{m}_{1,\ell}(\bm{u}, \bm{x}_{\ell})$, where 
\begin{align*}
\hat{m}_{1,\ell}(\bm{u}, \bm{x}_{\ell}) &= {1 \over \hat{p}_{\ell}(\bm{u}, x_{\ell})n_{[0,1]^{p}}}\sum_{j=1}^{n}I(\bm{X}_{\bm{s}_{j},A_{n}} \in [0,1]^{p})\bar{K}_{h}\left(\bm{u}, {\bm{s}_{j} \over A_{n}}\right)K_{h}(x_{\ell}, X_{\bm{s}_{j}, A_{n}}^{\ell})\epsilon_{\bm{s}_{j},A_{n}},\\
\hat{m}_{2,\ell}(\bm{u}, \bm{x}_{\ell}) &= {1 \over \hat{p}_{\ell}(\bm{u}, x_{\ell})n_{[0,1]^{p}}}\sum_{j=1}^{n}I(\bm{X}_{\bm{s}_{j},A_{n}} \in [0,1]^{p})\bar{K}_{h}\left(\bm{u}, {\bm{s}_{j} \over A_{n}}\right)K_{h}(x_{\ell}, X_{\bm{s}_{j}, A_{n}}^{\ell})\\
&\quad \times \left(m_{0}\left({\bm{s}_{j} \over A_{n}}\right) + \sum_{k=1}^{p}m_{k}\left({\bm{s}_{j} \over A_{n}}, X_{\bm{s}_{j},A_{n}}^{k}\right)\right).
\end{align*}

\begin{lemma}\label{lemma-C3}
Suppose Assumptions \ref{Ass-S}, \ref{Ass-M}, \ref{Ass-KB}, \ref{Ass-U(add)} (with $W_{\bm{s}, A_{n}} = 1$ and $\epsilon_{\bm{s}, A_{n}}$) and \ref{Ass-Rb} hold. Then
\begin{align*}
\sup_{\bm{u} \in [0,1]^{d}, x_{\ell} \in [0,1]}\left|\hat{m}_{1,\ell}(\bm{u}, \bm{x}_{\ell})\right| &= O_{P_{\cdot|\bm{S}}}\left(\sqrt{\log n \over nh^{d+1}}\right).
\end{align*}
\end{lemma}
\begin{proof}
Replacing $n_{[0,1]^{p}}$ in the definition of $\hat{m}_{1,\ell}$ by $n_{0}$ and applying Proposition \ref{general-unif-rate} gives the desired result. 
\end{proof}

\begin{lemma}\label{lemma-C4}
Let $I_{h,0}^{c} = [0,1]\backslash I_{0,h}$ and $I_{h}^{c} = [0,1]^{d}\backslash I_{h}$. Suppose Assumptions \ref{Ass-S}, \ref{Ass-M}, \ref{Ass-KB}, \ref{Ass-U(add)} (with $W_{\bm{s}, A_{n}} = 1$ and $\epsilon_{\bm{s}, A_{n}}$) and \ref{Ass-Rb} hold. Then
\begin{align*}
\sup_{\bm{u} \in I_{h},x_{\ell} \in I_{h,0}}\left|\hat{m}_{2,j}(\bm{u},x_{\ell}) - \hat{\mu}_{\ell}(\bm{u},x_{\ell})\right| &= o_{P_{\cdot|\bm{S}}}(h^{2}),\\
\sup_{\bm{u} \in I_{h},x_{\ell} \in I_{h,0}^{c}}\left|\hat{m}_{2,j}(\bm{u},x_{\ell}) - \hat{\mu}_{\ell}(\bm{u},x_{\ell})\right| &= O_{P_{\cdot|\bm{S}}}(h^{2}),
\end{align*}
where
\begin{align*}
\hat{\mu}_{\ell}(\bm{u},x_{\ell}) &= \alpha_{0}(\bm{u}) + \alpha_{\ell}(\bm{u},x_{\ell}) + \sum_{k \neq \ell}\int_{\mathbb{R}}\alpha_{k}(\bm{u},x_{k}){\hat{p}_{\ell, k}(\bm{u},x_{\ell},x_{k}) \over \hat{p}_{\ell}(\bm{u}, x_{\ell})}dx_{\ell} + h^{2}\int_{\mathbb{R}^{p-1}}\!\!\!\!\beta(\bm{u},\bm{x}){p(\bm{u},\bm{x}) \over p_{\ell}(\bm{u},x_{\ell})}d\bm{x}_{-\ell},
\end{align*}
where 
\begin{align*}
\alpha_{0}(\bm{u}) &= m_{0}(\bm{u}) + h\sum_{i=1}^{d}\kappa_{1}(u_{i})\partial_{u_{i}}m_{0}(\bm{u})\\
&\quad + {h^{2} \over 2}\left(\sum_{i=1}^{d}\kappa_{2}(u_{i})\partial_{u_{i}}m_{0}(\bm{u}) + \sum_{1 \leq i_{1},i_{2} \leq d, i_{1}\neq i_{2}}\kappa_{1}(u_{i_{1}})\kappa_{1}(u_{i_{2}})\partial_{u_{i_{1}}u_{i_{2}}}^{2}m_{0}(\bm{u})\right),\\
\alpha_{k}(\bm{u},x_{k}) &= m_{k}(\bm{u},x_{k}) + h\left\{\sum_{i=1}^{d}\kappa_{1}(u_{i})\partial_{u_{i}}m_{k}(\bm{u},x_{k}) + {\prod_{i=1}^{d}\kappa_{0}(u_{i})\kappa_{1}(x_{k}) \over \kappa_{0}(x_{k})}\partial_{x_{k}}m_{k}(\bm{u},x_{k})\right\},\\
\beta(\bm{u},\bm{x}) &= \kappa_{2}\sum_{i=1}^{d}\partial_{u_{i}}m_{0}(\bm{u})\partial_{u_{i}}\log p(\bm{u},\bm{x}) + \kappa_{2}\sum_{k=1}^{p}\left\{\sum_{i=1}^{d}\partial_{u_{i}}m_{k}(\bm{u},x_{k})\partial_{u_{i}}\log p(\bm{u},\bm{x}) \right. \\ 
&\left. \quad + {1 \over 2}\sum_{i=1}^{d}\partial_{u_{i}u_{i}}^{2}m_{k}(\bm{u},x_{k}) + \partial_{x_{k}}m_{k}(\bm{u},x_{k})\partial_{x_{k}}\log p(\bm{u},\bm{x}) + {1 \over 2}\partial_{x_{k}x_{k}}^{2}m_{k}(\bm{u},x_{k})\right\}.
\end{align*}
Here, $\kappa_{2} = \int_{\mathbb{R}}x^{2}K(x)dx$ and $\kappa_{j}(v) = \int_{\mathbb{R}}w^{j}K_{h}(v,w)dw$ for $j = 0,1,2,$. 
\end{lemma}

\begin{proof}
Although the detailed proof is lengthy and involved, we can obtain the desired result by applying almost the same strategy in the proof of Lemma C.4 in \cite{Vo12}. Therefore, we omit the proof. We note that Lemmas \ref{lemma-C1}, \ref{Masry-thm2} and \ref{Masry-thm3} in this paper, which correspond to Lemmas C.1, B.1 and B.2 in \cite{Vo12}, and the similar argument in the proof of Proposition \ref{general-unif-rate} and Theorem \ref{unif-rate-m} in this paper are applied for the proof. See also \cite{MaLiNi99} for the original idea of the proof.
\end{proof}

\begin{lemma}\label{lemma-C5}
Suppose Assumptions \ref{Ass-S}, \ref{Ass-M}, \ref{Ass-KB}, \ref{Ass-U(add)} (with $W_{\bm{s}, A_{n}} = 1$ and $\epsilon_{\bm{s}, A_{n}}$) and \ref{Ass-Rb} hold. Then
\begin{align*}
\sup_{\bm{u} \in I_{h}}\left|\tilde{m}_{0}(\bm{u}) - m_{0}(\bm{u})\right| &= O_{P_{\cdot|\bm{S}}}\left(\sqrt{\log n \over nh^{d}} + h^{2}\right). 
\end{align*}
\end{lemma}
\begin{proof}
Replacing $n_{[0,1]^{d}}$ by $n_{0}$ in the definition of $\tilde{m}_{0}$ and applying similar arguments in the proof of Theorem \ref{unif-rate-m}, we obtain the desired result.
\end{proof}

\begin{lemma}\label{proof-CARMA-example}
Suppose that $E[|L([0,1]^{d})|^{q'}]<\infty$ for $1 \leq q' \leq q$ where $q$ is some even integer. Then a univariate CARMA-type random field with a kernel function $g$ of the form (\ref{exp-decay-kernel}) satisfies (\ref{CARMA-tail-decay}).
\end{lemma}
\begin{proof}
Observe that 
\begin{align*}
X_{\bm{s},A_{n}} &= \int_{\mathbb{R}^{d}}g\left({\bm{s} \over A_{n}}, \|\bm{s} - \bm{v}\|\right)L(d\bm{v})\\
&= \int_{\mathbb{R}^{d}}g\left({\bm{s} \over A_{n}}, \|\bm{s} - \bm{v}\|\right)\iota(\|\bm{s} - \bm{v}\|:A_{2,n})L(d\bm{v}) + \int_{\mathbb{R}^{d}}g\left({\bm{s} \over A_{n}}, \|\bm{s} - \bm{v}\|\right)(1 - \iota(\|\bm{s} - \bm{v}\|:A_{2,n}))L(d\bm{v})\\
&=: X_{\bm{s}, A_{n}: A_{2,n}} + \epsilon_{\bm{s}, A_{n}:A_{2,n}}. 
\end{align*}
Note that $X_{\bm{s}, A_{n}: A_{2,n}}$ is $A_{2,n}$-dependent from the definition of the function $\iota$. Then we have that 
\begin{align*}
&E[|\epsilon_{\bm{s},A_{n}:A_{2,n}}|^{q}] \lesssim \int_{\mathbb{R}^{d}}e^{-qc_{0}\|\bm{u}\|}\left(1- \iota\left(\|\bm{u}\| : A_{2,n}\right)\right)^{q}d\bm{u}\\
& \lesssim \int_{\|\bm{u}\| \geq A_{2,n}/2}e^{-qc_{0}\|\bm{u}\|}\left|1 + {2 \over A_{2,n}}(\|\bm{u}\| - A_{2,n})\right|^{q}d\bm{u} \lesssim \int_{\|\bm{u}\| \geq A_{2,n}/2}\!\!\!\!\!\!\!e^{-qr_{0}\|\bm{u}\|}\left|1 + 2{\|\bm{u}\| \over A_{2,n}}\right|^{q}d\bm{u}\\
&\lesssim 2^{q-1}\int_{\|\bm{u}\| \geq A_{2,n}/2}e^{-qr_{0}\|\bm{u}\|}\left(1 + {2^{q}\|\bm{u}\|^{q} \over A_{2,n}^{q}}\right)d\bm{u} \lesssim \int_{A_{2,n}/2}^{\infty}e^{-qr_{0}t}\left(1 + {2^{q}t^{q} \over A_{2,n}^{q}}\right)t^{d-1}dt\\
&\lesssim e^{-{qr_{0}A_{2,n} \over 2}}\left(1 + {(A_{2,n}/2)^{q} \over A_{2,n}^{q}}\right)(A_{2,n}/2)^{d-1} \lesssim A_{2,n}^{d-1}e^{-{qr_{0}A_{2,n} \over 2}}.
\end{align*}
For the first inequality, we used the finiteness of the $E[|L([0,1]^{d})|^{q}]$ (see also \cite{BrMa17} for computation of moments of a L\'evy-driven MA process). For the second inequality, we used the definition of $\iota$. Hence we obtain the desired result.
\end{proof}

\begin{lemma}\label{LS-density}
Let $c$ be a positive constant and $g$ be a bounded function such that $\int_{\mathbb{R}^{d}}|g(\bm{v})|d\bm{v}<\infty$ and $\lim_{\|\bm{v}\| \to \infty}g(\bm{v}) = 0$. Suppose that the L\'evy random measure $L$ of a random field $X(\bm{s}) = \int_{\mathbb{R}}g(\bm{s}-\bm{v})L(d\bm{s})$ is purely non-Gaussian with L\'evy density
\begin{align*}
\nu_{0}(x) &= {c_{\beta} \over |x|^{1+\beta}}g_{0}(x),\ x \neq 0
\end{align*} 
where $\beta \in (0,1)$, $c_{\beta}>0$ and $g_{0}$ is a positive, continuous and bounded function on $\mathbb{R}$ with $\lim_{|x| \downarrow 0}g_{0}(x) = 1$. Then there exists a constant $\tilde{C}>0$ such that
\begin{align*}
u^{2}\int_{\mathbb{R}^{d}}g^{2}(\bm{v})\int_{|x| \leq {1 \over |g(\bm{v})||u|}}x^{2}\nu_{0}(x)dxd\bm{v} \geq \tilde{C}|u|^{2-\alpha}
\end{align*}  
for any $|u| \geq c_{0}>1$, where $\alpha = 2-\beta \in (0,2)$.
\end{lemma}

\begin{proof}
From the assumption on $g$, there exists a bounded set $S_{c_{0}} \subset \mathbb{R}^{d}$ such that $S_{c_{0}} \subset \{\bm{v} \in \mathbb{R}^{d} : {1 \over c_{0}} \leq |g(\bm{v})|\}$.
Then we have that
\begin{align*}
&u^{2}\int_{\mathbb{R}^{d}}g^{2}(\bm{v})\int_{|x| \leq {1 \over |g(\bm{v})||u|}}x^{2}\nu_{0}(x)dxd\bm{v}\\
&= u^{2}\int_{ {1 \over |u||g(\bm{v})|} \leq 1}\!\!\!\! g^{2}(\bm{v})\int_{|x| \leq {1 \over |g(\bm{v})||u|}}\!\!\!\!x^{2}\nu_{0}(x)dxd\bm{v} + u^{2}\int_{ {1 \over |u||g(\bm{v})|} >1}\!\!\!\!\!\!g^{2}(\bm{v})\int_{|x| \leq {1 \over |g(\bm{v})||u|}}\!\!\!\!\!\!\!\!x^{2}\nu_{0}(x)dxd\bm{v}\\
& \gtrsim u^{2}\int_{ {1 \over |u||g(\bm{v})|} \leq 1}\!\!\!\!g^{2}(\bm{v})\int_{|x| \leq {1 \over |g(\bm{v})||u|}}\!\!\!\! |x|^{1-\beta}g_{0}(x)dxd\bm{v} + u^{2}\int_{ {1 \over |u||g(\bm{v})|} >1}\!\!\!\! g^{2}(\bm{v})\int_{|x| \leq 1}\!\!\!\! x^{2}\nu_{0}(x)dxd\bm{v}\\
& \gtrsim  u^{2}\int_{ {1 \over |u||g(\bm{v})|} \leq 1}g^{2}(\bm{v})|u|^{\beta-2}|g(\bm{v})|^{\beta -2}d\bm{v} + u^{2}\int_{ {1 \over |u||g(\bm{v})|} >1}g^{2}(\bm{v})d\bm{v}\\
&\geq |u|^{\beta}\int_{ {1 \over |u||g(\bm{v})|} \leq 1}|g(\bm{v})|^{\beta}d\bm{v} \geq |u|^{\beta}\int_{{1 \over c_{0}} \leq |g(\bm{v})|}|g(\bm{v})|^{\beta}d\bm{v} \geq |u|^{\beta}\int_{S_{c_{0}}}|g(\bm{v})|^{\beta}d\bm{v} \gtrsim |u|^{\beta}.
\end{align*}
Then we obtain the desired result.
\end{proof}

\section{Technical lemmas}

We refer to the following lemmas without those proofs. 
\begin{lemma}[Lemmas A.1 and 5.1 in \cite{La03b}]\label{n summands}
Let $\mathcal{I}_{n} = \{\bm{i} \in \mathbb{Z}^{d}: (\bm{i} +(0,1]^{d})\cap  R_{n} \neq \emptyset \}$. Then we have that 
\[
P_{\bm{S}}\left(\sum_{j=1}^{n}1\{A_{n} \bm{S}_{0,j} \in (\bm{i} +(0,1]^{d})\cap  R_{n} > 2(\log n + nA_{n}^{-d})\ \text{for some $\bm{i} \in \mathcal{I}_{n}$, i.o.}\right) = 0
\]
and 
\[
P_{\bm{S}}\left(\sum_{j=1}^{n}1\{A_{n} \bm{S}_{0,j} \in \Gamma_{n}(\bm{\ell}; \bm{\epsilon})\} > CA_{1,n}^{q(\bm{\epsilon})}A_{2,n}^{d-q(\bm{\epsilon})}nA_{n}^{-d}\ \text{for some $\bm{\ell} \in L_{1,n}$, i.o.}\right) = 0
\]
for any $\bm{\epsilon} \in \{1,2\}^{d}$, where $C>0$ is a sufficiently large constant. 
\end{lemma}

\begin{remark}
Lemma \ref{n summands} implies that each $\Gamma_{n}(\bm{\ell};\bm{\epsilon})$ contains at most $CA_{1,n}^{q(\bm{\epsilon})}A_{2,n}^{d-q(\bm{\epsilon})}nA_{n}^{-d}$ samples $P_{\bm{S}}$-almost surely. 
\end{remark}

We may define the $\beta$-mixing coefficients for any probability measure $Q$ on a product measure space $(\Omega_{1} \times \Omega_{2}, \Sigma_{1} \times \Sigma_{2})$ as follows: 
\begin{definition}[Definition 2.5 in \cite{Yu94}]\label{beta-mixing-meas}
Suppose that $Q_{1}$ and $Q_{2}$ are the marginal probability measures of $Q$ on $(\Omega_{1}, \Sigma_{1})$ and $(\Omega_{2}, \Sigma_{2})$, respectively. Then we define
\[
\beta(\Sigma_{1}, \Sigma_{2},Q) = P\sup\{|Q(B|\Sigma_{1}) - Q_{2}(B)|: B \in \Sigma_{2}\}.
\] 
\end{definition}

\begin{lemma}[Corollary 2.7 in \cite{Yu94}]\label{indep_lemma}
Let $m \geq 1$ and let $Q$ be a probability measure on a product space $(\prod_{i=1}^{m}\Omega_{i}, \prod_{i=1}^{m}\Sigma_{i})$ with marginal measures $Q_{i}$ on $(\Omega_{i}, \Sigma_{i})$. Suppose that $h$ is a bounded measurable function on the product probability space such that $|h|\leq M_{h}<\infty$. Let $Q_{a}^{b}$ (with $1 \leq a \leq b$) be the marginal measure on $(\prod_{i=a}^{b}\Omega_{i}, \prod_{i=a}^{b}\Sigma_{i})$. Write 
\begin{align*}
\beta(Q) = \sup_{1 \leq i \leq m-1}\beta\left(\prod_{j=1}^{i}\Sigma_{j}, \Sigma_{i+1}, Q_{1}^{i+1}\right).
\end{align*}
Suppose that, for all $1 \leq k \leq m-1$, 
\begin{align}\label{prod_bound}
\|Q - Q_{1}^{k}\times Q_{k+1}^{m}\| \leq \beta(Q),
\end{align}
where $Q_{1}^{k}\times Q_{k+1}^{m}$ is a product measure and $\| \cdot \|$ is $1/2$ of the total variation norm. Then
\[
| Qh -  Ph| \leq M_{h}(m-1)\beta(Q).
\]
where $P=\prod_{i=1}^{m}Q_{i}$, $Qh=\int hdQ$, and $Ph=\int hdP$. 
\end{lemma}

\begin{remark}
Lemma \ref{indep_lemma} is a key tool to construct independent blocks for $\beta$-mixing sequence. Note that Lemma \ref{indep_lemma} holds for each \textit{finite} $n$. 
\end{remark}

\begin{assumption}\label{Ass-KD}
\item[(KD1)] The kernel $\bar{K}: \mathbb{R}^{d} \to [0,\infty)$ is bounded and has compact support $[-C,C]^{d}$. Moreover, 
\begin{align*}
\int_{[-C,C]^{d}}\bar{K}(\bm{x})dx &= 1,\ \int_{[-C,C]^{d}}\bm{x}^{\bm{\alpha}}\bar{K}(\bm{x})dx = 0,\ \text{for any $\bm{\alpha} \in \mathbb{Z}^{d}$ with}\ |\bm{\alpha}| = 1,
\end{align*}
and $|\bar{K}(\bm{u}) - \bar{K}(\bm{v})| \leq C\|\bm{u} - \bm{v}\|$.
\item[(KD2)] For any $\bm{\alpha} \in \mathbb{Z}^{d}$ with $|\bm{\alpha}| = 1,2$, $\partial^{\bm{\alpha}}f_{\bm{S}}(\bm{s})$ exist and continuous on $(0,1)^{d}$. 
\end{assumption}

Define $\hat{f}_{\bm{S}}(\bm{u}) = {1 \over nh^{d}}\sum_{j=1}^{n}\bar{K}_{h}\left({\bm{u} - \bm{S}_{0,j}}\right)$.

\begin{lemma}[Theorem 2 in \cite{Ma96}]\label{Masry-thm}
Under Assumption \ref{Ass-KD} and $h \to 0$ such that $nh^{d}/(\log n) \to \infty$ as $n \to \infty$, we have that
\begin{align*}
\sup_{\bm{u} \in [0,1]^{d}}\left|\hat{f}_{\bm{S}}(\bm{u}) - f_{\bm{S}}(\bm{u})\right| = O\left(\sqrt{\log n \over nh^{d}} + h^{2}\right)\ P_{\bm{S}}-a.s.
\end{align*}
\end{lemma}

Similar arguments of the proof of Lemma \ref{Masry-thm} yields the following Lemmas \ref{Masry-thm2} and \ref{Masry-thm3}.
\begin{lemma}\label{Masry-thm2}
Under Assumption \ref{Ass-KD} and $h \to 0$ such that $nh^{d}/(\log n) \to \infty$ as $n \to \infty$, we have that
\begin{align*}
&\sup_{\bm{u} \in I_{h}}\left|{1 \over nh^{d}}\sum_{j=1}^{n}\bar{K}_{h}(\bm{u} - \bm{S}_{0,j})\left({\bm{u} - \bm{S}_{0,j} \over h}\right)^{\bm{k}} - {1 \over h^{d}}\int_{\mathbb{R}^{d}}\bar{K}_{h}(\bm{u} - \bm{w})\left({\bm{u} - \bm{w} \over h}\right)^{\bm{k}}f_{\bm{S}}(\bm{w})d\bm{w}\right|\\
&\quad = O\left(\sqrt{\log n \over nh^{d}}\right)\ P_{\bm{S}}-a.s.
\end{align*}
for any $\bm{k} \in \mathbb{Z}^{d}$ with $|\bm{k}| = 0,1,2$, where $\bm{x}^{\bm{k}} = \prod_{\ell=1}^{d}x_{\ell}^{k_{\ell}}$. 
\end{lemma}

\begin{lemma}\label{Masry-thm3}
Let $g: [0,1]^{d} \times \mathbb{R}^{p} \to \mathbb{R}$, $(\bm{u},\bm{x}) \mapsto g(\bm{u},\bm{x})$ be continuously partially differentiable w.r.t. $\bm{u}$. Under Assumption \ref{Ass-KD} and $h \to 0$ such that $nh^{d}/(\log n) \to \infty$ as $n \to \infty$, we have that
\begin{align*}
\sup_{\bm{u} \in I_{h},\bm{x} \in S_{c}}\left|{1 \over nh^{d}}\sum_{j=1}^{n}\bar{K}_{h}^{m}(\bm{u} - \bm{S}_{0,j})g(\bm{S}_{0,j},\bm{x}) - \bar{\kappa}_{m}f_{\bm{S}}(\bm{u})g(\bm{u},\bm{x})\right| &= O\left(\sqrt{\log n \over nh^{d}}\right) + o(h)
\end{align*}
$P_{\bm{S}}$-a.s. for $m=1,2$, where $\bar{\kappa}_{m} = \int_{\mathbb{R}^{d}}\bar{K}^{m}(\bm{x})d\bm{x}$.
\end{lemma}

\begin{lemma}[Bernstein's inequality]\label{Bernstein}
Let $X_{1},\hdots, X_{n}$ be independent zero-mean random variables. Suppose that  $\max_{1 \leq i \leq n}|X_{i}|\leq M<\infty$ a.s. Then, for all $t>0$,
\begin{align*}
P\left(\sum _{i=1}^{n}X_{i}\geq t\right)\leq \exp \left(-{{t^{2} \over 2} \over \sum_{j=1}^{n}E[X_{j}^{2}] + {Mt \over 3}}\right). 
\end{align*}
\end{lemma}

\section{Multivariate L\'evy-driven moving average random fields}\label{Multi-LevyMA}

In this section, we discuss multivariate extension of univariate L\'evy-driven moving average random fields and give examples of multivariate locally stationary random fields.
Let $M_{p}(\mathbb{R})$ denote the space of all real $p\times p$ matrices and let $\text{tr}(A)$ denote the trace of the matrix $A$. 

Let $\bm{L}=\{L(A) = (L_{1}(A),\hdots,L_{p}(A))': A \in \mathcal{B}(\mathbb{R}^{d})\}$ be an $\mathbb{R}^{p}$-valued infinitely divisible random measure on some probability space $(\Omega, \mathcal{A}, P)$, i.e., a random measure such that 
\begin{itemize}
\item[1.] for each sequence $(E_{m})_{m \in \mathbb{N}}$ of disjoint sets in $\mathcal{B}(\mathbb{R}^{d})$ it holds
\begin{itemize}
\item[(a)] $\bm{L}(\cup_{m=1}^{\infty}E_{m}) = \sum_{m=1}^{\infty}\bm{L}(E_{m})$ a.s., whenever $\cup_{m=1}^{\infty}E_{m} \in \mathcal{B}(\mathbb{R}^{d})$,
\item[(b)] $(\bm{L}(E_{m}))_{m \in \mathbb{N}}$ is a sequence of independent random vectors. 
\end{itemize}
\item[2.] the random variable $\bm{L}(A)$ has an infinitely divisible distribution for any $A \in \mathcal{B}(\mathbb{R}^{d})$. 
\end{itemize}
The characteristic function of $\bm{L}(A)$ which will be denoted by $\varphi_{\bm{L}(A)}(\bm{t})$ ($\bm{t} = (t_{1},\hdots,t_{p})' \in \mathbb{R}^{p}$), has a L\'evy-Khintchine representation of the form $\varphi_{L(A)}(\bm{t}) = \exp\left(|A|\psi(\bm{t})\right)$ with 
\begin{align*}
\psi(\bm{t}) &= i\langle \bm{t}, \bm{\gamma}_{0}\rangle - {1 \over 2}\bm{t}'\Sigma_{0}\bm{t} + \int_{\mathbb{R}^{p}}\left\{e^{i\langle \bm{t}, \bm{x} \rangle }-1-i\langle \bm{t}, \bm{x} \rangle I(\|\bm{x}\| \leq 1)\right\}\nu_{0}(\bm{x})d\bm{x}
\end{align*}
where $i = \sqrt{-1}$, $\bm{\gamma}_{0} \in \mathbb{R}^{p}$, $\Sigma_{0}$ is a $p\times p$ positive definite matrix, $\nu_{0}$ is a L\'evy density with $\int_{\mathbb{R}^{p}}\min\{1,\|\bm{x}\|^{2}\}\nu_{0}(\bm{x})d\bm{x}<\infty$ and $|A|$ is the Lebesgue measure of $A$. The triplet $(\bm{\gamma}_{0}, \Sigma_{0}, \nu_{0})$ is called the L\'evy characteristic of $\bm{L}$ and it uniquely determines the distribution of the random measure $\bm{L}$.

Consider the stochastic process $\bm{X} = \{\bm{X}(\bm{s}) = (X_{1}(\bm{s}), \hdots, X_{p}(\bm{s}))': \bm{s} \in \mathbb{R}^{d}\}$ given by
\begin{align}\label{LevyMARF}
\bm{X}(\bm{s}) &= \int_{\mathbb{R}^{d}}g(\bm{s}, \bm{v})\bm{L}(d\bm{v}),
\end{align}
where $g: \mathbb{R}^{d} \times \mathbb{R}^{d} \to M_{p}(\mathbb{R})$ is a measurable function and $\bm{L}$ is a $p$-dimensional L\'evy random measure. Assume that 
\begin{align}
\int_{\mathbb{R}^{d}}&\text{tr}\left(g(\bm{s}, \bm{v})\Sigma_{0}g(\bm{s}, \bm{v})'\right)d\bm{v}<\infty, \label{A1}\\
\int_{\mathbb{R}^{d}}\int_{\mathbb{R}^{p}}&\min\{\|g(\bm{s}, \bm{v})\bm{x}\|^{2}, 1\}\nu_{0}(\bm{x})d\bm{x}d\bm{v}<\infty,\ \text{and} \label{A2} 
\end{align}
\begin{align}
\int_{\mathbb{R}^{d}}\left\|g(\bm{s}, \bm{v})\bm{\gamma}_{0} + \int_{\mathbb{R}^{p}}g(\bm{s}, \bm{v})\bm{x}\left(I(\|g(\bm{s}, \bm{v})\bm{x}\| \leq 1) - I(\|\bm{x}\| \leq 1)\right)\nu_{0}(\bm{x})d\bm{x}\right\|d\bm{v}<\infty. \label{A3}
\end{align}

Then the law of $\bm{X}(\bm{s})$ for all $\bm{s} \in \mathbb{R}^{d}$ is infinitely divisible with characteristic function
\begin{align*}
E[e^{i\langle \bm{t}, \bm{X}(\bm{s}) \rangle}] &= \exp\left\{i \langle \bm{t}, \int_{\mathbb{R}^{d}}g(\bm{s}, \bm{v})\bm{\gamma}_{0}d\bm{v} \rangle - {1 \over 2}\bm{t}'\left(\int_{\mathbb{R}^{d}}g(\bm{s}, \bm{v})\Sigma_{0}g(\bm{s}, \bm{v})'d\bm{v}\right)\bm{t} \right. \\ 
&\left. +  \int_{\mathbb{R}^{d}}\left(\int_{\mathbb{R}^{p}}(e^{i\langle \bm{t}, g(\bm{s}, \bm{v})\bm{x}\rangle } - 1 - i\langle \bm{t}, g(\bm{s}, \bm{v})\bm{x} \rangle I(\|g(\bm{s}, \bm{v})\bm{x}\| \leq 1))v_{0}(\bm{x})d\bm{x}\right)d\bm{v}\right\}.
\end{align*}

Therefore, the L\'evy characteristics ($\bm{\gamma}_{\bm{X}(\bm{s})}$, $\Sigma_{\bm{X}(\bm{s})}$, $\nu_{\bm{X}(\bm{s})}$) of $\bm{X}(\bm{s})$  are given by
\begin{align*}
\bm{\gamma}_{\bm{X}(\bm{s})} &= \int_{\mathbb{R}^{d}}g(\bm{s}, \bm{v})\bm{\gamma}_{0}d\bm{v}\\
&\quad + \int_{\mathbb{R}^{d}}\left(\int_{\mathbb{R}^{p}}g(\bm{s}, \bm{v})\bm{x}\left(I(\|g(\bm{s}, \bm{v})\bm{x}\| \leq 1) - I(\|\bm{x}\| \leq 1)\right)\nu_{0}(\bm{x})d\bm{x}\right)d\bm{v}, \\
\Sigma_{\bm{X}(\bm{s})} &= \int_{\mathbb{R}^{d}}g(\bm{s}, \bm{v})\Sigma_{0}g(\bm{s}, \bm{v})'d\bm{v}, \\
\nu_{\bm{X}(\bm{s})}(B) &= \int_{\mathbb{R}^{d}}\left(\int_{\mathbb{R}^{p}}I(g(\bm{s}, \bm{v})\bm{x} \in B)\nu_{0}(\bm{x})d\bm{x}\right)d\bm{v},\ B \in \mathcal{B}(\mathbb{R}^{p}).
\end{align*}
These results follow from Theorem 3.1, Proposition 2.17 and Corollary 2.19 in \cite{Sa05}. 
\begin{remark}
Assumptions (\ref{A1})-(\ref{A3}) are necessary and sufficient conditions for the existence of the stochastic integral (\ref{LevyMARF}). We refer to \cite{RaRo89} and \cite{Sa05} for details.  
\end{remark}
If $g(\bm{s}, \bm{v}) = g(\|\bm{s} - \bm{v}\|)$, that is, $\bm{X}$ is a strictly stationary isotropic random field, then we have that
\begin{align*} 
E[e^{i\langle \bm{t}, \bm{X}(\bm{s}) \rangle}]
&= \exp\left\{i \langle \bm{t}, \int_{\mathbb{R}^{d}}g(\|\bm{v}\|)\bm{\gamma}_{0}d\bm{v} \rangle - {1 \over 2}\bm{t}'\left(\int_{\mathbb{R}^{d}}g(\|\bm{v}\|)\Sigma_{0}g(\|\bm{v}\|)'d\bm{v}\right)\bm{t} \right. \\ 
&\left. \quad +  \int_{\mathbb{R}^{d}}\left(\int_{\mathbb{R}^{p}}(e^{i\langle \bm{t}, g(\|\bm{v}\|)\bm{x}\rangle } - 1 - i\langle \bm{t}, g(\| \bm{v}\|)\bm{x} \rangle I(\|g(\|\bm{v}\|)\bm{x}\| \leq 1))v_{0}(\bm{x})d\bm{x}\right)d\bm{v}\right\}.
\end{align*}

Then $\bm{X}$ has a density function when $\bm{L}$ is Gaussian, that is, $(\bm{\gamma}_{0}, \Sigma_{0}, \nu_{0}) = (\bm{\gamma}_{0}, \Sigma_{0}, 0)$. We can also give a sufficient condition for the existence of the density function of $\bm{X}$ when $\bm{L}$ is purely non-Gaussian, that is, $(\bm{\gamma}_{0}, \Sigma_{0}, \nu_{0}) = (\bm{\gamma}_{0}, 0 , \nu_{0})$.

\begin{lemma}\label{density-Levy-MARF}
Suppose that there exist an $\alpha \in (0,2)$ and a constant $C>0$ such that
\begin{align*}
\int_{\mathbb{R}^{d}}\int_{\mathbb{R}^{p}}\left|\langle \bm{u}, g(\|\bm{v}\|)\bm{x} \rangle\right|^{2}I\left(\left|\langle \bm{u}, g(\|\bm{v}\|)\bm{x} \rangle\right| \leq 1\right)\nu_{0}(\bm{x})d\bm{x}d\bm{v} \geq C\|\bm{u}\|^{2-\alpha}
\end{align*}
for any vector $\bm{u}$ with $\|\bm{u}\| \geq c_{0}>1$. Then $\bm{X}$ with a purely non-Gaussian L\'evy random field $\bm{L}$ has a bounded continuous, infinitely often differentiable functions whose derivatives are bounded.
\end{lemma}
\begin{proof}
From Proposition 0.2 in \cite{Pi96}, it is sufficient to show that $\int \|\bm{u}\|^{k}|\varphi(\bm{u})|d\bm{u}<\infty$ for any non-negative integer $k$, where $\varphi(\bm{u})$ denotes the characteristic function of $\bm{X}(\bm{s})$. Observe that 
\begin{align*}
|\varphi(\bm{t})| &= \left|\exp\left\{\int_{\mathbb{R}^{d}}\left(\int_{\mathbb{R}^{p}}(e^{i\langle \bm{t}, g(\|\bm{v}\|)\bm{x}\rangle } - 1 - i\langle \bm{t}, g(\| \bm{v}\|)\bm{x} \rangle I(\|g(\|\bm{v}\|)\bm{x}\| \leq 1))v_{0}(\bm{x})d\bm{x}\right)d\bm{v}\right\}\right|\\
&= \left(\exp\left\{\int_{\mathbb{R}^{d}}\int_{\mathbb{R}^{p}}\left(e^{i\langle \bm{t}, g(\|\bm{v}\|)\bm{x}\rangle} - e^{-i\langle \bm{t}, g(\|\bm{v}\|)\bm{x}\rangle}-2\right)\nu_{0}(\bm{x})d\bm{x}d\bm{v}\right\}\right)^{1/2}\\
&=  \exp\left\{\int_{\mathbb{R}^{d}}\int_{\mathbb{R}^{p}}\left(\cos(\langle \bm{t}, g(\|\bm{v}\|)\bm{x}\rangle )-1\right)\nu_{0}(\bm{x})d\bm{x}d\bm{v}\right\}\\
&\leq \exp\left\{\int_{\mathbb{R}^{d}}\int_{\mathbb{R}^{p}}\left(\cos(\langle \bm{t}, g(\|\bm{v}\|)\bm{x}\rangle )-1\right)I\left(\left|\langle \bm{t}, g(\|\bm{v}\|)\bm{x} \rangle\right| \leq 1\right)\nu_{0}(\bm{x})d\bm{x}d\bm{v}\right\}.
\end{align*}
Then using the inequality $1- \cos(z) \geq 2(z/\pi)^{2}$ for $|z| \leq \pi$, we have that
\begin{align*}
|\varphi(\bm{u})| &\leq \exp\left\{-C\int_{\mathbb{R}^{d}}\int_{\mathbb{R}^{p}}\left|\langle \bm{u}, g(\|\bm{v}\|)\bm{x} \rangle\right|^{2}I\left(\left|\langle \bm{u}, g(\|\bm{v}\|)\bm{x} \rangle\right| \leq 1\right)\nu_{0}(\bm{x})d\bm{x}d\bm{v}\right\}\\
&\leq \exp(-\bar{C}\|\bm{u}\|^{2-\alpha})
\end{align*}
for some constants $C, \bar{C}>0$. Therefore, we complete the proof.
\end{proof}

For $g(\bm{u}, \bm{s}) = g(\bm{u}, \|\bm{s}\|): [0,1]^{d} \times \mathbb{R}^{d} \to M_{p}(\mathbb{R})$ with bounded components, assume that $|g_{j,k}(\bm{u}, \cdot) - g_{j,k}(\bm{v}, \cdot)|\leq C\|\bm{s} - \bm{v}\|\bar{g}_{j,k}(\cdot)$ with $C<\infty$ and for any $\bm{u} \in [0,1]^{d}$, 
\[
\max_{1 \leq j,k \leq p}\int_{\mathbb{R}^{d}}(|g_{j,k}(\bm{u},\|\bm{s}\|)| + |\bar{g}_{j,k}(\bm{s})| + g^{2}_{j,k}(\bm{u},\|\bm{s}\|) +  \bar{g}^{2}_{j,k}(\bm{s}))d\bm{s}<\infty.
\]
Consider the processes 
\begin{align*}
\bm{X}_{\bm{s},A_{n}} &= \int_{\mathbb{R}^{d}}g\left({\bm{s} \over A_{n}}, \|\bm{s} - \bm{v}\|\right)\bm{L}(d\bm{v}),\ \bm{X}_{\bm{u}}(\bm{s}) = \int_{\mathbb{R}^{d}}g\left(\bm{u}, \|\bm{s} - \bm{v}\|\right)\bm{L}(d\bm{v})\\
\bm{X}_{\bm{u}}(\bm{s}:A_{2,n}) &= \int_{\mathbb{R}^{d}}g\left(\bm{u}, \|\bm{s} - \bm{v}\|\right)\iota(\|\bm{s}-\bm{v}\|:A_{2,n})\bm{L}(d\bm{v}).
\end{align*}
Note that $\bm{X}_{\bm{u}}(\bm{s})$ and $\bm{X}_{\bm{u}}(\bm{s}:A_{2,n})$ are strictly stationary random fields for each $\bm{u}$. In particular, $\bm{X}_{\bm{u}}(\bm{s}:A_{2,n})$ is $A_{2,n}$-dependent.  Assume that $\max_{1 \leq k \leq p}E[|L_{k}(A)|^{q}]<\infty$ for $1 \leq q\leq q_{0}$ where $q_{0}$ is an integer such that $q_{0} \geq 2$.  Observe that 
\begin{align*}
\|\bm{X}_{\bm{s}, A_{n}} - \bm{X}_{\bm{u}}(\bm{s}:A_{2,n})\| &\leq  \left\|\bm{X}_{\bm{s},A_{n}} - \bm{X}_{\bm{u}}(\bm{s})\right\| + \left\|\bm{X}_{\bm{u}}(\bm{s}) - \bm{X}_{\bm{u}}(\bm{s}:A_{2,n})\right\|\\
&\leq C\left\|{\bm{s} \over A_{n}} - \bm{u}\right\|\sum_{j=1}^{p}\int_{\mathbb{R}^{d}}\sum_{k=1}^{p}\left|\bar{g}_{j,k}\left(\|\bm{s} - \bm{v}\|\right) \right||L_{j}(d\bm{v})|\\
& + {1 \over A_{n}^{d}}\sum_{j=1}^{p}\int_{\mathbb{R}^{d}}\sum_{k=1}^{p}A_{n}^{d}|g_{j,k}\left(\bm{u}, \|\bm{s} - \bm{v}\|\right)|(1 - \iota(\|\bm{s} - \bm{v}\|:A_{2,n}))|L_{j}(d\bm{v})| \\
&\leq \left(\left\|{\bm{s} \over A_{n}} - \bm{u}\right\| + {1 \over A_{n}^{d}}\right)U_{\bm{s},A_{n}}(\bm{u}),
\end{align*}
where 
\begin{align*}
U_{\bm{s},A_{n}}(\bm{u}) &= \sum_{j=1}^{p}\int_{\mathbb{R}^{d}}\sum_{k=1}^{p}\left(C\left|\bar{g}_{j,k}\left(\|\bm{s} - \bm{v}\|\right) \right| + A_{n}^{d}\underbrace{|g_{j,k}\left(\bm{u}, \|\bm{s} - \bm{v}\|\right)|(1 - \iota(\|\bm{s} - \bm{v}\|:A_{2,n}))}_{=: g_{\bm{u}, j,k}(\|\bm{s} - \bm{v}\|:A_{2,n})}\right)|L_{j}(d\bm{v})|.
\end{align*}
Assume that 
\begin{align*}
\sup_{n \geq 1}\sup_{\bm{u} \in [0,1]^{d}}\max_{1\leq j,k \leq p}\int_{\mathbb{R}^{d}}(A_{n}^{d}g_{\bm{u}, j,k}(\|\bm{s}\|:A_{2,n}) + A_{n}^{2d}g_{\bm{u}, j,k}^{2}(\|\bm{s}\|:A_{2,n}))d\bm{s}<\infty. 
\end{align*}
Then we can show that $E[|U_{\bm{s},A_{n}}(\bm{u})|^{2}]<\infty$ and Condition (Ma2) in Assumption \ref{Ass-M(add)} is satisfied. Let $\{c_{j,k, \ell}\}_{1 \leq \ell \leq p_{j,k}, 1 \leq j,k \leq p}$ are positive constants and $\{r_{j,k,\ell}(\cdot)\}_{1 \leq \ell \leq p_{j,k}, 1 \leq j,k \leq p}$ are continuous functions on $[0,1]^{d}$ such that $|r_{j,k,\ell}(\bm{u}) - r_{j,k,\ell}(\bm{v})| \leq C\|\bm{u} - \bm{v}\|$ for $1 \leq \ell\leq p_{j,k}$, $1 \leq j,k \leq p_{1}$ with $C<\infty$. Consider the function
\begin{align*}
g(\bm{u},\|\bm{s}\|) &= \left(\sum_{\ell=1}^{p_{j,k}}r_{j,k, \ell}\left(\bm{u}\right)e^{-c_{j,k,\ell}\|\bm{s}\|}\right)_{1 \leq j,k \leq p}.
\end{align*}
Then we can show that the process $\bm{X}_{\bm{s},A_{n}} = \int_{\mathbb{R}^{d}}g\left({\bm{s} \over A_{n}}, \|\bm{s} - \bm{v}\|\right)\bm{L}(d\bm{v})$ is an approximately $\bar{D}(\log n)$-dependent random field for sufficiently large $\bar{D}>0$ by applying a similar argument in the proof of Lemma \ref{proof-CARMA-example}.


\end{document}